\documentclass[10pt,a4paper]{amsart}
\usepackage{amsmath, amssymb,amsthm}
\usepackage{color} 
\usepackage{hyperref}
\usepackage[all]{xy}
\usepackage{bbm}
\usepackage{enumerate,enumitem}
\usepackage{comment}
\usepackage{mathrsfs}
\usepackage{MnSymbol} 
\usepackage{tikz} 
\usepackage{tabularx}

\theoremstyle{plain}
\newtheorem{thm}{Theorem}[section]
\newtheorem{prop}[thm]{Proposition}
\newtheorem{lemma}[thm]{Lemma}
\newtheorem{cor}[thm]{Corollary}

\theoremstyle{definition}
\newtheorem{defn}[thm]{Definition}
\theoremstyle{remark}
\newtheorem{rem}[thm]{Remark}

\newcommand{\R}{\mathbb{R}}
\newcommand{\C}{\mathbb{C}}
\newcommand{\Q}{\mathbb{Q}}
\newcommand{\N}{\mathbb{N}}
\newcommand{\Z}{\mathbb{Z}}

\newcommand{\cC}{{\mathcal{C}}}
\newcommand{\cE}{{\mathcal{E}}}
\newcommand{\cD}{{\mathcal{D}}}

\newcommand{\cU}{\mathcal{U}}
\newcommand{\cV}{\mathcal{V}}

\newcommand{\cO}{\mathcal{O}}

\newcommand{\cR}{\mathcal{R}}

\newcommand{\cT}{\mathcal{T}}
\newcommand{\cW}{\mathcal{W}}
\newcommand{\gm}{\mathfrak{m}}

\newcommand{\gp}{\mathcal{P}}

\newcommand{\univ}{\mathrm{univ}}

\newcommand{\dR}{\mathrm{dR}}
\newcommand{\ord}{\mathrm{ord}}
\newcommand{\rig}{\mathrm{rig}}
\newcommand{\nor}{\mathrm{nor}}

\DeclareMathOperator{\ad}{ad} 
 
\DeclareMathOperator{\End}{End}
\DeclareMathOperator{\Frob}{Frob}

\DeclareMathOperator{\GL}{GL}
\DeclareMathOperator{\rH}{H}

\DeclareMathOperator{\Hom}{Hom}

\DeclareMathOperator{\Spec}{Spec}
\DeclareMathOperator{\Spm}{Spm}
\DeclareMathOperator{\Tr}{Tr} 
\DeclareMathOperator{\Ps}{Ps} 

\newcommand{\cris}{\mathrm{crys}}
\newcommand{\tot}{\mathrm{tot}}
\newcommand{\crit}{\mathrm{crit_p}}

\newcommand{\unr}{\mathrm{unr}}

\makeatletter
\DeclareFontFamily{OMX}{MnSymbolE}{}
\DeclareSymbolFont{MnLargeSymbols}{OMX}{MnSymbolE}{m}{n}
\SetSymbolFont{MnLargeSymbols}{bold}{OMX}{MnSymbolE}{b}{n}
\DeclareFontShape{OMX}{MnSymbolE}{m}{n}{
    <-6>  MnSymbolE5
   <6-7>  MnSymbolE6
   <7-8>  MnSymbolE7
   <8-9>  MnSymbolE8
   <9-10> MnSymbolE9
  <10-12> MnSymbolE10
  <12->   MnSymbolE12
}{}
\DeclareFontShape{OMX}{MnSymbolE}{b}{n}{
    <-6>  MnSymbolE-Bold5
   <6-7>  MnSymbolE-Bold6
   <7-8>  MnSymbolE-Bold7
   <8-9>  MnSymbolE-Bold8
   <9-10> MnSymbolE-Bold9
  <10-12> MnSymbolE-Bold10
  <12->   MnSymbolE-Bold12
}{}

\let\llangle\@undefined
\let\rrangle\@undefined
\let\lsem\@undefined
\let\rsem\@undefined
\DeclareMathDelimiter{\llangle}{\mathopen}%
                     {MnLargeSymbols}{'164}{MnLargeSymbols}{'164}
\DeclareMathDelimiter{\rrangle}{\mathclose}%
                     {MnLargeSymbols}{'171}{MnLargeSymbols}{'171}
\DeclareMathDelimiter{\lsem}{\mathopen}%
                     {MnLargeSymbols}{'102}{MnLargeSymbols}{'102}
\DeclareMathDelimiter{\rsem}{\mathclose}%
                     {MnLargeSymbols}{'107}{MnLargeSymbols}{'107}                     
                     
\makeatother

\title{On Siegel eigenvarieties at Saito-Kurokawa points}

\author{ Tobias Berger and Adel Betina}
\address{School of Mathematics and Statistics, University of Sheffield, Hicks Building, Hounsfield Road, Sheffield S3 7RH, UK.}
\email{tberger@cantab.net, adelbetina@gmail.com}

\thanks{Both authors acknowledge  support from the EPSRC Grant EP/R006563/1.}


\begin{document}

\begin{abstract}
 We study the geometry of the $p$-adic Siegel eigenvariety $\cE$ of paramodular tame level at certain Saito-Kurokawa points having a critical slope. For $k \geq 2$ let $f$ be a cuspidal new eigenform of $\mathrm{S}_{2k-2}(\Gamma_0(N))$ ordinary at a prime $p\nmid N$ with sign $\epsilon_f=-1$ and write $\alpha$ for the $p$-adic unit root of the Hecke polynomial of $f$ at $p$. Let $\pi_\alpha$   be the semi-ordinary $p$-stabilization of the Saito-Kurokawa lift of the cusp form $f$ to $\mathrm{GSp}(4)$ of weight $(k,k)$ and paramodular tame level. Under the assumption that the dimension of the Selmer group $\rH^1_{f,\unr}(\Q,\rho_f(k-1))$ attached to $f$ is at most one and some mild assumptions on the automorphic representation attached to $f$,  we show that $\cE$ is smooth at the point corresponding to $\pi_\alpha$, and that the irreducible component of $\cE$ specializing to $\pi_\alpha$ is not globally endoscopic. Finally we give an application to the Bloch-Kato conjecture, by proving under some mild assumptions that the smoothness failure of $\cE_{\Delta}$ at $\pi_\alpha$ yields that $\dim \rH^1_{f,\unr}(\Q,\rho_f(k-1)) \geq 2$.
\end{abstract}

\maketitle

\section{Introduction}
Let $p$ be a prime number. Eigenvarieties are $p$-adic rigid analytic spaces interpolating the Hecke eigenvalues of automorphic representations of a particular reductive group $\mathrm{G}$ of finite slope eigenvalues for Hecke operators at $p$, fixed tame level away from $p$ and varying weights. Following the seminal works of Hida \cite{hida85} and Coleman-Mazur \cite{coleman-mazur}  their geometry has been studied by many people, e.g. Bella{\"{\i}}che-Chenevier \cite{B-C},  Majumdar \cite{Maj}, Bella{\"{\i}}che-Dimitrov \cite{D-B}, Betina-Dimitrov-Pozzi and Betina-Dimitrov \cite{BDPozzi, BD19} for $\mathrm{G}={\rm GL}_2(\Q)$, 
and by Bella{\"{\i}}che-Chenevier \cite{Bellaiche08}, \cite{bb} for unitary groups.

Andreatta, Iovita and Pilloni constructed in \cite{AIP} an eigenvariety parametrizing locally analytic overconvergent cuspidal Siegel eigenforms of genus two, principal level $N$ and finite slope, and they proved that the Siegel eigenvariety of tame level $1$ is \'etale over the weight space at certain classical non-critical  points of regular cohomological weights with Iwahoric level at $p$. The proof uses the classicality criteria for overconvergent Siegel cusp forms of Hida  \cite[Prop.3.6]{Hida1}, Tilouine and Urban \cite[Theorem   3.2]{U-T}, Pilloni \cite[Theorem   2]{Pilloni} and the multiplicity one theorem of Arthur's classification for $\mathrm{GSp}_4$ \cite{A}.

We investigate in this work the geometry of the Siegel eigenvariety $\cE_\Delta$ of paramodular level $N$ at the points corresponding to Saito-Kurokawa lifts of ordinary cusp forms for $\GL_2(\Q)$ (which have a critical slope), including the case of the non-cohomological weight $(2,2)$.

In order to state our results, we recall some facts and fix some notations: 
Let $N \geq 1$ be an integer prime to $p$. For a prime $\ell$  the paramodular subgroup of $\mathrm{GSp}_4(\Q_\ell)$ is defined as $\Delta_\ell=\gamma \mathrm{M}_4(\Z_\ell)\gamma^{-1} \cap \mathrm{GSp}_4(\Q_\ell)$ for $\gamma= \mathrm{diag}[1,1,\ell,1]$. We write $\Delta:=\prod_{\ell \mid N} \Delta_{\ell} \cap {\rm GSp}_4(\Q)$ for the paramodular congruence subgroup of level $N$. If $N=1$ we put $\Delta={\rm GSp}_4(\Z)$.

Let $f \in \mathrm{S}_{2k-2}(\Gamma_0(N),K_f)$ be a weight $2k-2$ cuspidal $N$-new eigenform for $\GL_2(\Q)$ with coefficients  in a number field $K_f$ (and trivial central character). Assume that $f$ has an ordinary $p$-stablization and denote it by $f_\alpha$, where $U_p (f_\alpha)=\alpha.f_\alpha$. 

The L-function $L(f,s)$ attached to $f$ satisfies the following functional equation:
$$L(f,s)= \epsilon_f L(f,2k-2-s).$$  
We have that $\epsilon_f=(-1)^{\mathrm{\ord}_{s=k-1}L(f,s)}$. Assume until the end of this paper that $\epsilon_f=-1$\footnote{When $N=1$, one has $\epsilon_f= (-1)^{k-1}$.}, which means that there exists a lift $\mathrm{SK}(f)$ to a holomorphic weight $(k,k)$ cuspform of level $\Delta$ called the Saito-Kurokawa lift of $f$. It satisfies $$L^N(\mathrm{SK}(f), {\rm spin}, s)=\zeta^N(s-k+1) \zeta^N(s-k+2) L^N(s,f).$$  When $N=1$ this lift was constructed by Maass, Andrianov and Zagier; Gritsenko generalized it to any level $N$.  A representation theoretic approach building on results of Piatetski-Shapiro and  Waldspurger is discussed in \cite{Schmidt07} (see also \cite[Theorem   5.3]{Schmidt18}).

In order to  $p$-adically  deform  $\mathrm{SK}(f)$, one must first choose a semi-ordinary\footnote{Semi-ordinary means that the eigenvalue for the Hecke operator $U_0$ is a $p$-adic unit. Following Tilouine-Urban this is also called Siegel ordinary.} {\it  $p$-stabilization }
of $\mathrm{SK}(f)$, that is an eigenform  of tame level the paramodular group $\Delta$  and sharing the same eigenvalues as $\mathrm{SK}(f)$ away from $p$ and of finite slope. Denote by $\pi_\alpha$ the $p$-stablization of $\mathrm{SK}(f)$ such that $U_{0}(\pi_\alpha)=\alpha.\pi_\alpha \text{, and } U_{1}(\pi_\alpha)=p.\alpha.\pi_\alpha$ 
where $U_0,U_1$ are the Hecke operators attached to $\mathrm{diag}[1,1,p,p]$ ($U_0$ is often denoted by $U_p$), $\mathrm{diag}[1,p,p^2,p]$, and $U_1$ has been renormalized to have a good $p$-adic interpolation (see for example \cite[Theorem   2.4.14]{urban}).

Let $\cE_\Delta$ be Siegel eigenvariety of tame paramodular level $\Delta$ (see appendix \S \ref{eigenvariety2}). It is reduced and equidimensional of dimension $2$, and endowed with a morphism $$\kappa: \cE_\Delta\rightarrow \cW$$ called the weight map (which is locally finite and torsion-free), where  the weight space $\cW$  is  the rigid analytic space over $\Q_p$ such that  
$\cW(\C_p)=\Hom_{\mathrm{cont}}((\Z_p^{\times})^2, \C_p^\times).$

The cuspidal eigenform $\pi_\alpha$ defines a point of $\cE_\Delta$ which we denote again (by a slight abuse of notation) by $\pi_\alpha$. Write $L$ for the residue field of  the point $\pi_\alpha \in \cE_{\Delta}$, a finite extension of $\Q_p$. Note that the slopes of $U_{0}$ and $U_{1}$ are locally constant on $\cE_\Delta$, and equal to $0$ for $U_{0}$ and $1$ for $U_{1}$ locally at $\pi_\alpha$. This means that the cuspform $\pi_\alpha$ has a critical slope since it does not satisfy the small slope condition of \cite[Theorem    7.3.1]{AIP}.

One can show that there exists a pseudo-character $\mathrm{Ps}=\mathrm{Ps}_{\cE_\Delta}:G_\Q \rightarrow  \mathcal{O}(\cE_\Delta)$ of dimension $4$ such that the specialization $\mathrm{Ps}(y)$ of $\mathrm{Ps}$ at a classical point $y \in \cE_\Delta(\bar{\Q}_p)$ is the trace of the semi-simple $p$-adic Galois representation $\rho_y:G_\Q \to \GL_4(\bar{\Q}_p)$ of dimension $4$ attached to a cuspidal Siegel eigenform $g_y$ corresponding to $y$ (i.e. $L(g_y,{\rm spin}, s)=L(\rho_y,s)$).
For $y=\pi_{\alpha}$ we have $$\mathrm{Ps}(\pi_\alpha)=\epsilon_p^{1-k} + \epsilon_p^{2-k}  + \Tr\rho_f,$$
where $\rho_f :G_{\Q} \to \GL_2(L)$ is the $p$-adic Galois representation attached to $f$ (i.e. $L(f,s)=L(\rho_f,s)$) and $\epsilon_p$ is the $p$-adic cyclotomic character ($L$ contains the Hecke field $K_f$ of $f$). 

Let $\cT$ be the local ring of $\cE_\Delta$ at $\pi_\alpha$ for the rigid topology, $\gm$  the maximal ideal of $\cT$ and $\varLambda$ the local ring of $\cW$ for the rigid topology at the weight $\kappa(\pi_\alpha)$ of $\pi_\alpha$ (they are both Henselian rings). Note that $\cT$ is an equidimensional ring of dimension $2$.

\begin{defn}
 
We say that  an {\it irreducible} affinoid $\mathcal{Z} \subset \cE_\Delta$ of dimension $2$ is {\it stable} if and only if the reducibility locus of the pseudo-character $\mathrm{Ps}_{\mathcal{Z}}:G_\Q  \to \cO(\mathcal{Z})$ given by the composition of $\Ps_{\cE_{\Delta}}$ with the natural morphism $\cO(\cE_{\Delta}) \to \cO(\mathcal{Z})$ is strictly contained in $\mathcal{Z}$ (i.e. of dimension less or equal to $1$). 
\end{defn} 

Let $\pi_f=\underset{\ell }{\bigotimes} \pi_{f,\ell}$ be the automorphic representation attached to $f$. We list below the main assumptions that we need in this work:

\begin{itemize}

\item[$\bullet$] $({ \bf BK})$  The Bloch-Kato Selmer group {\small  \[\rH^{1}_{f,\unr}(\Q,\ad^0 \rho_f)=\ker(\rH^1(\Q,\ad^0 \rho_f) \to \rH^1(\Q_{p},\ad^0 \rho_f \otimes B_{\cris}) \oplus_{\ell \nmid p} \rH^1(I_{\ell},\ad^0 \rho_f))\] } is trivial  ($B_{\cris}$ is Fontaine's crystalline period ring).

\item[$\bullet$] $({ \bf Reg})$ If $k=2$ assume  $\alpha \neq 1$ ($\alpha$ is the $p$-adic unit root of the Hecke polynomial at $p$ of $f$).

\item[$\bullet$] $({ \bf St})$  For any prime $\ell \mid N$,  $ \rho_f|_{I_{\ell}}$ is special and the sign $\epsilon_\ell(f)$ in the functional equation of $f$ at $\ell$ is always equal to $1$ and the local sign $\epsilon_\infty(f)$ at the archimedean place $\infty$ is equal to $-1$ (which means that $k$ is even). In that case, for any prime $\ell \mid N$, $\pi_{f,\ell} \simeq \mathrm{St} \otimes \xi$, where $\xi$ is the unramified character with $\xi(\ell)=-1$.

\end{itemize}

We highlight that the sign $\epsilon_f=\underset{\ell \mid N\cup\{\infty\}}{\prod}\epsilon_\ell(f) $ and $\epsilon_\infty(f)=(-1)^{k-1}$. Moreover, if $f$ satisfies $({ \bf St})$, then the integer $N$ is necessarily squarefree.

Andreatta-Iovita-Pilloni pose the following question in \cite[\S.8]{AIP}:

\medskip
{\bf  Open problem.} Let $x(g)$ be a classical point of the Siegel eigenvariety $\cE_N$ of tame level $\Gamma(N)$, the principal congruence subgroup  of level $N$. Is the map $\kappa: \cE_N \to \cW$ unramified at $x(g)$?

\medskip

Let $\gm_{\varLambda}$ be the maximal ideal of $\varLambda$, the completed local ring of $\cW$ at $\kappa(\pi_\alpha)$, $\cT'=\cT/\gm_{\varLambda}\cT$ be the local ring of the fiber $\kappa^{-1}(\kappa(\pi_\alpha)) \subset \cE_{\Delta}$ at $\pi_\alpha$ (since $\kappa$ is locally finite, $\cT'$ is an Artinian algebra), and let $\mathfrak{t}_{\pi_\alpha}$ (resp. $\mathfrak{t}_{\pi_\alpha}^0$) be the Zariski tangent space of  $\cT$ (resp. $\cT'$, i.e the relative tangent space of $\kappa^{\#}:\varLambda \to \cT$).

Let $\omega_p:G_\Q \to \Z_p^{\times}$ be the Teichm\"{u}ller character and $L_p(f_\alpha,\omega^{-1}_p,\cdot) \in \Lambda:=\bar{\Z}_p\lsem T \rsem$ be the Manin-Vishik $p$-adic L-function attached to $f_\alpha \otimes \omega_p^{-1}$ (see e.g. \cite[
   16.2]{Kato}), and let  {\small \[\rH^{1}_{f,\unr}(\Q,\rho_f(k-1))=\ker(\rH^1(\Q,\rho_f(k-1)) \to \rH^1(\Q_{p},\rho_f(k-1) \otimes B_{\cris}) \oplus_{\ell \nmid p} \rH^1(I_{\ell},\rho_f(k-1)))\] } be the Bloch-Kato Selmer group attached to $f$.

Recall that $\cE_{\Delta}$ is equidimensional of dimension 2. Our main result is the following theorem describing the local geometry at $\pi_\alpha$ of the rigid analytic space $\cE_{\Delta}$:
\begin{thm}[see \S.\ref{noendoscopic} and \S.\ref{mainthmsection}] \label{A}\ 

{Put  $s=\dim \rH^{1}_{f,\unr}(\Q,\rho_f(k-1))$.}
\begin{enumerate}
\item Assume that $k \geq 2$ and $({ \bf Reg})$.  Then all the irreducible affinoids of $\cE_\Delta$ of dimension $2$ specializing to $\pi_\alpha$ are stable.

\item Assume that $k \geq 2$, $({ \bf Reg})$, $({ \bf BK})$ and $({ \bf St})$, and assume also when $k=2$ that $L_p(f_\alpha,\omega^{-1}_p,T=p) \ne 0$, then $$2 \leq \dim \mathfrak{t}_{\pi_\alpha} \leq 1 +s^2 \text{ and } \dim \mathfrak{t}_{\pi_\alpha}^0 \leq s^2.$$  Moreover, if $\dim \rH^{1}_{f,\unr}(\Q,\rho_f(k-1))=1$, then $\cE_{\Delta}$ is smooth at $\pi_{\alpha}$, and the reducibility locus of the pseudo-character $\mathrm{Ps}_{\cT}:G_\Q \rightarrow \cO(\cE_{\Delta}) \to \cT$ is the closed irreducible smooth subscheme of $\Spec \cT$ of dimension $1$ associated to the Saito-Kurokawa lift of the ordinary Hida family $\mathcal{F}$ passing through $f_\alpha$, and it is even a principal Weil divisor of $\Spec \cT$.

\end{enumerate}
\end{thm}

\medskip A key step in the proof is the determination  of the schematic reducibility locus of the pseudo-character $\mathrm{Ps}_{\cT} : G_\Q \to \cO(\cE_{\Delta}) \to \cT$ carried by $\cE_{\Delta}$ at $\pi_\alpha$, and our approach uses pseudo-representations of $p$-adic families of cuspidal Siegel eigenforms and $p$-adic Hodge theory.
We provide a more detailed sketch of the proof in section \ref{sketch}.

 A direct consequence of (ii)  of the above theorem is that under these assumptions there exists a unique irreducible component of $\cE_{\Delta}$ specializing to $\pi_\alpha$ when the Selmer group $\rH^{1}_{f,\unr}(\Q,\rho_f(k-1))$ is $1$-dimensional.

The smoothness of the eigencurve at critical points is a crucial ingredient for the construction of a family of $p$-adic  $L$-functions on an open neighborhood of these points, see e.g. \cite{Be}. {Our result on the smoothness of $\cE_\Delta$ opens up the possibility of constructing a family of $p$-adic L-functions in a neighbourhood of $\pi_\alpha$, a challenging question in Iwasawa theory.}

\begin{rem}
We require assumption $({ \bf BK})$ for applying an $R=T$ theorem in characteristic zero (see Proposition \ref{RTBK}). Assumption $({ \bf BK})$  is conjectured to always hold by the Bloch-Kato conjectures as $L({\rm ad}^0(\rho_f), 1) \neq 0$. It has been proven under various assumptions, in particular, by Weston \cite[Theorem 5.5]{Weston} if $f$ is not CM, and special or supercuspidal at all $\ell \mid N$, and by Kisin \cite{Kisin04} under mild assumptions on the residual representation $\overline{\rho}_f$, e.g. if $\overline{\rho}_f|_{G_L}$ is absolutely irreducible for any quadratic extension $L/\Q$ with $L \subset \Q(\zeta_{p^3})$ or if $\overline{\rho}_f^{\rm ss}$ over
an algebraic closure is given by the sum of two characters that are distinct on $G_{\Q(\zeta_{p^{\infty}})}$. 
\end{rem}
\medskip

Using results about $\Lambda$-adic Selmer groups we exhibit many examples where the Selmer group $\rH^{1}_{f,\unr}(\Q,\rho_f(k-1))$ is $1$-dimensional (see Appendix \S.\ref{example}). {We also have an example of an elliptic curve satisfying all the assumptions of (ii) of the above theorem (see \S.\ref{example}).}

\begin{cor}\

Assume that $k \geq 2$, $({ \bf Reg})$, $({ \bf BK})$ and $({ \bf St})$, and assume also when $k=2$ that $L_p(f_\alpha,\omega^{-1}_p,T=p) \ne 0$. If the rigid analytic  space $\cE_{\Delta}$ is singular at $\pi_{\alpha}$ then $\dim \rH^{1}_{f,\unr}(\Q,\rho_f(k-1)) \geq 2$.

\end{cor} 

Hence we have a geometric criterion to detect if $\dim \rH^{1}_{f,\unr}(\Q,\rho_f(k-1)) \geq 2$. Thus, the question of finding a lower bound of the dimension of the Selmer group $\rH^{1}_{f,\unr}(\Q,\rho_f(k-1))$ can be reduced to certain computations of spaces of semi-ordinary $p$-adic modular cuspforms for $\mathrm{GSp}_4$. 

\medskip

It turns out that the geometry of $\cE_N$ at $\pi_\alpha$ depends on the tame level. When we change the tame level to the principal Siegel congruence subgroup $\Gamma(N)$ it is in general non-smooth. In particular, the answer to the question in \cite{AIP} is negative if $N$ is not prime.

\begin{cor}[see Corollary \ref{cor12.5}]
Assume that $\ell_1.\ell_2 \mid N$, where $\{\ell_i\}_{\{1,2\}}$ are prime numbers and assume that $f$ is Steinberg at both these primes. Then the eigenvariety $\cE_N$ of level $\Gamma(N)$ is singular  at $\pi_\alpha$ and has at least two irreducible endoscopic components specializing to $\pi_\alpha$. 

\end{cor}

\subsection{Sketch of the proof of Theorem \ref{A}} \label{sketch}

Using \cite[Theorem   3.2.9]{urban} we show that any irreducible affinoid $\mathcal{Z} \subset \cE_\Delta$ of dimension $2$ specializing to $\pi_\alpha$ such that its pseudo-character $\mathrm{Ps}_{\mathcal{Z}}$ is reducible contains necessarily a classical point corresponding to an endoscopic cuspidal eigenform of trivial central character, Iwahoric level and old at $p$. In fact, it follows from Arthur's classification and results of Roberts and Schmidt  that no such Yoshida lift exists for level $\Delta$.  This establishes that all irreducible components of $\cE_{\Delta}$ containing $\pi_\alpha$ are stable.

Hence by localizing the pseudo-character $\mathrm{Ps}_{\cE_{\Delta}}:G_\Q \rightarrow  \mathcal{O}(\cE_{\Delta})$ of dimension $4$ at the local Henselian ring $\cT$, we get a  pseudo-character $\mathrm{Ps}_{\cT}:G_{\Q} \rightarrow  \cT$ deforming $\mathrm{Ps}(\pi_\alpha)$ which is generically irreducible on each irreducible component containing $\pi_\alpha$. Following the results of \cite{bb}, we obtain a GMA matrix $S=\cT[G_\Q]/\ker(\mathrm{Ps}_{\cT}) $ with orthogonal idempotents lifting the natural idempotents of the semi-simple representation $\varrho=\epsilon_p^{2-k} \oplus \rho_f \oplus \epsilon_p^{1-k}$. 

The total reducibility ideal $\mathcal{I}^{\tot}$ of $\mathrm{Ps}_{\cT}$ is defined to be the smallest ideal $I$ of $\cT$ such that $$\mathrm{Ps}_{\cT} \mod{I}=T_1+T_2+T_3$$ for pseudocharacters $T_i$ with $T_i \mod{\gm}=\Tr(\rho_i)$ for $\rho_1=\epsilon^{2-k}_p$, $\rho_2= \rho_f$, $\rho_3= \epsilon^{1-k}_p$. By results of \cite{bb} it is controlled by the entries of the GMA $S$ (see Proposition \ref{redideal2}). These in turn give rise to $S$-extensions of $\rho_i$ by $\rho_j$ for $i \neq j$. We prove in Theorem \ref{princ1}{ when  $s:=\dim \rH^{1}_{f,\unr}(\Q,\rho_f(k-1))=1$} that $\mathcal{I}^{\tot}$ is principal ({or more generally we bound the number of its generators by $s^2$}) by proving that these extension satisfy the required local properties to lie in the corresponding Selmer groups
$\rH^1_{f,\unr}(\Q,\rho_f(k-2))=\{0\}$ (a deep result of Kato \cite{Kato}), $\rH^1_{f,\unr}(\Q,\epsilon_p)\overset{\text{Kummer}}{\simeq} \Z^{\times} \otimes L= \{0\}$ and $\rH^1_{f,\unr}(\Q,\rho_f(k-1))$.

This local analysis forms the technical heart of the paper. At $p$ we use that any representation $\rho_z$ attached to a classical point $z$ of a sufficiently small open admissible affinoid $\mathcal{Z} \subset \cE_{\Delta}$ containing $\pi_\alpha$ is semi-ordinary (i.e. $\dim \rho_z^{I_p} \geq 1$). Using this we prove in \S\ref{ordrho} and \S \ref{ordrho1} that any $S$-extension $W$ (resp. $W'$) occuring in the cohomology group $\rH^1(\Q,\rho_f(k-1))$ (resp. $\rH^1(\Q,\rho_f(k-2))$) is in fact ordinary at $p$, in the sense that $W^{I_p} \ne 0$, $(W')^{I_p} \ne 0$ and $\Frob_p$ acts on them by $\alpha$. Therefore, $W$ (resp. $W'$ when $k \geq 3$) is ordinary in the sense of  Fontaine-Perrin-Riou (so de Rham, see Theorem \ref{nekthm}), and hence crystalline since $\rH^1_g(\Q_p,\rho_f(k-i))=\rH^1_f(\Q_p,\rho_f(k-i))$  for $i \in \{1,2\}$. 

To prove the crystallinity of the $S$-extensions in $\mathrm{Ext}^1_{G_{\Q}}(\epsilon_p^{1-k},\epsilon_p^{2-k})$ we apply  in \S\ref{cnsht} the results of  \cite{bb} \S4  on the analytic continuation of crystalline periods for the smallest Hodge-Tate weight in families of $p$-adic $G_{\Q_p}$-representations occurring in a torsion free coherent module of generic $3$. To this end we establish in section \S\ref{s4} that classical points which are old at $p$ are very Zariski dense in $\cE_{\Delta}$. To be able to study the period we are interested in we need to consider the quotient by the line fixed by inertia due to semi-ordinarity. At  a classical point $z \in \mathcal{Z}$ of cohomological weight $(l_1,l_2)$ the smallest Hodge-Tate weight of the $3$-dimensional $G_{\Q_p}$-representation $\rho_z/\rho_z^{I_p}$ is $l_2-2$ and $\dim \cD_{\cris}(\rho_z/\rho_z^{I_p})^{U_{1}/U_{0}(z) p^{l_2-2}} =1$. 

This allows us to prove that the $S$-extensions occuring in $\mathrm{Ext}^1_{G_{\Q}}(\epsilon_p^{1-k},\epsilon_p^{2-k})$ have a crystalline period equal to $$\lim_{z_n \in \cE_\Delta , z_n \to \pi_\alpha}U_{1}/U_{0}(z) p^{l_2(z)-2}=U_{1}/U_{0}(\pi_\alpha) p^{k-2}=p^{k-1}.$$ This means that for any $S$-extension $V \in \mathrm{Ext}^1_{L[G_\Q^{Np}]}(\epsilon_p^{1-k},\epsilon_p^{2-k})$, we have $\cD_{\cris}^{\Phi=p^{k-1}}(V) \ne 0$  so that $\dim \cD_{\cris}(V)=2$, i.e. that $V$ is crystalline at $p$. 

For $\ell \mid N$, \cite[Lemma.4.1.3]{urban} shows that the $S$-extensions occuring in the cohomology group $\rH^1(\Q,\rho_f(k-1))$ (resp. $\rH^1(\Q,\rho_f(k-2))$)  are unramified outside $p$. For proving that the $S$-extensions occuring in $\rH^1(\Q,\epsilon_p)$  are unramified at $\ell \mid N$ (under the hypothesis $({ \bf St})$), we use the semi-continuity of the rank of the monodromy operator attached to the Weil-Deligne representation at $\ell$ of $p$-adic families and that the rank is generically one for families of paramodular tame level.

Having bounded the number of generators of $\mathcal{I}^{\tot}$ by $s^2$  we determine in \S \ref{modpar} the local ring $A:=\cT/\mathcal{I}^{\tot}$ by proving under the hypothesis $({ \bf BK})$ that the completion $\widehat{A}$ of $A$ with respect to its maximal ideal is isomorphic to the universal ring representing the $p$-ordinary minimally ramified deformations of $\rho_f$, which is a discrete valuation ring. Since $\cT$ is equidimensional of dimension $2$, $\cT/\mathcal{I}^{\tot}=A$ is regular of dimension one (implied by $\widehat{A}$ being regular) and $\mathcal{I}^{\tot}$ is principal when $\dim \rH^{1}_{f,\unr}(\Q,\rho_f(k-1))=1 $ ({or more generally generated by at most $s^2$ elements}), it follows that the generator of $\mathcal{I}^{\tot}$ is a regular local parameter of $\cT$ when $\dim \rH^{1}_{f,\unr}(\Q,\rho_f(k-1))=1 $ ({or more generally, we obtain the desired bound of the Zariski tangent space of $\cT$}). 

This means that the tangent space of $\cT$ is of dimension $2$ when \[\dim \rH^{1}_{f,\unr}(\Q,\rho_f(k-1))=1 \] and $\cT$ is regular of dimension $2$. Thus the rigid analytic space $\cE_\Delta$ is smooth at $\pi_\alpha$, and as a consequence, $\cE_\Delta$ has a unique irreducible component (of dimension $2$) specializing to $\pi_\alpha$. 

However, for the case when $k=2$ (i.e Theorem   \ref{A}(iii)), we need to prove in addition that the $S$-extensions occuring in $\rH^1(\Q,\rho_f)$ are crystalline at $p$. This seems difficult to establish (see Remark \ref{Jannsen}). But we know that these extensions are ordinary in the sense that they have an unramified line on which $\Frob_p$ acts by $\alpha$, and so they belong to a Greenberg's type Selmer group $\mathrm{Sel}_{\Q,f_\alpha}$ attached to $\rho_f^{\vee}(-1)$ (see \S.\ref{Selmergroupvanish}). Moreover, we know from the Iwasawa main conjecture for $\GL_2$ that the Pontryagin dual of the $\Lambda$-adic Greenberg's Selmer group of $f_\alpha$ is a torsion $\Lambda$-module, and its characteristic ideal contains the $p$-adic L function $L_p(f_\alpha,\omega^{-1}_p,\cdot)$ (see \cite[Theorem   3.25]{SkinnerUrban14}). Hence, the condition that $L_p(f_\alpha,\omega^{-1}_p,T=p) \ne 0$ is sufficient for the vanishing of $\mathrm{Sel}_{\Q,f_\alpha}$.

\medskip

By a general principle of Hida the adjoint $L$-function of modular forms controls congruences with other modular forms. This demonstrates that our assumptions are natural for understanding the geometry at $\pi_{\alpha}$ since $\rH^1_{f,\unr}(\Q, \ad^0(\varrho))$ contains the Selmer groups \[\rH^1_{f,\unr}(\Q, \rho_f(k-2)), \rH^1_{f,\unr}(\Q, \rho_f(k-1))\text{, and } \rH^1_{f,\unr}(\Q, \ad^0(\rho_{f})),\] where $\varrho=\epsilon_p^{1-k} \oplus \epsilon_p^{2-k}  \oplus \rho_f$ is the semi-simple $p$-adic representation attached to $\pi_\alpha$. As we show the first two groups control the generators of $\mathcal{I}^{\rm tot}$ (which correspond to congruences with stable components) and $\rH^1_f(\Q, \ad^0(\rho_{f}))$ controls $\cT/\mathcal{I}^{\rm tot}$ (corresponding to the number of ordinary families of $p$-adic representations specializing to $\rho_{f_{\alpha}}$).

\subsection{Relationship to other results in the literature} Bella{\"{\i}}che-Chenevier studied in \cite{bb} the geometry of  some eigenvarieties $X$ attached to unitary Shimura varieties at points with reducible Galois representation and gave applications to the Bloch-Kato conjecture. They focus on points $z \in X$ with Galois representation given by $\mathbbm{1} \oplus \epsilon_p \oplus \rho_z$, where $\rho_z$ is an irreducible $n$-dimensional representation anti-ordinary at $p$. They proved that at $z \in X$, the local Galois deformation at $p$ is irreducible on every Artinian thickening of $z$ (the reducibility locus at $z$ of the pseudo-character carried by $X$ is the maximal ideal of $\cO_{X,z}$). It should be pointed out that our setting is quite different since the reducibility locus at $\pi_\alpha$ of the pseudo-character $\mathrm{Ps}_{\cE_{\Delta}}$ is given by a principal Weil divisor of the $2$-dimensional affine scheme $\Spec \cT$ and corresponds on the modular side to the Saito-Kurokawa lift of the Hida family passing through $f_\alpha$. A further difference between these settings lies in the position of the Hodge-Tate weights and their distribution between the different pieces of the reducible Galois representations $\mathbbm{1} \oplus \epsilon_p \oplus \rho_z$ and  $\rho_{\pi_\alpha}:=\epsilon_p^{1-k} \oplus \epsilon_p^{2-k}  \oplus \rho_f$. More precisely, while the smallest Hodge-Tate of $\rho_{\pi_\alpha}$ is zero and occurs in the $2$-dimensional representation $\rho_f$, the smallest Hodge-Tate weight of $\mathbbm{1} \oplus \epsilon_p \oplus \rho_z$ is $-1$ and occurs in the one dimensional sub-representation $\epsilon_p$, and  $\rho_z$ has no Hodge-Tate weights equal to $\{0,-1\}$, and this difference makes  the proof of the crystallinity of the $S:=\cT[G_\Q]/\ker(\mathrm{Ps}_{\cT}) $-extensions occuring in $\mathrm{Ext}^1_{G_{\Q}}(\epsilon^{1-k},\epsilon_p^{2-k})$ (in our setting)  more subtle than \cite[Prop.8.2.14]{bb} (see \S.\ref{cnsht}). In addition, we investigate also in this paper the geometry of $\cE_{\Delta}$ at Saito-Kurokawa points $\pi_\alpha$ of non-cohomological weights (i.e when $k=2$) and in that case $\rho_{\pi_\alpha}$ has only two Hodge-Tate weights $\{0,1\}$ (with multiplicity two).

\medskip

Skinner-Urban  constructed in \cite[Thm.2.4.10]{urban} a semi-ordinary eigenvariety {as an} admissible open of $\cE_N$. Using a deep automorphic argument they established the existence of a stable semi-ordinary $p$-adic cuspidal component $\mathcal{Y}$ of dimension $2$ specializing to $\pi_\alpha$ (see \cite[Thm.4.2.7]{urban}), with fewer assumptions on the level and the local representation $\rho_f$ at $\ell \mid Np$ than us (they assumed only that $f$ is ordinary at $p$).  They then applied the lattice construction of \cite{urban1} (generalizing Ribet's Lemma to higher dimensions) to obtain non-trivial extensions in $ \rH^1_{f,\unr}(\Q, \rho_f(k-1))$.

\medskip

In \cite{bk} short crystalline, minimal, essentially self-dual deformations of non-semisimple mod $p$ Galois representations $\overline{\rho}$ with $\overline{\rho}^{\rm ss}=\overline{\rho}_{\pi_\alpha}=\overline{\epsilon}_p^{2-k} \oplus \overline{\rho}_f \oplus \overline{\epsilon}_p^{1-k}$ are studied. In this analysis the principality of the total reducibility ideal of the universal pseudodeformation of ${\rm Tr}(\overline{\rho})$ to $\mathcal{O}_L$-algebras also played a crucial role. 

\medskip

Hernandez constructed in \cite{He} a three dimensional $p$-adic eigenvariety for the group $\mathrm{U}(2,1)(E)$, where $E$ is a quadratic imaginary field in which $p$ is inert (the Picard modular surface has an empty ordinary locus in that case), and gave an application by reproving particular cases of the Bloch-Kato conjecture for Galois characters of $E$.

\subsection*{Acknowledgement} 
We would first like to thank the referee  for their careful reading, helpful comments and suggestions that helped to improve the exposition and weaken the hypotheses in our result. We would like also to thank Riccardo Brasca, Valentin Hernandez, Kris Klosin, Vincent Pilloni, Jacques Tilouine, Chris Skinner and Eric Urban for helpful communications related to the topics of this article.

\tableofcontents

\subsection*{Notation and some remarks}

\begin{enumerate}

\item Throughout this article, we fix a prime number $p$, as well as an algebraic closure
$\bar \Q$ of $\Q$ and $\bar{\Q}_p$ of $\Q_p$ and an embedding $\iota_p: \bar{\Q} \hookrightarrow \bar{\Q}_p$. Observe that $\iota_p$ yields a canonical embedding
$G_{\Q_p} \hookrightarrow  G_{\Q}$ through which we will see $G_{\Q_p}$ as a decomposition group of $G_{\Q}$ at $p$. We denote by $I_p$ the inertia subgroup of $G_{\Q_p}$.

\item We fix an embedding $\bar{\Q} \subset \C$  and a complex conjugation $c \in G_{\Q}$.

\item Let $\Q_p(1)$ denote the $G_\Q$ representation of dimension $1$ on which $G_{\Q}$ acts by the $p$-adic cyclotomic character $\epsilon_p:G_\Q \twoheadrightarrow \Z_p^{\times} \hookrightarrow \Q_p^{\times}$.

\item The Hodge-Tate-Sen weight of $\Q_p(1)$ is $-1$ and its Sen polynomial is $X+1$ (we are following the geometric convention).

\item Let $B_{\cris}$ denote the crystalline period ring endowed with the semi-linear Frobenius $\Phi$ and the natural $G_{\Q_p}$-action.

\item Let $t \in B_{\cris} $ be the element on which $G_{\Q_p}$-acts by $\epsilon_p$ and $\Phi(t)=p.t$. Note that $t$ generates the maximal ideal of the integral de Rham periods ring $B_{\dR}^{+}$; i.e $B_{\dR}^{+}/t.B_{\dR}^{+} \simeq \C_p$ as $G_{\Q_p}$-modules. 

\item Let $B_{\cris}^+ \subset B_{\cris}$ denote the ring of period defined in \cite[Expos\'e II, \S.2.3]{PP}.

\item Let $V $ be a $G_{\Q_p}$-representation of finite dimension over a $p$-adic field $L$. Let $\cD_{\cris}(V)$ denote the $L$-vector space $(B_{\cris} \otimes_{\Q_p} V)^{G_{\Q_p}}$ of dimension at most $\dim_L V$. And we denote again by $\Phi$ for the $L$-linear action given by $\Phi \otimes \mathrm{Id}_{V}$ on $\cD_{\cris}(V)$. Denote also by $\cD_{\cris}^+(V)$ for $(B_{\cris}^+ \otimes_{\Q_p} V)^{G_{\Q_p}}$. 

\item Let $z \in \cE_N$ be a classical point such that  the Galois representation $V$ attached to $z$ is crystalline. Then the $(\Phi,\Gamma)$-module attached to $V$ is trianguline in the sense of Colmez. However, the triangulation can be given by  non-\'etale $(\Phi,\Gamma)$-submodules, and hence $V_{\mid G_{\Q_p}}$ is not necessarily ordinary at $p$.

\item Remark that $\cD_{\cris}^+(\epsilon_p)=0$, $\cD_{\cris}(\epsilon_p)=\Q_p.t^{-1}$ ($t^{-1}$ is not in $ B^{+}_{\cris}$), and  $\cD_{\cris}^+(\epsilon_p^{-1})=\Q_p.t$.

\item Let  $\mathbbm{1}$ be the trivial representation of dimension $1$.

\item We shall always write $\Frob_\ell$ for the geometric Frobenius at the prime $\ell$.   

\item Let $\alpha \in \Q$, we shall denote $\cE_N^{\alpha}$ for the admissible open locus of $\cE_N$ defined by $\mid U_0 U_1 \mid_p=\alpha$.

\item We write $G_\Q^{Np}$ for the Galois group of the maximal extension of $\Q$ unramified outside of $Np$ and $\infty$. For any $G_{\Q}$-geometric representation $V$ we define the Bloch-Kato Selmer groups
 $$\rH^{1}_{f,\unr}(\Q,V)=\ker(\rH^1(\Q,V) \to \rH^1(\Q_{p},V \otimes B_{\cris}) \oplus_{\ell \nmid p} \rH^1(I_{\ell},V)) $$
 and
  $$\rH^{1}_{f}(\Q,V)=\ker(\rH^1(\Q,V) \to \rH^1(\Q_{p},V \otimes B_{\cris})).$$

\item Let $A$ be a ring and $M$ be a finite length $A$-module. We shall always denote by $l(M)$ the length of $M$ as $A$-module.

\end{enumerate}

\section{Some properties of automorphic $p$-adic representations}\label{section: bounds}
In this section we recall some facts about the Galois representations associated to classical and Siegel modular forms. The end of this section is devoted to the proof of the first part of Theorem \ref{A}.

\subsection{Ordinary Galois representations}
We recall in the following definition the notion of ordinary $p$-adic representations.

\begin{defn}\label{ordinaryrepsentation} Let $V$ be  finite-dimensional  $L$-vector space endowed with a continous $G_{\Q_p}$-action. We say that $V$ is ordinary, if it has a decreasing filtration by $G_{\Q_p}$-
subrepresentations $\ldots F^{i+1} \subset F^{i} \ldots$ satisfying $\bigcap_i F^i =\{0\}$ and $\bigcup_i F^i=V$, and such that all graded factors have of the following form

{\small

\[ F^{i} /F^{i+1} = V_i(i), \text{ with $V_i$ unramified at $p$, and $V_i(i)$ is the twist by $\epsilon_p^i$}.\]}
\end{defn}

Nekovar proved  the following useful theorem.
\begin{thm}\label{nekthm}\cite[Theorem   1.30]{Nek93}Any ordinary $p$-adic representation is semistable.
\end{thm}

\subsection{Elliptic modular forms}Recall that $(p,N)=1$, $f \in \mathrm{S}_{2k-2}(\Gamma_0(N),L)$ is an elliptic cuspidal eigenform with coeffiecients in the $p$-adic field $L$, and let $\rho_f:G_\Q \to \GL_2(L)$ be the $p$-adic Galois representation attached to $f$ in the sense that $L(\rho_f,s)=L(f,s)$. We note that $\rho_f^{\vee} \simeq \rho_f(2k-3)$ by the duality of $2$-dimensional representations. It is known that $\rho_f$ is de Rham and that its Hodge-Tate-Sen weights are $(2k-3,0)$. Moreover, $\rho_{f}$ is crystalline at $p$ since $p \nmid N$. 

Since $f_\alpha$ is ordinary at $p$, $(\rho_f)_{\mid G_{\Q_p}} \sim  \begin{pmatrix} 
\psi & *  \\ 
 0 & \psi^{-1}\epsilon_p^{3-2k}
\end{pmatrix} $, where $\psi:G_{\Q_p} \to L^\times$ is the unramified character such that $\psi(\Frob_p)=\alpha$ (where $U_p(f_\alpha)=\alpha f_{\alpha}$) and $\det \rho_f=\epsilon_p^{3-2k}$. Note that the characteristic polynomial of the $L$-linear Frobenius $\Phi$ acting of $\cD_{\cris}(\rho_f)$ is equal to the $p$-th Hecke polynomial of $f$.

\begin{prop}\label{locl}\ Let $\ell \mid N$ be a prime number.

\begin{enumerate}
\item Assume that $\pi_{f,\ell}$ is special at $\ell$, then $$\dim \mathrm{Ext}^1_{G_{\Q_\ell}}(\rho_f,\epsilon_p^{2-k})=\dim \mathrm{Ext}^1_{G_{\Q_\ell}}(\epsilon_p^{1-k},\rho_f)=\dim \rH^1(\Q_{\ell},\rho_f(k-1))=0.$$
\item Assume that $\pi_{f,\ell} \simeq \mathrm{St} \otimes \xi$ (i.e $a_{\ell}(f)=-\ell^{k-2}$), then $$\dim  \mathrm{Ext}^1_{G_{\Q_\ell}}(\rho_f,\epsilon_p^{1-k})=\dim \mathrm{Ext}^1_{G_{\Q_\ell}}(\epsilon_p^{2-k},\rho_f)=\dim \rH^1(\Q_{\ell},\rho_f(k-2))=0.$$

\end{enumerate}

\end{prop}

\begin{rem}

When $k=2$, the assumption that  $a_{\ell}=-1$ when $\ell \mid N$ is a prime holds if and only if the abelian variety $A_f$ attached to the weight $2$ cuspidal eigenform $f$ has  non-split multiplicative reduction at $\ell$.

\end{rem}

\begin{proof}
We know, in fact, that $(\rho_f)|_{G_{\Q_\ell}}= \begin{pmatrix} 
\psi_\ell^{-1}  & *     \\ 
       0 & \psi_\ell^{-1} \epsilon_p^{-1} 
\end{pmatrix}$ with infinite image of inertia, where $\psi_\ell$ is an unramified character such that $\psi_\ell(\Frob_\ell) =a_{\ell}(f)$. Note that by \cite[Theorem 4.6.17(2)]{Miyake} $a_\ell^2(f)=\ell^{2k-4}$. Our assumption on $a_{\ell}$ implies that $\rH^0(\Q_\ell, \rho_f(k-1))=\rH^0(\Q_\ell, \rho_f(k-2))=0$.

By applying the Euler characteristic formula and Tate duality, we obtain:
 $$\dim \rH^1(\Q_{\ell}, \rho_f(k-1)) = \dim \rH^0(\Q_{\ell}, \rho_f(k-1)) +\dim \rH^0(\Q_{\ell}, (\rho_f(k-1))^{\vee}(1)).$$

Since $\rho_f^{\vee}=\rho_f(2k-3)$ (the duality for $2$-dimensional representations), the above equality yields that \begin{equation}\label{vanishingatell} \dim \rH^1(\Q_{\ell}, \rho_f(k-1))=0.
\end{equation}
The other cases are proved similarly.

\end{proof}

\subsection{Siegel modular forms}
We define the abstract Hecke algebra $\mathcal{H}_N$ as the $\Z$-algebra generated by the Hecke operators $T_{\ell,1},T_{\ell,2}, \mathrm{S}_\ell$ for $\ell \nmid Np$ and the Hecke operators $U_0,U_1$ at $p$, where $T_{\ell,1}$ (resp. $T_{\ell,2}$, $\mathrm{S}_\ell$) is the Hecke operator attached to ${\rm diag}[1,1,\ell, \ell]$ (resp. ${\rm diag}[1,\ell,\ell^2, \ell]$, ${\rm diag}[\ell,\ell,\ell, \ell]$); and $U_0,U_1$ are the Hecke operators attached respectively to $\mathrm{diag}[1,1,p,p]$, $\mathrm{diag}[1,p,p^2,p]$ ($U_1$ is renormalized as in \cite[Theorem   2.4.14]{urban}).

We recall the $p$-adic properties of Galois representation arising from Siegel modular eigenforms. 
The following theorem has been proved by Laumon and Weissauer (see \cite{W} and \cite {Laumon}).

\begin{thm}\label{exrep} Let $\pi$ be a Siegel modular eigenform of central character $\omega_\pi$ of level $\Gamma(N)$ and of cohomological weight $k=(l_1,l_2)$  with corresponding Hecke character $\lambda_{\pi}:\mathcal{H}_N \to L_\pi$, where $L_\pi \subset \bar{\Q}_p$ is a $p$-adic field generated over $\Q_p$ by the systems of Hecke eigenvalues of $\pi$. Then there exist a finite extension $L_\pi'$ of $L_\pi$ and a $p$-adic representation $\rho_\pi : G_{\Q} \rightarrow \GL_4(L_\pi')$ unramified outside $Np$  and such that for all $\ell \nmid Np$, $$ \det (X.\mathrm{Id}- \rho_\pi(\Frob_\ell))=P_{\pi,\ell}(X),$$
where $P_{\pi,\ell}(X)$ is the Hecke-Andrianov polynomial at $\ell$ attached to $\pi$. Moreover, we have the symplectic relation :

\begin{equation}\label{symplectic} \overset{\vee}{\rho_\pi} \simeq \rho_\pi \otimes \chi^{-1}_\pi,
\end{equation}
and $\det \rho_\pi=\chi_\pi^{2}$. We also have the following relation between the similitude character $\chi_\pi$ and the central character:$$\omega_\pi\epsilon_p^{3-l_1 - l_2}=\chi_\pi.$$
\end{thm}

We have the following properties at $p$ of $\rho_\pi$ following from the works of Chai-Faltings, Laumon, Taylor, Urban and Weissauer (see \cite{Laumon}, \cite{U1}, \cite{Taylor2} \cite{W} and \cite{CF}).

\begin{thm}\label{Hodge-Tate} Under the notations of the above theorem we have :
\begin{enumerate}
\item The Galois representation $\rho_\pi $ is de Rham and its Hodge-Tate weights are \[\{0,l_2-2,l_1-1,l_1+l_2-3\}.\]
\item If $\pi$ is old at $p$ and semi-ordinary (i.e $U_{0}$ acts on $\pi$ by a $p$-adic unit), then the $p$-adic representation $\rho_{\pi}$ is crystalline at $p$, and the characteristic polynomial of $\Phi$ acting on $\cD_{\cris}(\rho_\pi)$ is the Hecke polynomial at $p$. The eigenvalues of the $L_\pi'$-linear Frobenius $\Phi$ acting on $\cD_{\cris}(\rho_\pi)$ are  $$\{ \lambda_{\pi}(U_{0}), \lambda_{\pi}(U_{1}.U_{0}^{-1}) p^{l_2-2}, \lambda_{\pi}(U_{0}.U_{1}^{-1})^{-1} p^{l_1-1}, \lambda_{\pi}(U_{0})^{-1} p^{l_1+l_2-3}\}.$$

\item Assume that $\pi$ is semi-ordinary at $p$ and of finite slope for $\mathbb{U}=U_0 U_1$, then \begin{equation}\label{2-dimensionalsubspace}(\rho_{\pi})_{\mid G_{\Q_p}} \sim \begin{pmatrix} 
\phi_\pi & * & * & * \\ 
       0 & * & * & * \\
 0 & * & * & * \\
 0 & 0 & 0 & \phi^{-1}_\pi\epsilon_p^{-l_1-l_2+3}
\end{pmatrix}, \end{equation} where  $\phi_\pi:G_{\Q_p} \to  L_\pi^\times$ is the unramified character having $\lambda_\pi( U_{0})$ as value at $\Frob_p$.

\end{enumerate}

\end{thm}

\begin{cor}\label{refinemen} Assume that $\pi$ is old at $p$, semi-ordinary, non-endoscopic and cohomological. Let $\rho_{\pi}^{I_p} $ be the $G_{\Q_p}$-stable line of $(\rho_\pi)_{\mid G_{\Q_p}}$ on which $G_{\Q_p}$ acts by the unramified character $\phi_\pi$, then the $2$-dimensional $G_{\Q_p}$-stable subspace $W_{\pi} $ of the quotient of $ (\rho_\pi)_{\mid G_{\Q_p}}$ by $\rho_{\pi}^{I_p}$ (see \eqref{2-dimensionalsubspace}) is crystalline with Hodge-Tate weight $(l_1-1,l_2-2)$. Moreover, the eigenvalues of the $L_\pi'$-linear Frobenius $\Phi$ acting on $\cD_{\cris}(W_\pi)$ are $\lambda_{\pi}(U_{1}U_{0}^{-1})  p^{l_2-2}$ and $\lambda_{\pi}(U_{0}U_{1}^{-1})  p^{l_1-1}$.
\end{cor}

\begin{rem}

Note that the $p$-adic Galois representation attached to a cuspidal Siegel eigenform is not necessarily irreducible. Schmidt makes the consequences of Arthur's classification for ${\rm GSp}_4$ explicit in \cite{Schmidt}. All cuspidal automorphic representations are either of type (G), (Y), (B), (Q), or (P). The latter three are CAP representations, with type (P) for the Siegel parabolic including the Saito-Kurokawa lifts. Type (Y) representations are endoscopic representations (``of Yoshida type''). Type (G) representations are ``stable'' in the sense that their transfer to ${\rm GL}_4$ stays cuspidal, and therefore their Galois representations are expected to be irreducible.
\end{rem}

Recall that  we write $\Delta:=\prod_{\ell \mid N} \Delta_{\ell} \cap {\rm GSp}_4(\Q)$ for the paramodular congruence subgroup of level $N$. If $N=1$ we put $\Delta={\rm GSp}_4(\Z)$.

\begin{prop} \label{noparamodendoscopic}
Let $\pi$ be a holomorphic Siegel modular eigenform of cohomological weight $(l_1, l_2)$, trivial central character $\omega_{\pi}$ and of paramodular level $\Delta$ for $N \geq 1$. Then $\pi$ is either of type (G) or a Saito-Kurokawa lift (in particular it is not endoscopic).
\end{prop}

\begin{proof}
This is proven in \cite[Lemma 2.5]{Schmidt} and \cite[Proposition 5.2]{Schmidt18}.
\end {proof}

\subsection{Properties at $\ell \ne p$  of a $p$-adic representation arising form a Siegel cusp form}

We have the following result on  the local properties of $\rho_\pi$ at the primes $\ell \mid N$ (compare \cite[Conj.3.1.7]{urban}) proved by \cite[Theorem 3.5]{Mok} (local-global compatibility up to Frobenius semi-simplification for type (G) representations) and \cite[Corollary 1]{Sch}  (monodromy rank 1). Mok \cite{Mok} used Arthur's classification for ${\rm GSp}_4$, whose proof was completed by Gee-Taibi in \cite{G-T}\footnote{Note the comments at  \cite[p. 472]{G-T} about this result still being conditional on unpublished work of Arthur and cases of the twisted weighted fundamental lemma.}.

\begin{prop}\label{monodromy} Under the notations of Theorem \ref{exrep}, and assuming that $\pi$ is non-CAP and non-endoscopic and $\pi^{\Delta_{\ell}} \neq 0$ for $\ell \mid N$, the rank of the monodromy operator of the Weil-Deligne representation attached to the Galois representation $(\rho_\pi)_{\mid G_{\Q_\ell}}$ is at most one.

\end{prop}

\subsection{Non-existence of endoscopic components of $\cE_\Delta$ specializing to $\pi_\alpha$} \label{noendoscopic}

Recall that $f$ is an elliptic eigenform of $ \mathrm{S}_{2k-2}(\Gamma_0(N),L)$, $\mathrm{SK}(f)$ is the Saito-Kurokawa lift of $f$ to a weight $(k,k)$ cuspform of paramodular level $\Delta$, and $\pi_\alpha$
 is a semi-ordinary $p$-stabilization of $\mathrm{SK}(f)$ (the eigenvalue of $\pi_\alpha$ for  $U_0$ is a $p$-adic unit). We recall also that $\pi_\alpha$ defines a point of the eigenvariety $\cE_\Delta$ which we denote again by $\pi_\alpha$.  The following theorem is a consequence of  \cite[3.2.9]{urban} and Proposition \ref{noparamodendoscopic}.

 \begin{thm}\label{irrecomplevel1}\

Assume $({ \bf Reg})$ and $k \geq 2$, then any irreducible affinoid $\mathcal{Z}$ of $\cE_\Delta$ of dimension two containing $\pi_\alpha$ is stable.
 \end{thm}

\begin{proof}
Assume that there exists such an irreducible affinoid $\mathcal{Z}$ that is not stable, i.e. that  the pseudo-character $\mathrm{Ps}_{\mathcal{Z}}:G_\Q \to \cO(\mathcal{Z})$ is reducible.
Since $\mathcal{Z}$ specializes to $\pi_\alpha$ and $\rho_f$ is absolutely irreducible, a subconstituent of the pseudo-character $\mathrm{Ps}_{\mathcal{Z}}:G_\Q \to \cO(\mathcal{Z}) \to \cO_{\mathcal{Z},\pi_\alpha}$ is a pseudo-character of dimension $2$ whose reducibility  locus is of dimension at most one. (One can rule out the existence of a $3$-dimensional irreducible constituent by specializing at sufficiently regular classical weights and applying the argument from the proof of Case A(iii) in \cite[Theorem 3.2.1]{urban}.) Hence one can find a sufficiently small affinoid neighborhood $\mathcal{X}=\Spm R$ of $\pi_\alpha$ with an odd representation $\varrho_1:G_\Q \to \GL_2(R)$ specializing to the $2$-dimensional odd representation $\rho_f$ and such that any classical specialization of $\varrho_1$ is irreducible, and a $2$-dimensional pseudo-character $\mathrm{Ps}_2:G_\Q \to R$  specializing to $\epsilon^{1-k}\oplus \epsilon^{2-k}$, and such that  $\Tr \varrho_1 +\Ps_2=\mathrm{Ps}_{\mathcal{X}}$ ($\Ps_2$ is the trace of a $2$-dimensional $G_\Q$-representation $\varrho_2$ valued in $\mathrm{Frac} R$). Moreover, the $p$-regularity assumption on $\pi_\alpha$ (when $k=2$) and \cite[Prop.3.3.6]{urban} yield (after shrinking again $\mathcal{X}$ to a smaller affinoid which we denote again by $\mathcal{X}$) that  $\varrho_1$ is ordinary at $p$ (in the sense that $\varrho_1^{I_p}$ is a direct summand in $\varrho_1$ of rank $1$). Hence, Theorem \cite[3.2.9]{urban} implies that any specialization of  $\mathcal{X}$ at a classical point $z \in \mathcal{X}$  of a cohomological weight is CAP or endoscopic. Note that $\pi_\alpha$ is of algebraic weight, has trivial central character, and is old at $p$. This implies that  the old (at $p$) classical points of trivial central character and non-parallel very regular weights  of $\mathcal{X}$ are  Zariski dense (see Cor.\ref{veryzariskidense1}). By \cite[Prop.3.3]{urban1} the specialization of $\mathcal{X}$ at these points cannot be a CAP form  and hence necessarily endoscopic by  Theorem \cite[3.2.9]{urban}. Since such forms do not exist by Proposition \ref{noparamodendoscopic} we get the desired contradiction.

\end{proof}

\section{The GMA $S$ and ordinarity of $S$-extensions occurring in $\rH^1(\Q,\rho_f(k-1))$}\label{ordrho}

Recall that Theorem \ref{irrecomplevel1} implies that all irreducible components of $\cE_\Delta$ passing through $\pi_\alpha$ are stable, and that $\cT$, the local ring of $\cE_\Delta$ at $\pi_\alpha$,  is reduced and equidimensional of dimension $2$ since $\cE_\Delta$ is reduced and equidimensional of dimension $2$. Let $\mathfrak{m}$ be the maximal ideal of $\cT$ and $L$ be the residue field of $\cT$.

 Let $A$ be a reduced Noetherian ring. Recall that the total fraction ring of $A$ is the fraction ring $Q(A) := \mathcal{S}^{-1}A$ where $\mathcal{S} \subset A$ is the multiplicative subset of non-zerodivisors of $A$. We check at once that the natural map $A \rightarrow  \mathcal{S}^{-1}A$ is injective and flat, and that the non-zerodivisors of $A$ are invertible in $\mathcal{S}^{-1}A$. Moreover, since $A$ is Noetherian the zero divisors  of $A$ are the elements of the union of the (finitely many) minimal primes ideal of $A$, so $\mathcal{S}^{-1}A=\prod_{\gp_i} A_{\gp_i}$, where $\gp_i$ runs over the minimal prime ideals of $A$. Moreover, each $A_{\gp_i}$ is a field, since it is reduced, local and of Krull dimension equal to zero. 
  Let $K=\prod K_i$ be the total field of fractions of the reduced equidimensional ring $\cT$, where $K_i$ is the localisation of $\cT$ at a minimal prime ideal.
 
 \begin{defn}[Definition/Proposition]
 The pseudo-character\footnote{The pseudo-character $\mathrm{Ps}_{\cT}$ is obtained by composing $\mathrm{Ps}_{\cE_\Delta}$ with the localization map $\cO(\cE_\Delta) \to \cT$.} $$\mathrm{Ps}_{\cT}: G_\Q \rightarrow \cO(\cE_\Delta) \rightarrow \cT$$ is residually multiplicity free and the corresponding  Cayley-Hamilton faithful algebra $$S:=\cT[G_\Q]/ \ker\mathrm{Ps}_{\cT}$$ can by \cite[Theorem   1.4.4(i)]{bb} be equipped with the structure of a GMA (in the sense of \cite[Def. 1.3.1]{bb}). It is of finite type and torsion-free as $\cT$-module. Since $\cT$ is reduced we further have an associated  Galois representation $\rho_K:G_\Q \rightarrow \GL_4(K)$ by  \cite[Theorem   1.4.4(ii)]{bb}. Note that $\rho_K:G_\Q \rightarrow \GL_4(K)$ is absolutely irreducible, since all the minimal prime ideals of $\cT$ correspond to stable irreducible components of $\cE_\Delta$ passing through $\pi_\alpha$ (so each Galois representation $\rho_{K_i}:G_\Q \rightarrow \GL_4(K_i)$ is irreducible).
 \end{defn}

Assume until the end of this paper that $\alpha \ne 1$ when $k=2$ (which we will refer to as ``$p$-adic regularity''). Recall that $\varrho=\begin{pmatrix} 
\epsilon_p^{2-k} & 0 & 0 \\ 
       0 & \rho_f & 0 \\
 0 & 0 & \epsilon_p^{1-k} 
\end{pmatrix}$ is the Galois representation attached to $\pi_\alpha$ in a basis such that $\varrho(\tau)\sim \begin{pmatrix} 
\mu  & 0 & 0 & 0 \\ 
       0 & \alpha & 0 & 0 \\
 0 & 0 & \beta & 0 \\
 0 & 0 & 0 & \gamma
\end{pmatrix}$, where the eigenvalues of $\tau \in G_{\Q_p}$ are all distinct (since $\alpha \ne 1$ when $k=2$ )  (necessarily in this basis { \small $\varrho(G_{\Q_p}) \sim \begin{pmatrix} 
\epsilon_p^{2-k}  & 0 & 0 & 0 \\ 
       0 & \psi & * & 0 \\
 0 & 0 & \psi^{-1}\epsilon_p^{3-2k} &0  \\
 0 & 0 & 0 & \epsilon_p^{1-k}
\end{pmatrix}$ }).

\begin{rem}

Note that the character $\phi_{\pi_\alpha}$ of Theorem \ref{Hodge-Tate}(iii) equals $\psi$ since $U_p(f_\alpha)=\alpha.f_\alpha$ and $U_{0}(\pi_\alpha)=\alpha. \pi_\alpha$.

\end{rem}

Since all reducible components of $\varrho$ have multiplicity one, \cite[Theorem   1.4.4]{bb} implies that there exist orthogonal idempotents $\tilde{e}_1,\tilde{e}_2,\tilde{e}_3$ of $S=\mathrm{Im}(\cT[G_\Q] \rightarrow M_4(K))$  lifting the idempotents $e_1,e_2,e_3$ of $\varrho$, and corresponding respectively to $\epsilon_p^{2-k}, \rho_f, \epsilon_p^{1-k}$. Moreover, we can see $S$ as $$S=\begin{pmatrix} 
\cT & M_{1,2}(\cT_{1,2}) & \cT_{1,3} \\ 
 M_{2,1}(\cT_{2,1}) & M_2(\cT) & M_{2,1}(\cT_{2,3}) \\
 \cT_{3,1} & M_{1,2}(\cT_{3,2}) & \cT 
\end{pmatrix},$$ where $\cT_{i,j}$ are fractional ideals of $K$ ($\cT_{i,j}$ are finite type $\cT$-modules).

Put $\rho_1=\epsilon_p^{2-k}$, $\rho_2=\rho_f$ and $\rho_3=\epsilon_p^{1-k}$. We recall Bella{\"{\i}}che and Chenevier's definition of reducibility ideals:

\begin{defn}[\cite{bb} Definition  1.5.2, Proposition 1.5.1] \label{redideal} Let $\mathcal{P}=(\mathcal{P}_1, \dots, \mathcal{P}_s)$ be a partition of the set $\mathcal{I}  = \{1,2,3\}$. The ideal of reducibility $I^{\mathcal{P}}$ (associated to ${\rm Ps}_{\cT}$ and the  partition $\mathcal{P}$) is the smallest ideal $I$ of $\cT$ with the the property that
 there exist pseudocharacters $T_1, \dots, T_s: \cT/I[G_{\Q}] \to \cT/I$ such that \begin{itemize}
\item[(i)] ${\rm Ps}_{\cT} \otimes \cT/I= \sum_{l=1}^s T_l$,
\item[(ii)] for each $l \in \{1, \dots, s\}$, $T_l \otimes L =\sum_{i \in \mathcal{P}_l} {\rm trace} \rho_i$. \end{itemize} \end{defn}

\begin{prop}[\cite{bb} Proposition 1.5.1, \cite{bk} Corollary 6.5] \label{redideal2}
One has $$I^{\mathcal{P}} = \sum_{\substack{(i,j)\\ \textup{$i,j$ not in the same $\mathcal{P}_l$}}} \cT_{i,j} \cT_{j,i}.$$
\end{prop}

For $\mathcal{P}=\{\{1\}, \{2\}, \{3\}\}$ we write $$\mathcal{I}^{\tot}:=\mathcal{I}^{\mathcal{P}}=\cT_{3,1} \cT_{1,3} +\cT_{2,3} \cT_{3,2} + \cT_{1,2} \cT_{2,1}.$$

Let $\cT'_{i,j}=\cT_{i,k} \cT_{k,j}$ for $i, j, k$ distinct.
Since  the maximal ideal $\gm$ of $\cT$ contains the total reducibility ideal $\mathcal{I}^{\tot}$ \cite[Theorem 1.5.5]{bb}  implies that for $i \neq j \in \{1, 2, 3\}$ there exists an injective homomorphism of $L$-modules { \small \begin{equation}\label{1.5.5extn} \Hom(\cT_{i,j}/\cT_{i,j}',L) \hookrightarrow \rH^{1}(\Q_{Np},\rho_i^{\vee} \otimes \rho_j \otimes L). \end{equation} }

\begin{thm}\label{Bcrys} Assume that $\alpha \ne 1$ when $k=2$. For $(i,j)=(1,2)$ the injective homomorphism of $L$-modules of \eqref{1.5.5extn} gives rise to \begin{equation}\label{imcrys} \Hom(\cT_{1,2}/\cT_{1,2}',L) \hookrightarrow \rH^{1}_{f,\unr}(\Q,\rho_f(k-1)). \end{equation}

\end{thm}

\begin{proof}
The proof of \cite[Theorem 1.5.5]{bb} tells us that the homomorphism \eqref{1.5.5extn} is given by \begin{equation}\label{ext}
\begin{split}
\Hom(\cT_{1,2}/\cT_{1,2}',L) \hookrightarrow \rH^1(\Q, \rho_f(k-1)) \\ h  \mapsto (g \rightarrow  h(\bar{b}_1(g),\bar{b}_2(g))\rho^{-1}_f(g)), 
\end{split}
\end{equation}
where  $(\bar{b}_1(g),\bar{b}_2(g))$ is the class of $t_{1,2}(g)=(b_1(g),b_2(g)) \in M_{1,2}(\cT_{1,2})$ in $M_{1,2}(\cT_{1,2}/\cT_{1,2}')$.  The classical points old at $p$ and of regular weights form a  very Zariski dense set $\Sigma$ in every irreducible component of $\cE_\Delta$ specializing to $\pi_\alpha$ (see Lemma \ref{zdop} and \cite[Prop.3.3.6]{urban}). By Theorem \ref{Hodge-Tate}, the set of Hodge-Tate-Sen weights of the semi-simple representation $\rho_y$ attached to any point $y \in \cE_\Delta$ corresponding to a classical cuspidal Siegel eigenform old at $p$ of weight $(l_1,l_2)$ is $\{0,l_2-2,l_1-1,l_1+l_2-3\}$ and $\rho_y$ is crystalline at $p$.

 On the other hand, for any $y \in \Sigma$, let us denote by $\rho_y:G_\Q \rightarrow \GL_4(L_{y})$ the semi-simple $p$-adic Galois representation attached to the Siegel eigenform corresponding to $y$ (i.e.$\Tr \rho_y$ is the specialization of the universal pseudo-character $\mathrm{Ps}_{\cE_\Delta}:G_\Q \to \cO(\cE_\Delta)$ at $y$). Theorem \ref{Hodge-Tate} implies that $\dim \cD^{+}_{\cris}(\rho_y)^{\Phi=U_{0}(y)}=1 $, and then $(\mathrm{Ps}_{\cE_{\Delta}},\Sigma,U_0,\{\kappa_i\})$ is a weakly refined family  (in the sense of  \cite[Definition 4.2.7]{bb}) since $U_{0} \in \cO(\cE_\Delta)^{\times}$. Note also that condition (*) of \cite[Definition 4.2.7]{bb} is satisfied since we have a torsion free morphism $\kappa:\cE_\Delta \to \cW$; condition (v) of \cite[Definition 4.2.7]{bb} is satisfied by Lemma \ref{vzd}, Lemma \ref{zdop} and Corollary \ref{veryzariskidense1} (so the classical points of $\mathcal{E}_\Delta$  which are old at $p$ accumulate to $\pi_\alpha$).

Moreover, $\dim \cD^{+}_{\cris}(\varrho)^{\Phi=\alpha=U_{0}(\pi_\alpha)}=1$ by regularity assumption on $\varrho$ at $p$. 
 Hence, \cite[Theorem   4.3.6]{bb} implies that any $G_{\Q}$-representation $V$ corresponding to a cohomology class in the image of the morphism \eqref{ext} satisfies $$\dim \cD_{\cris}^{+}(V)^{\Phi=\alpha}=1.$$

We use this to first prove that $V$ is crystalline at $p$. One can see $V$ as the following $ G_\Q^{Np}$-extension:

$$ 0 \to \epsilon_p^{2-k} \to V \to \rho_f \to 0.$$ 

Let $\tilde{\rho}=\begin{pmatrix} 
\epsilon_p^{2-k}  & *  \\ 
       0 & \rho_f 
\end{pmatrix}$ be the realization of $V$ by a matrix. The restriction to $G_{\Q_p}$ of $\tilde{\rho}$ has the form $\begin{pmatrix} 
\epsilon_p^{2-k}  & b   & c  \\ 
       0 & \psi  & * \\
       0  & 0 &  \psi^{-1} \epsilon_p^{3-2k}
\end{pmatrix}$. Hence, we have an extension of $G_{\Q_p}$-modules $$0 \rightarrow \begin{pmatrix} 
\epsilon_p^{2-k}  & b     \\ 
       0 & \psi 
\end{pmatrix} \rightarrow \tilde{\rho}_{\mid G_{\Q_p}} \rightarrow \psi^{-1}\epsilon_p^{3-2k} \rightarrow 0.$$

Let $V^{0} \subset V$ be the $L$-vector space of dimension $2$ on which $G_{\Q_p}$ acts by $ \begin{pmatrix} 
\epsilon_p^{2-k}  & b     \\ 
       0 & \psi 
\end{pmatrix}$.   By applying the left exact functor $\cD^{+}_{\cris}(.)^{\Phi=\alpha}$ to the above exact sequence, we obtain  $$ \cD_{\cris}^{+}( V^0 )^{\Phi=\alpha} \simeq \cD_{\cris}^{+}(\tilde{\rho}_{\mid G_{\Q_p}})^{\Phi=\alpha}.$$

Since $\dim \cD_{\cris}(V)^{\Phi=\alpha}=1$, we get $\dim \cD_{\cris}^{+}(V^0 )^{\Phi=\alpha}=1$. Hence,  $V_0=\begin{pmatrix} 
\epsilon_p^{2-k}  & b     \\ 
       0 & \psi  
\end{pmatrix} $ is crystalline at $p$ which implies that the cohomology class of $ b$ in $\mathrm{Ext}^1_{G_{\Q_p}}(\psi,\epsilon_p^{2-k})$ is trivial (i.e.$\tilde{\rho}_{\mid G_{\Q_p}} \simeq \begin{pmatrix} 
\epsilon_p^{2-k} & 0   & c  \\ 
       0 & \psi& * \\
       0  & 0 &  \psi^{-1}\epsilon_p^{3-2k}
\end{pmatrix} $). Thus, $\tilde\rho$ is ordinary in the sense of Definition \ref{ordinaryrepsentation} and then semi-stable (hence de Rham) at $p$ by Theorem   \ref{nekthm}. 
Therefore the extension $V$ gives a cohomology class in  $$\rH^1_g(G_\Q^{Np}, \rho_f(k-1))=\ker(\rH^1(\Q, \rho_f(k-1)) \rightarrow \rH^1(\Q_{p}, \rho_f(k-1) \otimes B_{\dR})).$$ Since $\rH^1_g(G_\Q^{Np}, \rho_f(k-1)) \simeq \rH^1_f(G_\Q^{Np}, \rho_f(k-1))$ (see e.g. \cite[Lemme 4.1.3]{urban}) we deduce that $V$ is crystalline at $p$.

Finally, the restriction of the map $$\rH^1(\Q, \rho_f(k-1)) \rightarrow \rH^1(I_{\ell}, \rho_f(k-1)),$$ when $\ell \mid N$ is trivial by Proposition \cite[Lemme 4.1.3]{urban}.

\end{proof}

\subsection{Symplectic relation and the the anti-involution $\tau$ on $S$}\label{symmetry}

Recall that \begin{equation}
\mathrm{Ps}_{\mathcal{E}_{\Delta}} :G_\Q \to  \cO(\mathcal{E}_{\Delta})\end{equation}
are pseudo-characters of dimension $4$ and since the classical points of $\mathcal{E}_{\Delta}$ are Zariski dense, the relation \eqref{symplectic} implies that the pseudo-character $\mathrm{Ps}_{\cT}$ is invariant under the anti-involution $$\tau:\cT[G_\Q] \to \cT[G_\Q] \text { sending }g \to \epsilon_p^{-\kappa_1-\kappa_2+3}.g^{-1},$$ where $ \epsilon_p^{-\kappa_1-\kappa_2+3}$ is the character $G_\Q \to  \cO(\cU)^{\times}$ interpolating the similitude character of the $G_\Q$-semi-simple representations whose trace correspond to the classical specializations of the pseudo-character $\mathrm{Ps}_\cU:G_\Q \to \cO(\mathcal{E}_{\Delta}) \to \cO(\cU)$ for a small enough connected affinoid neighborhood $\mathcal{U}$ of $\pi_\alpha$. 

Hence $\tau$ yields an anti-automorphism on $S$ given by $\rho_K \circ \tau$ and it follows from \cite[Lemma.1.8.3]{bb} that we can choose our idempotent $\tilde{e}_1,\tilde{e}_2,\tilde{e}_3$ of $S$ lifting the idempotents $e_1,e_2,e_3$ attached respectively to $\epsilon_p^{2-k},\rho_f,\epsilon_p^{1-k}$, and such that $\tilde{e}_{\tau(1)}=\tilde{e}_3$ ($\tau$ preserves the idempotent corresponding to $\rho_f$, and switches the idempotents corresponding to $\epsilon_p^{1-k}$, $\epsilon_p^{2-k}$). 

By \eqref{1.5.5extn} there exists an injection \begin{equation}\label{extinv} \Hom(\cT_{2,3}/\cT'_{2,3},L) \hookrightarrow \mathrm{Ext}^1_{G_\Q^{Np}}(\epsilon_p^{1-k},\rho_f) \simeq  \rH^{1}(\Q,\rho_f(k-1)).\end{equation}

Proposition \cite[Prop.1.8.6]{bb} yields immediately the following corollary.

\begin{cor}\label{antiinvB'}Assume that $\alpha \ne 1$ when $k=2$. Then the image of \eqref{extinv} lands in $\rH^{1}_{f,\unr}(\Q,\rho_f(k-1))$ and has dimension equal to the dimension of the image of the morphism \eqref{imcrys}.

\end{cor}

\section{Crystallinity of the $S$-extensions occuring in $\rH^1(\Q,\epsilon_p)$}\label{cnsht}

In this section we show using the analytic continuation of the crystalline periods in a family of $p$-adic $G_{\Q_p}$-representations of generic rank $3$ interpolating $\{ \rho_z/\rho_z^{I_p}, z \in \cE_\Delta^{\mid U_0 \mid_p=1} \}$ the crystallinity of the $S$-extensions occuring in $\rH^1(\Q,\epsilon_p)$. Assume in this section  $({ \bf Reg})$, $({ \bf St})$ and that $k \geq 2$.

By \eqref{1.5.5extn} we have a natural injection \begin{equation}\label{ext1cyc}\begin{split}
\Hom(\cT_{1,3}/\cT'_{1,3},L) &\hookrightarrow \mathrm{Ext}^1_{G_\Q^{Np}}(\epsilon_p^{1-k},\epsilon_p^{2-k}) \simeq \rH^{1}(G_\Q^{Np},\epsilon_p) \\ h & \mapsto (g \rightarrow  \frac{h(\bar{t}_{1,3}(g))}{\epsilon_p^{1-k}(g)}), \end{split}
\end{equation}

where $\bar{t}_{1,3}(g)$ is the class of $t_{1,3}(g) \in \cT_{1,3}$ in $\cT_{1,3}/\cT'_{1,3}$.

 Now we have to determine the exact image of the injective morphism $\eqref{ext1cyc}$. In \cite[\S1.5.4]{bb}, Bella{\"{\i}}che-Chenevier introduce a left ideal $M_3=S.E_3$ of $S \subset M_4(K)$ which is the third column of the GMA matrix $S$ and hence it is a projective left $S$-module (see \cite[1.3.3]{bb} for the definition of $E_3$), and they proved in \cite[Theorem   1.5.6]{bb} and \cite[Lemma.4.3.9]{bb} the following results: 

\begin{enumerate}\label{N}

\item  There exists an exact sequence of $S$-left modules \begin{equation}\label{multfree1of1} 0 \rightarrow E \rightarrow M_3/\gm M_3\rightarrow \epsilon_p^{1-k} \to 0 \end{equation}

\item Any simple $S$-subquotient of $E$ occurs in the set $\{\rho_f, \epsilon_p^{2-k}\}$ (in particular it is not isomorphic to $\epsilon_p^{1-k}$).

\item The image of the morphism \eqref{ext1cyc} consists of extensions occurring as quotient of the  $S/\gm S$-module $M_3/\gm M_3 \oplus \epsilon_p^{2-k}$ by an $S$-submodule $\mathcal{Q}$ such that the $S$-simple subquotient of $\mathcal{Q}$ occurs in  $\{\rho_f, \epsilon_p^{2-k}\}$ (in particular it is not isomorphic to $\epsilon_p^{1-k}$).
\end{enumerate}

We will need the following additional property:
\begin{lemma} \label{S_p}
Let $\mathrm{S}_p$ be the subring generated by the image of $G_{\Q_p}$ in $S$. Then the $\mathrm{S}_p$-simple subquotients of $\mathcal{Q}$ occur in $\{\epsilon_p^{2-k}, \psi, \psi^{-1}\epsilon_p^{3-2k}\}$.
\end{lemma}
\begin{proof}
Let ${\rm Ps}_p$ be the restriction of ${\rm Ps}_{\cT}$ to $\mathrm{S}_p$. By \cite[Lemma 1.2.7]{bb} we have $S/{\rm rad}(S)\cong \overline{S}/{\rm ker}\overline{{\rm Ps}}$ and $\mathrm{S}_p/{\rm rad}(\mathrm{S}_p)\cong \overline{S}_p/{\rm ker}\overline{{\rm Ps}}_p$, hence ${\rm rad}(S ) \cap \mathrm{S}_p \subset {\rm rad}(\mathrm{S}_p)$, and we obtain a morphism $\mathrm{S}_p/{\rm rad}(S ) \cap \mathrm{S}_p \twoheadrightarrow \mathrm{S}_p/{\rm rad}(\mathrm{S}_p)=\overline{S}_p/{\rm ker}\overline{{\rm Ps}}_p\cong \prod_{i=1}^4 {\rm End}_L(\rho_i),$ where $\rho_i \in \{\epsilon_p^{2-k}, \psi, \psi^{-1}\epsilon_p^{2k-3}, \epsilon_p^{1-k}\}$. In particular, one can see that all $\{\epsilon_p^{2-k}, \psi, \psi^{-1}\epsilon_p^{2k-3}, \epsilon_p^{1-k}\}$ are simple $\mathrm{S}_p$-modules. Now, we claim that any simple $\mathrm{S}_p$-representation occurs in \[ \{\epsilon_p^{2-k}, \psi, \psi^{-1}\epsilon_p^{2k-3}, \epsilon_p^{1-k}\},\] and it follows immediately from the injection $\mathrm{S}_p/{\rm rad}(S ) \cap \mathrm{S}_p  \hookrightarrow S/{\rm rad}(S) \simeq \overline{S}/{\rm ker}\overline{{\rm Ps}}\cong \prod_{i=1}^3 {\rm End}_L(\rho_i)$ whose image is $\prod_{i=1}^3 {\rm End}_L((\rho_i)_{\mid G_{\Q_p}})$ (so $\mathrm{S}_p/{\rm rad}(\mathrm{S}_p)$ is a semi-simple quotient of $\prod_{i=1}^3 {\rm End}_L((\rho_i)_{\mid G_{\Q_p}})$).

The rest of the lemma follows  from the fact (see \eqref{multfree1of1}(iii)) that the $S$-module $\mathcal{Q}$ has a Jordan-Holder sequence, all subquotients of which are isomorphic to either $\rho_f$ or $\epsilon_p^{2-k}$, and  it has a refinement as $\mathrm{S}_p$-module for which the $\mathrm{S}_p$-simple subquotients occur in $\{\epsilon_p^{2-k}, \psi, \psi^{-1}\epsilon_p^{2k-3}\}$.

\end{proof}

We recall that a torsion-free $A$-module is a module over a ring $A$ such that $0$ is the only element annihilated by a regular element (i.e non-zero-divisor of $A$) of the ring. A coherent sheaf $\mathcal{F}$ over a rigid analytic space $X$ is a sheaf of $\mathcal {O}_{X}^{\rig}$-modules such that there exists an admissible covering of $X$ by affinoid subdomains $\{U_i= \Spm R_i\}$ of $X$ for which the restriction $\mathcal{F}_{\mid U_i}$ is associated to $\tilde{M_i}$ and $M_i$ is a finite type $R_i$-module.

The sheaf $\mathcal{F}$ is said to be torsion-free if all those modules $M_i$ are torsion-free over their respective rings. Alternatively, $\mathcal{F}$ is torsion-free if and only if it has no local torsion sections.

\begin{lemma}\ \begin{enumerate}

\item One has $M_3 \subset K^4$ and $M_3.K=K^4$. Moreover, $M_3$ is a $\cT$-torsion-free lattice of the representation $\rho_K \to \GL_4(K)$.

\item The natural morphism $M_3 \to M_3 \otimes_{\cT} K$ is injective and the natural morphism $$M_3 \otimes_{\cT} K \to M_3.K$$ is an isomorphism.
\end{enumerate}
\end{lemma}

\begin{proof}
i) Note that the finite type $\cT$-module $M_3$ corresponds to the third column of the GMA matrix $S \subset M_4(K)$, hence $M_3 \subset K^4$. Since $\rho_K:G_\Q \to S^{\times} \subset \GL_4(K)$ is irreducible, then $M_3.K$ is necessarily of rank $4$ over $K$. 

ii) Recall that  $M_3 \otimes_{\cT} K= M_3 \otimes_{\cT} \mathcal{S}^{-1} \cT$, where $\mathcal{S}$ is the set of non-zero divisors of $\cT$. Hence, $M_3 \otimes_{\cT} K= \mathcal{S}^{-1}M_3$ and the injection follows from the fact that $M_3$ is torsion-free. It yields also that the natural surjection $M_3 \otimes_{\cT} K \twoheadrightarrow M_3.K=K^4$ is in fact an isomorphism.
\end{proof}

\begin{thm}\label{lastentry}Assume $({ \bf Reg})$ and $({ \bf St})$.  Then the image of the injective morphism of $L$-modules $$\Hom(\cT_{1,3}/\cT'_{1,3},L) \hookrightarrow \mathrm{Ext}^1_{G_\Q^{Np}}(\epsilon_p^{1-k},\epsilon_p^{2-k}) \simeq \rH^{1}(G_\Q^{Np},\epsilon_p)$$
lands in $\rH^{1}_{f,\unr}(\Q,\epsilon_p)=\{0\}$.

\end{thm}

\begin{proof} 

To simplify notation, let $M$ denote the finite type   $S$-module $M_3$. We recall that $M$ is a torsion-free finite type $\cT$-module, because $S$ is of finite type over $\cT$ and $M \subset S \subset M_4(K)$. According to \cite[Lemma.4.3.7]{bb}, there exists an open affinoid neighborhood $\mathcal{U}=\Spm A$ of $\pi_\alpha$ inside $\cE_\Delta$ such that we can extend $M$ to an analytic torsion-free coherent sheaf $\tilde{\mathcal{M}}$ over $\mathcal{U}$ ($\mathcal{M}$ is the $A$-module associated to $\tilde{\mathcal{M}}$) and such that:

\begin{enumerate}

\item $Q(A) \otimes \mathcal{M}=Q(A)^4$ (i.e the generic rank of $\mathcal{M}$ is $4$ \footnote{We have to choose $\Spm A$ small enough in the aim that it is connected and it contains no more irreducible components than $\Spec \cT$, to have a natural inclusion $Q(A) \subset K$.}) 

\item $\mathcal{M} \otimes_A \cT=M$ (i.e the stalk of $\tilde{\mathcal{M}}$ at $\pi_\alpha$ is $M$). 

\item The $A$-module $\mathcal{M}$ carries a continuous action of $G_\Q$ compatible with the action of $G_\Q$ on its localization $M$ at $\pi_\alpha$, and the generic representation $G_\Q \to \GL_4(Q(A))$ is semi-simple and its trace is just the trace given by $\mathrm{Ps}_A:G_\Q \rightarrow \cO(\cE_\Delta) \to A$. 

\end{enumerate}

On the other hand, by semi-ordinarity at $p$, the action of $I_p$ on $Q(A)^4$ stabilizes a line $(Q(A)^4)^{I_p}$ on which $\Frob_p$ acts by $U_{0}$. Let $\tilde{\mathcal{L}}$ be the subsheaf of $\mathcal{M}$ given by 
$(Q(A)^4)^{I_p} \cap \mathcal{M}$ (i.e the sections of $\tilde{\mathcal{L}}$ are the sections of $\tilde{\mathcal{M}}$ on which $I_p$ acts trivially and $\Frob_p$ acts by $U_{0}$). Moreover, $\tilde{\mathcal{L}}$ is the coherent sheaf associated to the $A$-submodule $\mathcal{L}$ of $\mathcal{M}$  given by the elements which are invariant under the actions of the inertia $I_p$ and on which $\Frob_p$ acts by $U_0$.

Let $\tilde{\mathcal{M}'}_+$ be the quotient presheaf $\tilde{\mathcal{M}}/\tilde{\mathcal{L}}$ and $\tilde{\mathcal{M}'}$ be the sheaf associated to the presheaf $\tilde{\mathcal{M}}'_+$, and it is $\widetilde{\mathcal{M}/\mathcal{L}}$ since $\mathcal{U}$ is an affinoid, and is endowed naturally with an action of $G_{\Q_p}$.

Let $M':=\mathcal{M}' \otimes_A \cT$. Since $M_K :=M \otimes_{\cT} K= M. K=K^4$, it is obvious that $M' $ is a $\cT$-submodule of $K^4/(K^4)^{I_p}$, where $(K^4)^{I_p}$ means the $I_p$-invariant subspace on which $\Frob_p$ acts by $U_0$.
Hence, $M'$ is a finite type torsion-free $\cT$-module of generic rank $3$ over $K$, and the regularity assumption when $k=2$ yields that the $G_{\Q_p}$-semi-simplification $M' \otimes_{\cT} L$ does not contain $\psi$. 

Similarly, since $Q(A) \subset K$ and $(K^4)^{I_p} \cap \mathcal{M}=\mathcal{M}^{I_p} $, we obtain that $\mathcal{M}' $ injects into $K^4 /(K^4)^{I_p}$ and $\mathcal{M}'$ is torsion-free over $A$ and with generic rank equal to $3$. Moreover, the regularity assumption  yields that the $G_{\Q_p}$-semi-simplification of its specialization at $\pi_\alpha=x$ does not contain the character $\psi_{\mid G_{\Q_p}}$. In fact, Corollary \ref{refinemen} implies the characteristic polynomial of the  Frobenius $\Phi$ acting on the crystalline module of almost of the classical specializations $y$ of $\mathcal{M}'$ has no root equal to $U_{0}(y)$.  

Let $Z=V(\mathcal{I}) \subset \mathcal{U}$ be the Zariski closed set defined by the ideal $\mathcal{I}$ generated by the $4$-th Fitting ideal $\mathrm{Fitt}_4$ of the $A$-module $\mathcal{M}$ and by the $3$-rd Fitting ideal $\mathrm{Fitt}_3$ of the $A$-module $\mathcal{M}'$, then any point $y$ lies in $V(\mathrm{Fitt}_4)$ (resp. $V(\mathrm{Fitt}_3)$) if and only if $\dim_{k(y)}(\mathcal{M}(y)) \geq 5$ (resp. $\dim_{k(y)}(\mathcal{M}'(y)) \geq 4$), where $\mathcal{M}(y)$ (resp. $\mathcal{M}'(y)$) is the fiber of $\mathcal{M}$(resp. $\mathcal{M}'$) at $y$ and $k(y)$ is the residue field at $y$.

Thus $\mathcal{U}-V (\mathcal{I})$ is the biggest admissible open subset of $\mathcal{U}$ on which $\mathcal{M}$ (resp. $\mathcal{M}'$) can be locally generated (on stalks) by $4$ elements (resp. $3$ elements). Moreover, since the coherent $\mathcal{M}$ (resp. $\mathcal{M}'$) is generically of rank $4$ (resp. $3$) and torsion-free then one can deduce that the coherent sheaf $\mathcal{M}$ (resp. $\mathcal{M}'$) is locally free of rank $4$ (resp. $3$) on the admissible open $\mathcal{U}-Z=\mathcal{U}'$ ($\mathcal{U}'$ does not necessarily contain  $\pi_\alpha$). Thus, the direct summand $\mathcal{M}'$ of $\mathcal{M}$ is also locally free of rank $3$ on $\mathcal{U}'$. Hence one can deduce that the Hodge-Tate weights of the specialization of $\mathcal{M}'$ at classical points of $\mathcal{U}'$ of weight $l_1>l_2+1$ and having crystalline representation (they form a very Zariski dense set) are $l_2-2,l_1-1,l_1+l_2-3$; and then $l_2-2$ is the smallest Hodge-Tate weight (see Corollary \ref{refinemen}).  

In addition, if $\mathcal{M'}(y)$ (resp. $\mathcal{M}(y)$) denotes the specialization of $\mathcal{M}'$ (resp. $\mathcal{M}$) at a very  classical point $y \in \mathcal{U}'$. We can enlarging $Z$ if it is necessary to have that for any $y \in \mathcal{U}'$, $\mathcal{M}(y)^{ss}=\mathcal{M}(y)$.  Now, if $y \in \mathcal{U}'$ is a classical point  of weight $(l_1,l_2)$ and $\rho_y$ is a crystalline representation at $p$, then the eigenvalues of the  Frobenius $\Phi$ acting on $\cD_{\cris}(\mathcal{M'}(y))$ are $\lambda_{y}(U_{1}U_{0}^{-1})  p^{l_2-2}$, $\lambda_{y}(U_{0}U_{1}^{-1})  p^{l_1-1}$ and $\lambda_{y}(U_{0}^{-1})  p^{l_2+l_1-3}$, where $\lambda_y: \Spm L_y \to \cE_\Delta$ is the morphism corresponding to $y$. When $y=x$,  we have $\lambda_{x}(U_{1}U_{0}^{-1})=p$ ($\underline{k}=(k,k)$ is the weight of $\pi_\alpha$). 

The exact sequence \eqref{multfree1of1} (i.e. $\epsilon_p^{1-k}$ occurs with multiplicity one in $\mathcal{M}'/\gm \mathcal{M}'$), the regularity assumption (i.e. $\alpha \ne 1$) of $\varrho$ at $p$ when $k=2$, and the fact that $\mathcal{M}'\otimes_A \mathcal{T}=M'$ (since $\mathcal{M}^{I_p} \otimes_{A} \cT=M^{I_p}$), yield that
\begin{equation}\label{multM'}
\dim \cD^{+}_{\cris}(\mathcal{M}'(\pi_\alpha)^{ss})^{\Phi=p^{k-1}}=  \dim \cD_{\cris}^{\Phi=p^{k-1}} (\epsilon_p^{1-k})= 1.
\end{equation} 

Hence, one has (after a twist by $\epsilon^{k-2}$)

\begin{equation}\label{multM'2}
\dim \cD^{+}_{\cris}(\mathcal{M}'(\pi_\alpha)^{ss}(k-2))^{\Phi=p}= 1.
\end{equation}

Since the set $\Sigma$ of classical points of $\cE_{\Delta}$ of cohomological weights and old at $p$ (i.e having a crystalline representation) of $\cE_\Delta$ are very Zariski dense (see Cor.\ref{veryzariskidense1}), it follows from Lemma \ref{restzarden} that $\Sigma \cap \mathcal{U}- (\Sigma \cap Z)$ is Zariski dense in $\mathcal{U}$, and hence we obtain a refined family $$(G_{\Q_p} \to \mathrm{Aut}_{\mathcal{U}}(\mathcal{M}'), \Sigma \cap \mathcal{U}- (\Sigma \cap Z), \{\kappa_i\},U_0/U_1 \in \cO(\cE_\Delta)^{\times})$$ of generic rank equal to $3$ over $K$.  Note also that condition (*) of \cite[Definition 4.2.7]{bb} is satisfied since we have a torsion free morphism $\kappa:\cE_\Delta \to \cW$; the condition (v) of \cite[Definition 4.2.7]{bb} is satisfied by Lemma \ref{vzd}, Lemma \ref{zdop} and Corollary \ref{veryzariskidense1} (so $ \Sigma \cap \mathcal{U}- (\Sigma \cap Z)$ accumulate to $\pi_\alpha$).

Since $\mathcal{M}' \otimes_A \cT=M'$, it follows from \cite[Theorem   3.4.1]{bb} that $$\dim \cD_{\cris}^{+}(M'/\gm M'(k-2))^{\Phi=p}=1.$$
 Then

$$\dim \cD^{+}_{\cris}(M'/\gm M')^{\Phi=p^{k-1}}=1.$$
 Finally, by \cite[Theorem   1.5.6]{bb} any $S/\gm S$-extension $V$ of $\epsilon_p^{1-k}$ by $\epsilon_p^{2-k}$ (i.e occurring in the image of the morphism \eqref{ext1cyc}) is a quotient of $M/\gm M \oplus \epsilon_p^{2-k}$ by an $S$-submodule $\mathcal{Q}$ (see (iii) of \eqref{surj1}).

 However, by the regularity assumption at $p$ the non-trivial unramified character $\psi_{\mid G_{\Q_p}}$ does not occur in $V_{\mid G_{\Q_p}} \in \mathrm{Ext}^{1}_{G_{\Q_p}}(\epsilon_p^{1-k}, \epsilon_p^{2-k})$, which implies that $V_{\mid G_{\Q_p}}$ is a quotient of $M'/\gm M' \oplus \epsilon_p^{2-k}$. Thus we obtain a surjection of $G_{\Q_p}$-modules \begin{equation}\label{surj1}
 M'/\gm M' \oplus \epsilon_p^{2-k} \overset{\pi'}{\twoheadrightarrow}  V_{\mid G_{\Q_p}},\end{equation}
 with kernel isomorphic to a quotient of the $G_{\Q_p}$-module $\mathcal{Q}$.
 
Since  the semi simplification of $M_3/\gm M_3$ is isomorphic to the representation $$\rho_f ^{n_1} \oplus (\epsilon_p^{2-k})^{n_2} \oplus \epsilon_p^{1-k}$$ by  \eqref{multfree1of1}, the regularity assumption at $p$ on  $\varrho$ when $k=2$ (i.e.$\alpha \ne 1$), and the fact that the $\mathrm{S}_p$-simple subquotients of $\mathcal{Q}$ do not equal $\epsilon_p^{1-k}$ by Lemma \ref{S_p} (they occur in $\{\epsilon_p^{2-k}, \psi, \psi^{-1}\epsilon_p^{3-2k}\}$), 
one has  $$\cD_{\cris}^{\Phi=p^{k-1}}(\ker(\pi'))=0.$$

Thus the surjective morphism \eqref{surj1} yields the following injection $$\cD_{\cris}^{\Phi=p^{k-1}}(M'/\gm M') \hookrightarrow \cD_{\cris}^{\Phi=p^{k-1}}(V),$$ and implies that $\cD_{\cris}^{\Phi=p^{k-1}}(V)\ne 0$ (since $\dim \cD_{\cris}^{\Phi=p^{k-1}}(M'/\gm M')=1$).

On the other hand, by applying the left exact functor $\cD_{\cris}^{\Phi=p^{k-1}}(.)$ to the exact sequence $$0 \to \epsilon_p^{2-k} \to V \to \epsilon_p^{1-k} \to 0,$$ and using the fact that $\cD_{\cris}^{\Phi=p^{k-1}}(\epsilon_p^{2-k})=0$ and $\dim \cD_{\cris}^{\Phi=p^{k-1}}(\epsilon_p^{1-k})=1$, we obtain that  $\dim \cD_{\cris}^{\Phi=p^{k-1}}(V)=1$ (since it is non-zero by the above discussion). Hence the characteristic polynomial of $\Phi$ has two roots $\{p^{k-2},p^{k-1}\}$ yielding that $\dim \cD_{\cris}(V)=2$ and that $V$ is crystalline, so $V \otimes \epsilon_p^{k-2}$ is also crystalline at $p$.

It remains to proof that the image of the map $$\Hom(\cT_{1,3}/\cT'_{1,3},L) \hookrightarrow  \rH^{1}_f(G_\Q^{Np},\epsilon_p)$$
consists of extensions which are unramified outside $p$. Let $\ell$ denote a prime number dividing $N$ (so prime to $p$), note that any $G_{\Q_\ell}$-extension of $\epsilon_p^{-1}$ by $\mathbbm{1}$ is trivial or its restriction to the inertia has the following form $$ \begin{pmatrix} 
1  & 1     \\ 
       0 & 1 
\end{pmatrix},$$
and hence the monodromy operator of the Weil-Deligne representation attached to its $2$-dimensional $G_{\Q_\ell}$-representation is of rank $1$ (i.e a Steinberg type).

We know that the rank over $L$ of the monodromy operator attached to the Weil-Deligne representation corresponding to $(\rho_f)_{\mid G_{\Q_\ell}}$ is one (since we assumed that $\rho_f$ is a twisted Steinberg at every prime $\ell \mid N$).

Recall that the $G_{\Q}$-coherent sheaf $\mathcal{M}$ is locally free of rank $4$ on the admissible open $\mathcal{U}-Z=\mathcal{U}'$ and it admits a Weil-Deligne representation $(r_\mathcal{U} ,N_\mathcal{U})$ by \cite[Lemma.7.8.14]{bb} at $\ell$ (for which $N_\mathcal{U} \in \End_A(\mathcal{M}$). Since the rank of the monodromy operator of the Weil-Deligne representation attached to the specializations of $(r_\mathcal{U} , N_\mathcal{U})$ at classical points of non-endoscopic, non-CAP points $\mathcal{U}'$ is at most $1$ by Theorem \ref{monodromy},   \cite[Prop.7.8.19(ii)]{bb} implies that the generic rank over $K$ of the monodromy operator of the Weil-Deligne representation attached $(r_\mathcal{U} ,N_\mathcal{U})$ is also $1$ (since it is non-trivial at $\pi_\alpha$). Therefore, the generic rank of the monodromy $N_K=N_\mathcal{U} \otimes K$ operator of the Weil-Deligne representation attached to $(\rho_K)_{\mid G_{\Q_\ell}}$ is one.

Let $\mathrm{S}_\ell$ be the image of $\cT[G_{\Q_\ell}]$ inside $S$. Thanks to Proposition \ref{locl}, one can apply  \cite[Lemma.8.2.11]{bb} \footnote{The assumption that $\pi_{f,\ell} \simeq \mathrm{St} \otimes \xi$ is crucial to prove the vanishing of $\rH^1(G_{\Q_\ell},\rho_f(k-2))$.} to $\mathcal{P}=\{\epsilon_p^{1-k}, \epsilon_p^{2-k}\}$), and we obtain that there exists idempotents $(\tilde{e_1},\tilde{e_2},\tilde{e_3})$ of $S$ lifting the idempotents attached respectively to $\epsilon_p^{2-k},\epsilon_p^{1-k},\rho_f$ and such that $\tilde{e}=\tilde{e_1} + \tilde{e_2}$ is in the center of $\mathrm{S}_\ell$ (see \cite[Lemma.8.2.12]{bb}), and hence $\mathrm{S}_\ell$ is block diagonal of type $(2,2)$ in $S$. Thus,

 $$\mathrm{S}_\ell/\gm \mathrm{S}_\ell=\begin{pmatrix} 
*  &    *& *  \\ 
       * & *  & * \\
       0  & 0 &  \rho_f
\end{pmatrix}\text{, and } \mathrm{S}_\ell=\begin{pmatrix} 
* & *   & *  \\ 
     *   & *  & * \\
       0  & 0 &  M_{2,2}(\cT) 
\end{pmatrix}.$$

By \cite[Lemma.7.8.14]{bb} one can see $N_K$ as element of $\mathrm{S}_\ell$. By \eqref{multfree1of1}(iii)  it is enough to prove that $\tilde{e}N_K \in \tilde{e}\mathrm{S}_\ell$ is trivial for showing that the image of $\Hom(\cT_{1,3}/\cT'_{1,3},L) \hookrightarrow  \rH^{1}(G_\Q^{Np},\epsilon_p)$ gives rise to classes unramified at $\ell$.

As an element of $\mathrm{S}_\ell$ we know that
$N_K=\begin{pmatrix} 
*  & *   & * &* \\   * & *   & *&*  \\ 
      0 & 0 & *  & * \\
      0 &0  & * &  *
\end{pmatrix}$ and it is of rank $1$ as discussed before.  By \cite[Prop.7.8.8]{bb} applied to $(1-\tilde{e})\mathrm{S}_\ell ( 1-\tilde{e})$ we further know that  the rank of $(1-\tilde{e})N_K$ is one, using that $\rho_f$ is a twisted Steinberg at $\ell$ (the rank of the monodromy operator of $WD_\ell(\rho_f)$ is one) and the surjection $(1-\tilde{e}).S/\gm S.(1-\tilde{e}) \twoheadrightarrow \rho_f$.
Hence, $\tilde{e_i}N_K=0$ for $i\in \{1,2\}$, which  yields that $\tilde{e}N_K=0$.

\end{proof}

The proof of Theorem \ref{lastentry} yields the following corollary.
\begin{cor}\label{familyM'}
There exists a representation $\rho_{\mathcal{M}'}:G_{\Q_p} \to \mathrm{Aut}_{\mathcal{U}}(\mathcal{M}')$, where $\mathcal{M}'$ is a torsion-free coherent sheaf on an admissible open affinoid $\mathcal{U}=\Spm A \subset \cE_\Delta $ containing $\pi_\alpha$ and $\rho_{\mathcal{M}'}$ is of generic rank $3$ over the total ring of fractions $K$ of $\mathcal{U}$ such that:
\begin{enumerate}
\item  There exists a very Zariski dense set $\Sigma' \subset \mathcal{U}$ such that the specialization of the representation $\rho_{\mathcal{M}'}$ at any point $z$ of $\Sigma'$ gives rise to a crystalline $G_{\Q_p}$-representation $\rho_z'$ of dimension $3$, with Hodge-Tate-Sen weights given by $(\kappa_2-2,\kappa_1-1,\kappa_1+\kappa_2-3)$.

\item The smallest Hodge-Tate weight of $\rho_z'$ is $\kappa_2(z)-2$ and $U_{1}/U_{0} \in \cO(\cE_\Delta)^{\times}$ interpolates the crystalline period  of the smallest Hodge-Tate weight. In other words, one has $$\dim \cD_{\cris}(\rho'_z)^{\Phi=U_{1}/U_{0}(z)p^{\kappa_2(z)-2}}=1.$$

\item Let $M':=\mathcal{M}'\otimes_A \cT$, then for any cofinite ideal $\mathcal{J}$ of $\cT$ one has that $$l(\cD^{+}_{\cris}(M'/\mathcal{J} M'\otimes (\epsilon_p^{\kappa_2-2}))^{\Phi=U_{1}/U_{0}}) = l(\cT/\mathcal{J}).$$

\item The Sen operator of  $\cD_{\mathrm{sen}}(M'/\mathcal{J}M')$ is annihilated by the polynomial $$(T -(\kappa_2-2))(T-(\kappa_1-1))(T-(\kappa_1+\kappa_2-3)).$$

\end{enumerate}

\end{cor}

\begin{proof}

i) and ii)  follows directly from the proof of Theorem \ref{lastentry} and \cite[Theorem   1.5.6]{bb}. Thus, it remains to show iii) and iv), which follows immediately from similar arguments to those already used to prove of Theorem   \ref{lastentry}, \cite[Theorem   3.4.1]{bb} and \cite[Lemma.4.3.3]{bb}(i).

\end{proof}

\section{Crystallinity of the $S$-extensions occurring in $\rH^1(\Q,\rho_f(k-2))$}\label{ordrho1}

By \eqref{1.5.5extn} we have a natural injection \begin{equation}\label{extrho1}\Hom(\cT_{3,2}/\cT'_{3,2},L) \hookrightarrow \mathrm{Ext}^1_{G_\Q^{Np}}(\rho_f, \epsilon_p^{1-k}) \simeq \rH^{1}(G_\Q^{Np},\rho_f(k-2)).
\end{equation}

 Now we have to determine the exact image of the injective morphism $\eqref{extrho1}$. As in section \ref{cnsht} we apply the results of \cite[Theorem   1.5.6]{bb} and \cite[Lemma.4.3.9]{bb} for the left ideal $M_2=S.E_2$ of $S$ given by the second column of the GMA matrix $S$: 

\begin{enumerate} \label{ext3}

\item  There exists an exact sequence of $S$-left modules \begin{equation}\label{multfree1of2} 0 \rightarrow E' \rightarrow M_2/\gm M_2\rightarrow \rho_f \to 0 \end{equation}

\item Any simple $S$-subquotients of $E'$ is not isomorphic to $\rho_f$ and they occur in the set $\{\epsilon_p^{1-k}, \epsilon_p^{2-k} \}$.

\item The image of the morphism \eqref{extrho1} consists of extensions occurring as quotient of the  $S/\gm S$-module $M_2/\gm M_2 \oplus \epsilon_p^{1-k}$ by an $S$-submodule $\mathcal{Q}'$ whose $S$-simple subquotients occur in the set $\{\epsilon_p^{1-k}, \epsilon_p^{2-k} \}$.

\end{enumerate}
Since $\rho_K$ is absolutely irreducible and $M_2$ is a finite type torsion free $\cT$-module we again have $M_2.K= K^4$.

\begin{thm}\label{Bcrys2} Assume  $({ \bf Reg})$. Let $\cT'_{3,2}$  be the $\cT$-module $\cT_{3,1}\cT_{1,2} \subset \cT_{3,2}$, then:

\begin{enumerate}
\item There exists an injective homomorphism of $L$-modules { \small \begin{equation} \Hom(\cT_{3,2}/\cT_{3,2}',L) \hookrightarrow \ker(\rH^{1}(\Q,\rho_f(k-2)) \to \rH^1(\Q_p,\rho_f/\rho_f^{I_p}(k-2))\oplus_{\ell \nmid p} \rH^1(I_{\ell},\rho_f(k-2))). 
\end{equation}}

\item Assume that $k \geq 3$, then \begin{equation}\label{imcrys2} \Hom(\cT_{3,2}/\cT_{3,2}',L) \hookrightarrow \rH^{1}_{f,\unr}(\Q,\rho_f(k-2)).
\end{equation}
\end{enumerate}
\end{thm}

\begin{proof}\

i) By \eqref{multfree1of2}  we have a surjective morphism of $S$-modules $\pi: M_2/\gm M_2 \twoheadrightarrow \rho_f$ whose  kernel does not contain $\rho_f$ and whose semi-simplification contains only $G_\Q$-representations lying in the set $\{\epsilon_p^{1-k}, \epsilon_p^{2-k} \}$. Moreover, our assumptions yield that the irreducible constituents of the semi-simplification of $\varrho_{\mid G_{\Q_p}}$ are without multiplicity, hence $M_2^{I_p}:=\{x \in M_2, \forall g\in I_p, g.x=x \text{ and } \Frob_p.x=U_0.x\}$ is not contained in $\gm M_2$.  Let $V  \in \mathrm{Ext}^1_{G_\Q}(\rho_f,\epsilon_p^{1-k})=\rH^{1}(G_\Q^{Np},\rho_f(k-2)) $ be in the image of \eqref{imcrys2}. By \eqref{multfree1of2} (iii) we have an  exact sequence of left $S$-modules $$0 \to \mathcal{Q}' \to M_2/\gm M_2 \oplus \epsilon_p^{1-k} \to V \to 0.$$ 

Similar to Lemma \ref{S_p} we can show that $\mathcal{Q}'$ has no $L[G_{\Q_p}]$-simple subquotients equal to $\psi$ or $\psi^{-1} \epsilon_p^{2k-3}$. This shows that the image of $M_2^{I_p}$ in  $V$ is non-zero. It follows that $$V^{I_p} \ne 0.$$ Moreover, since $\Frob_p$ acts on $M_2^{I_p}$ by $U_{0}$, the action of $\Frob_p$ on $V^{I_p}$ is given by $\psi$. If the realization of $V$ is given by $\tilde{\rho}=\begin{pmatrix} 
\epsilon_p^{1-k}  & *  \\ 
       0 & \rho_f 
\end{pmatrix}$ then the restriction of $\tilde{\rho}$  to $G_{\Q_p}$ is given by $\begin{pmatrix} 
\epsilon_p^{1-k}  & 0   & c  \\ 
       0 & \psi  & * \\
       0  & 0 &  \psi^{-1} \epsilon_p^{3-2k}
\end{pmatrix}$ since $V^{I_p} \ne 0$. Finally, it remains to show that the extensions $V$ are unramified at every prime $\ell \mid N$, and this fact follows immediately from Proposition \cite[Lemme 4.1.3]{urban}.

ii) The fact that $k \geq 3$ implies that $3-2k < 1-k$ and hence $\tilde\rho$  is ordinary in the sense of Definition \ref{ordinaryrepsentation} and then semi-stable (hence de Rham) at $p$ by Theorem   \ref{nekthm}. Therefore the extension $V$ gives a cohomology class in $\rH^1_g(G_\Q^{Np}, \rho_f(k-2))$ which is isomorphic to $\rH^1_f(G_\Q^{Np}, \rho_f(k-2))$ by \cite[Lemme 4.1.3]{urban}.

\end{proof}

\begin{rem} \label{Jannsen}
For $k=2$ ordinarity/crystallinity of the extension would require us to prove additionally that  $\tilde{\rho}/\tilde{\rho}^{I_p} \cong  \begin{pmatrix} 
\epsilon^{-1} &  c  \\ 
   
       0  &  \psi^{-1} \epsilon^{-1}
\end{pmatrix}$ is a trivial extension. This would follow, e.g. if one could prove that the generator of $\rH^1(G_\Q^{Np}, \rho_f)$ (which is conjectured to be $1$-dimensional by Jannsen) has no line fixed by inertia at $p$). See section \ref{Selmergroupvanish} below for an alternative approach in this case.
\end{rem}

Similarly to Corollary \ref{antiinvB'}, \cite[Prop.1.8.6]{bb} yields immediately the following corollary.

\begin{cor}\label{antiinvB''}Assume $({ \bf Reg})$. Then the image of the natural injective morphism of $L$-modules $$\Hom(\cT_{2,1}/\cT_{2,3}\cT_{3,1},L) \hookrightarrow \rH^{1}(\Q,\rho_f(k-2))$$  is isomorphic to the image of \eqref{extrho1} (which is described in Theorem   \ref{Bcrys2}).

\end{cor}

\subsection{On the vanishing of the Greenberg's Selmer group attached to $f_\alpha$}\label{Selmergroupvanish}

Assume in this subsection that $k=2$ and let $$\mathrm{Sel}_{\Q,f_\alpha}= \ker(\rH^{1}(\Q_{Np},\rho_f) \to \rH^1(\Q_p,\rho_f/\rho_f^{I_p}) \underset{\ell \nmid p}{\oplus} \rH^1(I_\ell,\rho_f))$$ be the Greenberg-type Selmer group we used in Theorem \ref{Bcrys2}(i) attached to the ordinary elliptic cuspform $f_\alpha$. In the literature, Greenberg's Selmer group is often defined using the representation $\rho_f^{\vee}(-1)$ (arithmetic Frobenius convention). The $p$-adic representation $\rho_f^{\vee}$ corresponds to the Tate module $T_p(A_f)$ of the abelian variety $A_f$, and $\rho_f$ is the Galois representation obtained from the $p$-adic \'etale cohomology of $A_f$. We remark also that for $k=2$ the condition at $p$ is weaker than the ``usual'' condition for the ordinary representation $\rho_f$ (which would require the class to be split at $p$). Our condition of having an $I_p$-fixed quotient for the extension $\begin{pmatrix} \rho_f & *\\0& 1\end{pmatrix}$ (or dually an $I_p$-fixed line for $\begin{pmatrix} \epsilon^{-1}_p & *\\0&  \rho_f\end{pmatrix}$) is the one that would normally be required for $\rho_f(1) \cong \rho_f^{\vee}$.

Note that $\rho_f$ is not critical in the sense of Deligne. We use Iwasawa theory for the cyclotomic $\Z_p$-extension to bound $\mathrm{Sel}_{\Q,f_\alpha}$:
It follows from Kato \cite{Kato} that the the Pontryagin dual of the Selmer group $\mathrm{Sel}_{\Q_\infty,f_\alpha}$ is a torsion $\Lambda$-module with characteristic ideal $g(T) \in \Lambda$.  Furthermore, according to the Iwasawa main conjecture (Kato's bound, see e.g. \cite[Theorem   3.25]{SkinnerUrban14}),  $g(T) \mid L_p(f,\omega^{-1},\cdot)$. Hence $\dim \mathrm{Sel}_{\Q,f_\alpha}=0$ when $L_p(f_\alpha,\omega^{-1}_p,T=p) \ne 0$ (see \cite[Prop.2.10]{bk} and \cite[Theorem   2.11]{bk} for more details). Moreover, it follows from the control theorem for the $\Lambda$-adic Greenberg's Selmer group $\mathrm{Sel}_{\Q_\infty,f_\alpha}$ (see also \cite{Ochiai}) that  $g(T=p) \ne 0$ is a necessary condition for the vanishing of $\mathrm{Sel}_{\Q,f_\alpha}$.

\section{Schematic reducibility locus of the pseudo-character $\mathrm{Ps}_{\cT}$ on $\Spec \cT$ and applications to the Bloch-Kato conjecture}\label{redloc}

Recall that we view $S$ as the generalized matrix attached to the pseudo-character $$\mathrm{Ps}_\cT:G_\Q \to \cT$$ with respect to a set of idempotents compatible with the anti-involution $\tau$ and have $$S=\begin{pmatrix} 
\cT & M_{1,2}(\cT_{1,2}) & \cT_{1,3} \\ 
 M_{2,1}(\cT_{2,1}) & M_2(\cT) & M_{2,1}(\cT_{2,3}) \\
 \cT_{3,1} & M_{1,2}(\cT_{3,2}) & \cT 
\end{pmatrix},$$ where $\cT_{i,j}$ are fractional ideals of $K$ that satisfy $\cT_{i,j} \cT_{j,k} \subset \cT_{i,k}$ and $\cT_{i,j} \cT_{j,i} \subset \gm$.

In this section we will compute the total reducibility  ideal $\mathcal{I}^{\tot}\subset \cT$ (see Definition \ref{redideal}). 
By Proposition \ref{redideal2} it is given by
\begin{equation}\label{totred}
\mathcal{I}^{\tot}=\cT_{3,1} \cT_{1,3} +\cT_{2,3} \cT_{3,2} + \cT_{1,2} \cT_{2,1}.
\end{equation}
The following lemma follows directly from the anti-involution $\tau :S \to S$ and the fact that $\mathrm{Ps}_\cT$ is invariant under the action of $\tau$.

\begin{lemma}\label{antiinvre} One always has:

$\cT_{2,3} \cT_{3,2} = \cT_{1,2} \cT_{2,1}$.

\end{lemma} 

\begin{proof}
This is proved exactly as in Lemma \cite[8.2.16]{bb} using the anti-involution $\tau$.

\end{proof}

Hence, the above lemma implies that \begin{equation}\label{totred2}
\mathcal{I}^{\tot}= \cT_{3,1} \cT_{1,3} +\cT_{1,2} \cT_{2,1}. \end{equation}

\begin{lemma}\label{form} Assume $({ \bf Reg})$ and $({ \bf St})$. Then one has

$$\mathcal{I}^{\tot}=  \cT_{1,2} \cT_{2,1}= \cT_{2,3} \cT_{3,2}.$$

\end{lemma} 

\begin{proof} 

We first show that $ \cT_{1,3}=\cT'_{1,3}=\cT_{1,2} \cT_{2,3}$. By Theorem \ref{lastentry} we have an injective map $$\Hom(\cT_{1,3}/\cT'_{1,3},L) \hookrightarrow \rH^{1}_{f,\unr}(\Q,\epsilon_p).$$

Note that the Kummer map provides an isomorphism $$\rH^1_{f,\unr}(\Q,\epsilon_p) \simeq \Z^{\times} \otimes L.$$ Hence $\rH^{1}_{f,\unr}(\Q,\epsilon_p)$ is trivial, and then $ \cT_{1,3}/\cT'_{1,3}=0$ by Nakayama's lemma ($\cT_{1,3}$ is of finite type over $\cT$ since $S$ is). Thus, we have \begin{equation}\label{1,3}
 \cT_{1,3}=\cT_{1,2} \cT_{2,3}. \end{equation}
 It is easy to see that \begin{equation}
\begin{split}
\mathcal{I}^{\tot}=& \cT_{3,1} \cT_{1,3} +\cT_{1,2} \cT_{2,1} \\ =&\cT_{1,2} \cT_{2,3}\cT_{3,1} +\cT_{1,2} \cT_{2,1} \\ =&\cT_{1,2} \cT_{2,1} \text{, since $\cT_{2,3}\cT_{3,1} \subset \cT_{2,1} $}. 
\end{split}
\end{equation}

\end{proof}

\begin{cor}\label{description1,2}
One has
$$\cT'_{1,2}=\mathcal{I}^{\tot}.\cT_{1,2}.$$

\end{cor}

\begin{proof}
Since $\cT_{1,3}=\cT_{1,2} \cT_{2,3}$ by relation \eqref{1,3} we get $\cT'_{1,2}=\cT_{1,3} \cT_{3,2}=\cT_{1,2} \cT_{2,3}\cT_{3,2}$. On the other hand, we have by Lemma \ref{antiinvre} that $\cT_{2,3} \cT_{3,2} = \cT_{1,2} \cT_{2,1}$, and we have also by Lemma \ref{form} $\mathcal{I}^{\tot}=  \cT_{1,2} \cT_{2,1}= \cT_{2,3} \cT_{3,2}$. Thus $\cT'_{1,2}=\cT_{1,3} \cT_{3,2}=\cT_{1,2} \cT_{2,3}\cT_{3,2}=\mathcal{I}^{\tot}\cT_{1,2}.$

\end{proof}

\subsection{Application to Bloch-Kato conjecture}

Since we have assumed that the sign $\epsilon_f$ of $L(f,s)$ is $-1$, the functional equation $$L(f,s)= - L(f,2k-2-s)$$
yields that $L(f,s)$ vanishes at the central value $k-1$. The Selmer group $\rH^{1}_{f,\unr}(\Q,\rho_f(k-1))$ classifies the extensions with everywhere good reduction and one can think of the Bloch-Kato conjecture as a generalization of the Birch and Swinnerton-Dyer conjecture for the motive $M_f$ corresponding to $f$ of weight $2k-2 \geq 2$. On has the following application related to the Bloch-Kato conjecture:

\begin{cor}\label{nontrivsel}Assume that $k \geq 2$, $({ \bf Reg})$ and $({ \bf St})$. Then there exists an injection \begin{equation} \Hom(\cT_{1,2}/\gm.\cT_{1,2},L)  \hookrightarrow \rH^{1}_{f,\unr}(\Q,\rho_f(k-1)), \end{equation}
and $\dim \cT_{1,2}/\gm.\cT_{1,2} \geq 1$.

\end{cor}

\begin{proof}   The following injection follow from Theorem \ref{Bcrys} and Corollary \ref{description1,2}: 

\begin{equation}\label{imcrys3} \Hom(\cT_{1,2}/\gm.\cT_{1,2},L) \simeq \Hom(\cT_{1,2}/\mathcal{I}^{\tot}.\cT_{1,2},L) \hookrightarrow \rH^{1}_{f,\unr}(\Q,\rho_f(k-1)) \end{equation}

Moreover, $ \cT_{1,2}/\gm.\cT_{1,2} \ne \{0\}$ since $\rho_K:G_\Q \rightarrow \GL_4(K)$ is absolutely irreducible (so $\mathcal{I}^{\tot}=  \cT_{1,2} \cT_{2,1} \ne (0)$).

\end{proof}

\begin{prop}\label{bornjansen}

 $\dim \rH^{1}(G_\Q^{Np},\epsilon_p^{-1})=1$.

\end{prop}

\begin{proof} This follows from \cite[Prop.2.2]{Maj}.

\end{proof}

\begin{prop}\label{nontrivsel2}Assume that $k \geq 3$ and $({ \bf Reg})$. The image of the natural injection \begin{equation} \Hom(\cT_{1,2}/ \cT_{1,2}',L)  \hookrightarrow \rH^{1}_{f,\unr}(\Q,\rho_f(k-1)), \end{equation}
is non-trivial (i.e, $ \cT_{1,2}/ \cT_{1,2}' \ne \{ 0 \}$).
\end{prop}

\begin{proof}
Since $\dim \rH^{1}_{f,\unr}(\Q,\rho_f(k-2))=0$ (by \cite{Kato}), Theorem   \ref{Bcrys2} and Corollary \ref{antiinvB''} yield that $\Hom(\cT_{3,2}/\cT_{3,2}',L)=\{0 \}$, and hence Nakayama's lemma implies that  $\cT_{3,2}=\cT'_{3,2}=\cT_{3,1}\cT_{1,2}  $. We will proceed by contradiction. Assume  that $ \cT_{1,2}/ \cT_{1,2}' = \{ 0 \}$ and so $\cT_{1,2}=\cT_{1,2}'=\cT_{1,3} \cT_{3,2}$. Thus, $\cT_{1,2}=\cT_{1,3} \cT_{3,1}\cT_{1,2}$. Moreover, $\cT_{1,3} \cT_{3,1} \subset \gm \subset \cT$ by \cite[Theorem   1.4.4]{bb} and hence   $\cT_{1,2}=\gm \cT_{1,2}$. Therefore, Nakayama's lemma yields that $\cT_{1,2}=0$. The involution introduced in \S.\ref{symmetry} implies by \cite[Lemma.1.8.5]{bb} that $\cT_{1,2}=\cT_{2,3}$ and $\cT_{2,1}=\cT_{3,2}$, and so also $\cT_{2,3}=0$. Furthermore, Theorem \ref{Bcrys2} yields that $\cT_{2,1}=\cT'_{2,1}=\cT_{2,3}\cT_{3,1}$. Finally, this yields that $\cT_{3,2}=\cT_{2,1}=\cT_{2,3}\cT_{3,1}=0$ and also $\cT_{1,2}=\cT_{2,3}=0$, and it follows that the generic representation $\rho_K:G_\Q \to S^{\times} \subset \GL_4(K)$ attached to $K=Q(\cT)$ is reducible, which contradicts Theorem   \ref{irrecomplevel1}.

\end{proof}

Assume now that $\bar{\rho}_f$ is absolutely irreducible. Let $\mathbb{I}$ be the finite flat integral extension of the Iwasawa algebra $\Z_p   \lsem   T  \rsem $ generated by the coefficients of a Hida family $\mathcal{F}$ specializing to $f_\alpha$ ($\mathcal{F}$ is unique up to Galois conjugacy) and let $\rho_\mathcal{F}: G_\Q \to \GL_2(\mathbb{I})$ be the $p$-adic Galois representation attached to $\mathcal{F}$. Let $\chi_{\univ}:G_\Q \to \Z_p   \lsem   T  \rsem ^{\times}$ be the universal character given by the composition of the $p$-adic cyclotomic character $\epsilon_p: G_\Q \to 1+p^\nu \Z_p$  with the tautological character $ 1+p^\nu \Z_p\to \Z_p  \lsem  1+p^\nu \Z_p  \rsem ^{\times}\simeq \Z_p   \lsem   T  \rsem ^{\times}$, where $\nu=1$ if $p \geq 3$ and $\nu=2$ if $p=2$. It follows from the work of Nekovar \cite[Prop.4.2.3]{Nek} that the $\mathbb{I}$-adic Selmer group $\rH^1_{f,\unr}(\Q,\rho_{\mathcal{F}} \otimes \chi_{\univ}^{-1/2})$ is of finite type over the Iwasawa algebra $\Z_p   \lsem   T  \rsem $, and so over $\mathbb{I}$ since $\mathbb{I}$ is finite flat over $\Z_p   \lsem   T  \rsem $.

\begin{cor}\label{nonItor}Assume that $\bar{\rho}_f$ is absolutely irreducible, then the generic rank of the $\mathbb{I}$-adic Selmer group $\rH^1_{f,\unr}(\Q,\rho_{\mathcal{F}} \otimes \chi_{\univ}^{-1/2})$ is at least one (i.e. $\rH^1_{f,\unr}(\Q,\rho_{\mathcal{F}} \otimes \chi_{\univ}^{-1/2})$ has a non-torsion class over $\mathbb{I}$).
\end{cor} 

\begin{proof}
It follows from Corollary \ref{nontrivsel2} that  the $\mathbb{I}$-adic Selmer group $\rH^1_{f,\unr}(\Q,\rho_{\mathcal{F}} \otimes \chi_{\univ}^{-1/2})$ specializes at infinitely many classical points of $\Spm \mathbb{I}[1/p]$ to a non-trivial Selmer group. Hence, the generic rank of $\rH^1_{f,\unr}(\Q,\rho_{\mathcal{F}} \otimes \chi_{\univ}^{-1/2})$ over $\mathbb{I}$ is non-zero.

\end{proof}

\subsection{Bounding the number of generators of $\mathcal{I}^{\tot}$}

\begin{thm}\label{princ1} 
Assume $({ \bf Reg})$, $({ \bf St})$ and $\dim \rH^{1}_{f,\unr}(\Q,\rho_f(k-1))=1$. 
\begin{enumerate}

\item There exists idempotents $\{e_1',e_2',e_3'\}$ of $S$ lifting the idempotents of $\varrho$ attached to $\{\epsilon_p^{2-k},\rho_f,\epsilon_p^{1-k}\}$ such that  $S$ has the following form $$S=\begin{pmatrix} 
\cT & M_{1,2}(\cT) & \cT \\ 
 M_{2,1}(\mathcal{I}^{\tot}) & M_2(\cT) & M_{2,1}(\cT) \\
 \cT_{3,1} & M_{1,2}(\mathcal{I}^{\tot}) & \cT 
\end{pmatrix},$$ where $\cT_{3,1}=\mathcal{J} \subset \mathcal{I}^{\tot}$ is an ideal.

\item Assume $k \geq 3$. Then $\mathcal{I}^{\tot}=  \mathcal{J}=\cT_{3,1}$ and $\mathcal{I}^{\tot}=\cT.g + (\mathcal{I}^{\tot})^2$ for an element $g$ in $\mathcal{I}^{\tot}$, which yields that the reducibility ideal $\mathcal{I}^{\tot}$ is principal and generated by $g$.

\item Assume that $k=2$ and $\dim \mathrm{Sel}_{\Q,f_\alpha}=0$, then $\mathcal{I}^{\tot}=  \mathcal{J}=\cT_{3,1}$ and $\mathcal{I}^{\tot}=\cT.g + (\mathcal{I}^{\tot})^2$ for an element $g$ in $\mathcal{I}^{\tot}$, and the reducibility ideal  $\mathcal{I}^{\tot}$ is principal and generated by $g$.
\end{enumerate}

\end{thm}

\begin{rem}
Using results about $\Lambda$-adic Selmer groups we exhibit many examples where the Selmer group $\rH^{1}_{f,\unr}(\Q,\rho_f(k-1))$ is $1$-dimensional (see Appendix \S.\ref{example}).

\end{rem}

\begin{proof}
i) By Theorem \ref{Bcrys} and Corollary \ref{description1,2}, we have the following: 

\begin{equation}\label{imcrys3} \Hom(\cT_{1,2}/\mathcal{I}^{\tot}.\cT_{1,2},L) \hookrightarrow \rH^{1}_{f,\unr}(\Q,\rho_f(k-1)). \end{equation}
Moreover, since $\mathcal{I}^{\tot} \subset \gm$, we have an injection $$\Hom(\cT_{1,2}/\gm.\cT_{1,2},L) \hookrightarrow \Hom(\cT_{1,2}/\mathcal{I}^{\tot}.\cT_{1,2},L).$$
By the assumption on the dimension of $\rH^{1}_{f,\unr}(\Q,\rho_f(k-1))$ we get \[\dim \Hom(\cT_{1,2}/\gm.\cT_{1,2},L) \leq 1.\]
On the other hand, the fact that $\rho_K:G_\Q \rightarrow \GL_4(K)$ is irreducible implies that $\mathcal{I}^{\tot}=  \cT_{1,2} \cT_{2,1}\ne 0$ and hence \[\dim \Hom(\cT_{1,2}/\gm.\cT_{1,2},L)=1.\]

Thus, Nakayama's lemma implies that the $\cT$-modules $\cT_{1,2}$ is a monogenic $\cT$-module.

Since $\cT_{1,2}$ is a fractional ideal of $K$ and each component  $\rho_{K_i}$ of $\rho_K$ is absolutely irreducible, the annihilator of the generator of $\cT_{1,2}$ over $\cT$ is trivial. Hence, $\cT_{1,2}$ is a  free rank one $\cT$-module. Moreover,  the symmetry under the anti-involution implies that $\cT_{1,2} \simeq \cT_{2,3}$ and hence $\cT_{2,3}$ is also a free $\cT$-module of rank one.

Let $\alpha \in K$ (resp. $\beta \in K$) be a generator of $\cT_{1,2}$ (resp. of $\cT_{2,3}$) as  $\cT$-module. A direct computation shows that one can choose $e_1'=\alpha.e_1, e_2'=e_2, e_3'=\beta^{-1}.e_3$ as a suitable basis of idempotents.

Moreover, we recall that we have an injection by Theorem \ref{lastentry}  $$\Hom(\cT_{1,3}/\cT_{1,2} \cT_{2,3},L)=\Hom(\cT_{1,3}/\cT,L) \hookrightarrow \rH^1_{f,\unr}(\Q,\epsilon_p)=\{0\}.$$ 
Hence, Nakayama's lemma implies that $\cT_{1,3}=\cT$. Now, we can conclude from the fact that $\cT_{1,2}\cT_{2,1}=\cT_{2,3}\cT_{3,2}=\mathcal{I}^{\tot}$ that $\cT_{2,1}=\cT_{3,2}=\mathcal{I}^{\tot}$.

ii) By \eqref{1.5.5extn} applied with $(i,j)=(3,2)$ and $(3,1)$, respectively, applying Theorem \ref{Bcrys2} and Cor.\ref{antiinvB''}  for $(i,j)=(3,2)$ and using that $\cT_{3,2}=\mathcal{I}^{\tot}$, $\cT'_{3,2}=\cT_{3,1} \cT_{1,2}=\mathcal{J}$, $\cT_{3,1}=\mathcal{J}$, and $\cT'_{3,1}=\cT_{3,2}\cT_{2,1}=(\mathcal{I}^{\tot})^2$ we get injective morphisms 

 \begin{equation}\label{imcrys4}\begin{split} \Hom(\mathcal{I}^{\tot}/\mathcal{J},L) \hookrightarrow \rH^{1}_{f,\unr}(\Q,\rho_f(k-2)) \\ \Hom(\mathcal{J}/(\mathcal{I}^{\tot})^2,L) \hookrightarrow \rH^{1}(G_\Q^{Np},\epsilon_p^{-1}). \end{split}
 \end{equation}

One has  $\dim \rH^{1}_{f,\unr}(\Q,\rho_f(k-2))=0$ (by a deep result of Kato \cite{Kato}), hence Nakayama's lemma applied to $\mathcal{I}^{\tot}/\mathcal{J}$ yields that $\mathcal{I}^{\tot}=\mathcal{J}$. Moreover, the ideal $\mathcal{I}^{\tot}$ is non-zero since $\rho_K$ is irreducible. Thus, the fact that $\dim \rH^{1}(G_\Q^{Np},\epsilon_p^{-1}) = 1$  (by Proposition \ref{bornjansen}) yields that $\mathcal{I}^{\tot}=\cT.g + (\mathcal{I}^{\tot})^2$ and $g$ is a generator of the ideal $\mathcal{I}^{\tot}$.

iii) The assertion follows from similar arguments to those already used to prove i), ii) and the fact that $\Hom(\mathcal{I}^{\tot}/\mathcal{J},L) \hookrightarrow \mathrm{Sel}_{\Q,f_\alpha}$ by Theorem   \ref{Bcrys2} and Cor.\ref{antiinvB''}. 
\end{proof}

One has the following general bound of the number of generators of $\mathcal{I}^{\tot}$: 
\begin{cor}\label{bound} Let $s:=\dim \rH^{1}_{f,\unr}(\Q,\rho_f(k-1))$. Assume $({ \bf Reg})$, $({ \bf St})$, and that $\dim \mathrm{Sel}_{\Q,f_\alpha}=0$ if $k=2$. Then $\mathcal{I}^{\tot}$ is generated by at most $s^2$ elements.

\end{cor}

\begin{proof}
It follows from \eqref{Bcrys} and Corollary \ref{description1,2} that $\cT_{1,2}$ (resp. $\cT_{2,3}$) is generated by at most $s$ elements. Moreover, it follows from Theorem \ref{Bcrys2} and Cor.\ref{antiinvB''} that $\cT_{2,1}=\cT_{2,3}\cT_{3,1}$ and $\cT_{3,1}=\cT_{3,2}\cT_{2,1}+ g.\cT$. Thus, $\cT_{2,1}=(\cT_{3,2}\cT_{2,1}+ g.\cT) \cT_{2,3}= \mathcal{I}^{\tot} \cT_{2,1} + g.\cT_{2,3}$. Hence, $\cT_{2,1}$ is generated also by at most $s$ elements and then $\mathcal{I}^{\tot}=\cT_{1,2}\cT_{2,1}$ is generated by at most $s^2$ elements.

\end{proof}

\section{Smoothness of $\cE_{\Delta}$ at $\pi_\alpha$ }\label{modpar}

The goal of this section is to prove under the assumptions $({ \bf Reg})$, $({ \bf St})$ and $({ \bf BK})$ that $A:=\cT/\mathcal{I}^{\tot}$ is a regular ring of dimension one (so it is a discrete valuation ring) and deduce that $\cT$ is a regular ring of dimension two when $\dim \rH^{1}_{f,\unr}(\Q,\rho_f(k-1))=1$.

\subsection{Ordinary deformations of $\rho_f$}

Recall that $\rho_f: G^{Np}_{\Q} \rightarrow \GL_2(L)$ is the irreducible odd $p$-adic representation attached to $f$ and $f_\alpha$ is  the $p$-ordinary $p$-stabilisation of $f \in S_{2k-2}(\Gamma_0(N),L)$. We assume until the end of this subsection that $({ \bf St})$ holds for $f$ and consider the following  deformation problem attached to $\rho_f$: for $B$ any local $L$-Artinian algebra with maximal ideal $\gm_B$ and residue field $B/\gm_B=L$, we define $\cD(B)$ as the set of strict equivalence classes of representations $\rho_B : G_\Q^{Np} \rightarrow \GL_2(B)$ lifting $\rho_f$
(that is  $\rho_B\mod{\gm_B} \simeq \rho_f$) and which are  ordinary at $p$
in the sense that:
\begin{eqnarray} \rho_{B|G_{\Q_p}} \simeq \left( \begin{matrix} \psi_{1,B} &  * \\ 0 & \psi_{2,B} \end{matrix} \right),  \end{eqnarray} 
where $\psi_{1,B} : G_{\Q_p} \rightarrow B^\times$ is an unramified character, and such that they are minimally ramified at every $\ell \mid N$ (i.e $\rho_B^{I_\ell}$
is  free of rank one over $B$ for any $\ell \mid N$). Let $\cD'$ be the subfunctor of $\cD$ of deformation with constant determinant (so equal to $\det \rho_f=\epsilon_p^{3-2k}$). 

Since $\rho_f$ is absolutely irreducible, we know from \cite{Ma95} that $\cD$ is prorepresentable by a complete local noetherian ring $\cR^{\ord}$ together with its universal ordinary deformation $\rho^{\ord}:G_\Q \to \GL_2(\cR^{\ord})$. We know also from \cite{Ma95} that the tangent space $\cD(L[\epsilon])$ (resp. $\cD'(L[\epsilon])$) is isomorphic to a subspace $\rH^1_{\ord}(\Q, \ad  \rho_f)$ (resp. $\rH^1_{\ord}(\Q, \ad^0  \rho_f)$) of $\rH^1(\Q, \ad \rho_f)$ (resp. $\rH^1(\Q, \ad^0 \rho_f)$). 

The determinant of $\rho^{\ord}$ is a deformation of $\det \rho_f$, and endows $\cR^{\ord}$ with a structure of $\varLambda_1$-algebra, where $\varLambda_1$ is the completed local ring of $\Z_p \lsem  \Z_p^{\times} \rsem=\underset{ 0\leq l \leq p-1}{\oplus} \Z_p \lsem X \rsem$ \footnote{One has $\Spec \Z_p \lsem \Z_p^{\times} \rsem( \C_p)=\Hom_{\mathrm{cont}}(\Z_p^{\times},\C_p^{\times})$.  } at the height one prime ideal corresponding to $\det \rho_f$ ($\varLambda_1= \Q_p \lsem X \rsem$).

\begin{prop}\label{R=Ttg}Assume $({ \bf St})$. Then the $L$-vector space  $\rH^1_{\ord}(\Q, \ad^0  \rho_f)$ is contained in  $\rH^1_{f,\unr}(\Q, \ad^0  \rho_f)$.
\end{prop}
\begin{proof}Let $\rho_\epsilon= (1+ \epsilon \rho_1) \rho_f$ be a representative of a strict equivalence class in $\cD'(L[\epsilon])$, so $\rho_1$ is a 1-cocycle with respect to the adjoint action $\ad^0 \rho_f$, and its cohomology class lies in $\rH^1_{\ord}(\Q, \ad^0  \rho_f)$. We will show first that $[\rho_1] \in \rH^1_{g}(\Q, \ad^0  \rho_f):=\ker(\rH^1(\Q, \ad^0  \rho_f) \to \rH^1(\Q_p, \ad^0  \rho_f \otimes \mathrm{B}_{\mathrm{dR}}) )$, where $\mathrm{B}_{\mathrm{dR}}$ is Fontaine's de Rham period ring.  So it is enough so show that $\rho_\epsilon$ is ordinary in the sense of Definition \ref{ordinaryrepsentation} (so semi-stable by Theorem   \ref{nekthm}) as a $4$-dimensional $p$-adic representation over $L$ (since the $L$-representation $\rho_\epsilon$ corresponds to an element of $\mathrm{Ext}^1_{G_\Q}(L,\ad^0 \rho_f)=\rH^1(\Q, \ad \rho^0_f)$). Since $\rho_{\epsilon}$ is ordinary at $p$, one can find a basis $(e_1,e_2)$ of $L[\epsilon]^2$ such that the line $L[\epsilon] \cdot e_1$ is $G_{\Q_p}$-stable and invariant under the action of the inertia $I_p$. Let $v_1:=e_1, v_2=\epsilon \cdot e_1, v_3:=e_2; v_4=\epsilon \cdot e_2$ be a basis of $L[\epsilon]^2$ as $L$-vector space. 

By ordinarity at $p$ one has \[ (\rho_\epsilon)_{\mid G_{\Q_p}}=((1+ \epsilon \rho_1)\rho_f)_{\mid G_{\Q_p}}= \begin{pmatrix} 
\psi' & *\\ 
0 & \psi''   
\end{pmatrix},\] where $\psi'$ is unramified and $\psi''_{\mid I_p}=\epsilon_p^{3-2k}$ (since $\det \rho_\epsilon=\epsilon_p^{3-2k}$). Thus, the realization of $\rho_\epsilon$ by matrices in the basis $(v_1, v_2, v_3, v_4)$ of $L[\epsilon]^2$ has the following form on $I_p$:
\[ (\rho_\epsilon)_{\mid I_{_p}}= \begin{pmatrix} 
1 & 0 & \ast & \ast \\ 
0 & 1 & \ast & \ast\\ 
0 & 0 &\epsilon_p^{3-2k}& 0 \\ 0 & 0 &0 & \epsilon_p^{3-2k} \\ 
\end{pmatrix}. \]
Hence $\rho_\epsilon$ is ordinary and it follows from Theorem \ref{nekthm} that $\rho_\epsilon$ is semi-stable, and then $[\rho_1] \in \rH^1_{g}(\Q, \ad^0  \rho_f)$. It remains to show that 
$\rho_1 \in \rH^1_{f,\unr}(\Q, \ad^0  \rho_f)$. Since $\rho_\epsilon$ is minimally ramified at any $\ell \mid N$ (i.e $(\rho_\epsilon)_{\mid I_\ell}=\begin{pmatrix} 
1 & *\\ 
0 & 1  
\end{pmatrix} $), the restriction of the cohomology class of $\rho_1$ to $I_\ell$ is trivial (because $\Hom(I_\ell,L)$ is $1$-dimensional). To be able to conclude, we need to show that $\rH^1_{g}(\Q_p, \ad^0  \rho_f)= \rH^1_{f}(\Q_p, \ad^0  \rho_f)$. This follows immediately from the fact that { \small \[\dim \rH^1_{g}(\Q_p, \ad^0  \rho_f)=\dim \rH^1_{f}(\Q_p, \ad^0  \rho_f) + \dim \cD_{\cris}((\ad^0 \rho_f)_{\mid G_{\Q_p}} (1))^{\phi=1}\] }and \[ \dim \cD_{\cris}((\ad^0 \rho_f)_{\mid G_{\Q_p}} (1))^{\phi=1}=0.\]

\end{proof}

\subsection{Regularity of $\cT/\mathcal{I}^{\rm tot}$}\label{mainthmsection}
Recall that $(\kappa_{1},\kappa_{2}) \subset (\cO(\cW))^2$ are the universal weights interpolating $k_1,k_2$ (they are the derivative at $1$ of $\epsilon_p^{\kappa_1},\epsilon_p^{\kappa_2}$). Hence one can see  $\kappa_i$ as global section in $\cO(\cE_\Delta)$ via the weight map $\kappa : \cE_\Delta \to \cW$. Recall also that $\epsilon_p^{ -\kappa_1}$ and $\epsilon_p^{-\kappa_2}$ specialize at $\underline{k}=(k_1,k_2)$ to the characters $\epsilon_p^{-k_1},\epsilon_p^{-k_2}$, respectively.

Let $A$ be the local quotient ring $\cT/\mathcal{I}^{\tot}$ of dimension $\leq 2$. Note that $A$ is Henselian, since $\cT$ is Henselian (the local ring of a rigid analytic space for the rigid topology is always Henselian).

Let $\mathrm{Ps}_{A}:G_\Q \to \cO(\cE_\Delta) \to A$ be the natural pseudo-character of dimension $4$. Moreover, $\mathrm{Ps}_{A}=\Psi_1 + \Psi_2 + \Tr_{A}$ such that $\Tr_{A}:G_\Q \to A$ is a pseudo-character lifting the pseudo-character $\mathrm{Tr}(\rho_f)$ and $\{\Psi_i\}_{i=1,2}:G_\Q \to A^\times$ are characters lifting respectively $\epsilon_p^{2-k}$ and $\epsilon_p^{1-k}$. Moreover,  since $\rho_f$ is absolutely irreducible,   $\Tr_{A}:G_\Q \to A$ is the trace of a deformation $\rho_{A}:G_\Q \to \GL_2(A)$ of $\rho_f$. The deformation $\det \rho_A$ of $\det \rho_f$ yields a natural local morphism of $\bar{\Q}_p$-algebras $\varLambda_1 \to A$ (see \cite[\S.6]{D-B}).

\begin{thm}\label{Freeness}\ Assume  $({ \bf Reg})$. For any cofinite ideal $\mathcal{J} \subset \cT$ containing $\mathcal{I}^{\tot}$ the $\cT/\mathcal{J}$-module  $\cD_{\cris}^{+}(M'/\mathcal{J} M'\otimes (\epsilon_p^{\kappa_2-2}))^{\Phi=U_{1}/U_{0}})$ is  free of rank one  and $\kappa_1 - \kappa_2  \in \mathcal{I}^{\tot}$. Moreover, one has that $\Psi_2 \equiv \epsilon_p^{1-\kappa_2} \mod \mathcal{I}^{\tot}$.

\end{thm}

\begin{proof}

Recall that in the proof of Theorem \ref{lastentry} and Corollary \ref{familyM'}, we have constructed a family of $p$-adic representations $\rho_{\mathcal{M}'}:G_{\Q_p} \to \GL_{\mathcal{U}}(\mathcal{M}')$ over an affinoid $\mathcal{U}:=\Spm B \subset \cE_\Delta$ containing $\pi_\alpha$, and such that $\mathcal{M}'=\mathcal{M}/\mathcal{M}^{I_p}$ is a torsion-free quotient of $\mathcal{M}$ of generic rank $3$ (the generic rank of $\mathcal{M}$ over $\mathcal{U}$ is 4). By  \cite[Theorem   1.5.6]{bb} we have surjections  $$M=\mathcal{M} \otimes_B \cT \twoheadrightarrow M/\mathcal{J}M \twoheadrightarrow \Psi_2 \mod \mathcal{J},$$  such that any semi-simple $S$-subquotient of the $S$-module $\ker(M/\mathcal{J}M \to \Psi_2 \mod \mathcal{J})$ occurs in $\{\rho_f,\epsilon_p^{2-k}\}$ (any $S$-simple module is necessarily an $S/\mathfrak{m} S$-module). 

On the other hand, since $\mathcal{M}/\mathcal{M}^{I_p}=\mathcal{M}'$, the surjection $M/\mathcal{J}M \twoheadrightarrow \Psi_2 \mod \mathcal{J}$ must factor through \begin{equation}\label{surjconstantweights} M'/\mathcal{J}M' \twoheadrightarrow \Psi_2 \mod \mathcal{J} \end{equation} for $M'=\mathcal{M}' \otimes_B \cT$. 

We recall from  Corollary \ref{familyM'} that $$l(\cD_{\cris}^{+}(M'/\mathcal{J} M'\otimes (\epsilon_p^{\kappa_2-2}))^{\Phi=U_{1}/U_{0}})=l(\cT/\mathcal{J}).$$ 

On the other hand, it follows  from the fact that the semi-simple subquotients of $$\ker(M'/\mathcal{J}M' \to \Psi_2 \mod \mathcal{J})$$ occur in $\{\epsilon_p^{2-k}, \psi, \psi^{-1}\epsilon_p^{3-2k}\}$ that $$\cD_{\cris}(\ker(M'/\mathcal{J}M' \to \Psi_2 \mod \mathcal{J})\otimes \epsilon_p^{\kappa_2-2})^{\Phi=U_1/U_0}=\{0\}.$$ Therefore, $l(\cD_{\cris}^{+}(\Psi_2 \otimes (\epsilon_p^{\kappa_2-2})\mod \mathcal{J})^{\Phi=U_{1}/U_{0}})=l(\cT/\mathcal{J})$. Thus, \cite[Lemma.3.3.9]{bb} yields that \begin{equation}\label{cryspara}\cD_{\cris}^{+}(\Psi_2 \otimes (\epsilon_p^{\kappa_2-2})\mod \mathcal{J})^{\Phi=U_{1}/U_{0}} \text{ is a free rank one $\cT/\mathcal{J}$-module,} \end{equation} and then $$
\cD_{\cris}^{+}(M'/\mathcal{J} M'\otimes (\epsilon_p^{\kappa_2-2}))^{\Phi=U_{1}/U_{0}}) \text{ is a free rank one $\cT/\mathcal{J}$-module}.$$

At the same time \cite[Prop.2.5.4]{bb} (i.e the ``constant weight lemma'') yields that $\Psi_2 \otimes \epsilon_p^{\kappa_2-2}$ has a constant weight given by $1$ (i.e by the weight of $\epsilon_p^{-1} \equiv \Psi_2 \otimes \epsilon_p^{\kappa_2-2} \mod \gm$) and that the Sen operator acts on $\Psi_2$ by multiplication by $1$. It follows from  Corollary \ref{familyM'} that $1$ is a root of \[ (T )(T-(\kappa_1-\kappa_2+1))(T-(\kappa_1-1)).\]

Thus, $ (\kappa_2-\kappa_1) (2-\kappa_1) \equiv 0 \mod \mathcal{J}$. Assume now that $k \geq 3$, then  $(2-\kappa_1)$ is invertible in (the local ring) $\cT/\mathcal{J}$. Therefore, $\kappa_2-\kappa_1 \in \mathcal{J}$. Assume now that $k=2$, we can  consider by Theorem   \ref{Hodge-Tate}(iii) the sub-representation $\rho_{\mathcal{M}''}:G_{\Q_p} \to \mathrm{Aut}_{\mathcal{U}}(\mathcal{M}'')$ generically of dimension $2$ of $\rho_{\mathcal{M}'}$ ($\mathcal{M}''$ is a torsion-free $\cO_\cU$-module) and it specialises at every point $z$ of the Zariski dense set $\Sigma' \subset \mathcal{U}$ to a crystalline $G_{\Q_p}$-representation $\rho_z''$ of dimension $2$ such that:
\begin{enumerate}
\item The Hodge-Tate weights of $\rho_z''$ are  $\{ \kappa_1(z)-1,\kappa_2(z)-2 \}$ with $\kappa_2(z)-2 < \kappa_1(z) -1 $.
\item $\dim \cD_{\cris}(\rho''_z)^{\Phi=U_{1}/U_{0}(z)p^{p^{\kappa_2(z)-2}}}=1$. 
\end{enumerate}
Let $M'':=\mathcal{M}''\otimes_B  \cT=\mathcal{M}''_{\pi_\alpha}$ be the stalk of $\mathcal{M}''$ at $\pi_\alpha$. Similar arguments to those already used to prove \cite[Lemma.4.3.3]{bb}(i) yield that the Sen operator of  $\cD_{\mathrm{sen}}(M''/\mathcal{J}M'')$ is annihilated by the polynomial $(T -(\kappa_2-2))(T-(\kappa_1-1))$. The assumption $({ \bf Reg})$ and the fact that $G_{\Q_p}$ acts by $U_0 \epsilon_p^{-\kappa_1-\kappa_2+3}$ on $M'/M''$ ($U_0$ means the unramified character with value $U_0$ at $\Frob_p$) yield that the  composition $M''/\mathcal{J}M'' \to  M'/\mathcal{J}M' \twoheadrightarrow \Psi_2 \mod \mathcal{J}$ is surjective. Thus, $T (T-(\kappa_1-\kappa_2+1))$ annihilates the Sen operator of $\Psi_2 \otimes \epsilon_p^{\kappa_2-2}$, and since $1$ is a root of that polynomial, we deduce immediately that $\kappa_2-\kappa_1 \in \mathcal{J}$.

Meanwhile, the fact that $\cD_{\cris}^{+}(\Psi_2 \otimes (\epsilon_p^{\kappa_2-2})\mod \mathcal{J})^{\Phi=U_{1}/U_{0}}$ is a free rank one $\cT/\mathcal{J}$-module yields that  the character $\Psi_2 \otimes \epsilon_p^{\kappa_2-2} \mod \cT/\mathcal{J}$ is a crystalline $L[G_{\Q_p}]$-representation with Hodge-Tate weight $1$. Thus, $\Psi_2 \otimes (\epsilon_p^{\kappa_2-2}) \otimes \epsilon_p\mod \mathcal{J}$ is of Hodge-Tate weight $0 $ and crystalline, therefore unramified. Thus, by class field theory we deduce that $\Psi_2 \otimes (\epsilon_p^{\kappa_2-2}) \otimes \epsilon_p\mod \mathcal{J}$ is the trivial character (since $\Q$ has a unique $\Z_p$-extension). Therefore, $\Psi_2 \mod \mathcal{J}= \epsilon_p^{1-\kappa_2} \mod \mathcal{J}$.

Finally, we conclude that $\kappa_2-\kappa_1 \in \mathcal{I}^{\tot}$ and that $\Psi_2=\epsilon_p^{1-\kappa_2} \mod \mathcal{I}^{\tot}$ since the ideal $\mathcal{I}^{\tot}$ is the intersection of all cofinite length containing it (by Krull's theorem).

\end{proof}

\begin{cor} \label{prop10.2}
Assume $({ \bf Reg})$.  Then  the local ring $A$ is topologically generated by the image of $\Tr(\rho_A)$ over $\varLambda_1$.

\end{cor} 

\begin{proof}

Let $A'$ be the subring of $A$ topologically generated by the image of the trace $ \Tr(\rho_A)$ over $\varLambda_1$.  Since $\kappa_1 - \kappa_2  \in \mathcal{I}^{\tot}$ by Theorem \ref{Freeness}, the polarisation of $\mathrm{Ps}_\cT$ described in section \ref{symmetry} and the fact that the subconstituent of $\mathrm{Ps}_{\cT} \mod \mathcal{I}^{\tot}$ are uniquely determined (see  \cite[Prop.1.5.1]{bb}) yield that $\Tr \rho_A =\Tr (\rho^{\vee}_A \otimes \epsilon_p^{3-2\kappa_1})$, $\Psi_1=\Psi_2^{-1} \otimes \epsilon_p^{3-2\kappa_1}=\epsilon_p^{2-\kappa_1}$ (the last equality follows from Theorem \ref{Freeness}). Thus, $\det \rho_A$ is given by the character $\epsilon_p^{3-2\kappa_1}$. As the determinant can be expressed in terms of the trace the image of $\epsilon_p^{2\kappa_1}$ lies in $A'$, and so by Hensel's lemma the image of the character $\epsilon_p^{\kappa_1}= \epsilon_p^{\kappa_2} \mod \mathcal{I}^{\tot}$ (and therefore also of $\Psi_1$ and $\Psi_2$) lies in $A'$. Finally, as $\mathrm{Ps}_A$ is surjective onto $A$ by construction of $\cE_\Delta$ this establishes the proposition.

\end{proof}

\begin{prop} \label{prop10.3} Assume $({ \bf Reg})$ and $({ \bf St})$. Then the representation $\rho_A$ is $p$-ordinary and minimal.

\end{prop}

\begin{proof}

According to \cite[Theorem   1.5.6]{bb} and \cite[Lemma.4.3.9]{bb}, there exists a $\cT$-module $M \subset K^4$ of generic rank $4$ endowed with a $G_\Q$-continuous action which is generically given by the semi-simple representation $$\rho_{K}:G_\Q \to S^\times \subset \GL_4(K),$$
and  equipped with a surjection $\pi:M/\mathcal{I}^{\tot}M \twoheadrightarrow \rho_A$ such that the $S$-simple subquotients of its kernel are either $\epsilon_p^{1-k}$ or $\epsilon_p^{2-k}$.

 Since $\cT$ is reduced and $\rho_{K}$ is semi-ordinary ($\rho_{K}^{I_p}$ is of dimension one and $\Frob_p$ acts on it by $U_{0}$) and $\alpha \ne 1$ when $k=2$, we again have (as in \S \ref{ordrho1}) that $M^{I_p}$ is not contained in $\gm M$. Since the $S$-simple subquotients of $\ker \pi$ do not contain $\rho_f$ and contain only the representations in  the set $\{\epsilon_p^{1-k},\epsilon_p^{2-k}\}$, the regularity assumption further implies that the image of $M^{I_p}$ under the surjection $\pi':M/\gm M \twoheadrightarrow \rho_f$ is non-zero and hence the image of $M^{I_p}$ under the surjection $\pi:M/\mathcal{I}^{\tot}M \twoheadrightarrow \rho_A$ is non-zero and it is not contained in $\gm A^2$.
 
 Thus, we have an exact sequence of $A[G_{\Q_p}]$-modules:

\begin{equation}\label{exactsplit}
0\to \rho_A^{I_p} \to \rho_A \to \rho_A/\rho_A^{I_p} \to 0.
\end{equation}
Since  $ \rho_A/\rho_A^{I_p} \otimes_A L$ is of rank one Nakayama's lemma implies that $\rho_A/\rho_A^{I_p}$ and  $\rho_A^{I_p}$ are monogenic $A$-modules and generated respectively by $y_1,y_2$. Therefore $y_1,y_2$ generate $A^2$ and they must even form a basis of $A^2$. Hence 
  the exact sequence \eqref{exactsplit} splits as $A$-modules and yields that $\rho_A$ is $p$-ordinary.

We shall now prove that $\rho_A$ is minimally ramified at every $\ell \mid N$. Let $\ell$ be a prime number dividing $N$. From the proof of Theorem \ref{lastentry} we know that there exist idempotents $(\tilde{e_1},\tilde{e_2},\tilde{e_3})$ of $S$ lifting the idempotents attached respectively to $\epsilon_p^{2-k},\epsilon_p^{1-k},\rho_f$ such that $\tilde{e}=\tilde{e_1} + \tilde{e_2}$ is in the center of $\mathrm{S}_\ell= \begin{pmatrix} 
* & *   & *  \\ 
     *   & *  & * \\
       0  & 0 &  M_{2,2}(\cT) 
\end{pmatrix}$,  the image of $\cT[G_{\Q_\ell}]$ inside $S$. We also recall that $N_K$, the monodromy operator corresponding to the Weil-Deligne representation attached to $G_{\Q_\ell} \to \mathrm{S}_\ell^{\times}$, can be viewed as an element of $\mathrm{S}_{\ell}$, has rank 1 by Proposition \ref{monodromy} and satisfies $\tilde{e_3}N_K\tilde{e_3} \neq 0$. For $N:=\tilde{e_3}N_K\tilde{e_3} \in M_2(\cT)$ we know that $N$ is non-trivial modulo $\gm_\cT$  (since the rank of the monodromy operator of $WD_\ell(\rho_f)$ is one) and so the morphism $\rho_A|_{G_{\Q_\ell}}: G_{\Q_\ell} \to \GL_2(\cT) \to \GL_2(A)$ is also minimally ramified.

\end{proof}

\begin{rem}The assumption $({ \bf Reg})$ is crucial to ensure the existence of a line in $M$ fixed by inertia on which $\Frob_p$ acts by $U_0$ (i.e $M^{I_p} \not\subset \gm M$). We have many examples for $\GL_2(\Q)$ \cite[Thm.C]{BD19} where the $2$-dimensional $p$-adic Galois representation $\rho_{\mathfrak{h}}$ attached to the local ring $\mathfrak{h}$ of the eigencurve at $p$-irregular weight one forms is not ordinary, but the representation becomes ordinary when we extend the scalar to the field of fractions of $\mathfrak{h}$ (i.e, $\rho_{\mathfrak{h}}^{I_p} \ne 0$ but $\rho_{\mathfrak{h}}^{I_p} \subset \gm_{\mathfrak{h}}\rho_{\mathfrak{h}}^{I_p} $).

\end{rem}

 \begin{prop} \label{RTBK} Let $\widehat{A}$ denote the completion of $A$ with respect to its maximal ideal and assume $({ \bf St})$, $({ \bf Reg})$ and $({ \bf BK})$. Then the natural morphism $\cR^{\ord} \rightarrow
 \widehat{ A}$  associated to $\rho_A \otimes_A \widehat{A}$ is an isomorphism of discrete valuation rings. Moreover, $\cR^{\ord}$ is \'etale over $\varLambda_1$.
 \end{prop}
 \begin{proof}
According to Prop.\ref{prop10.2}, $\widehat{A}$ is topologically generated by $\Tr \rho_A(G_{\Q})$. It follows that $\cR^{\ord}$ surjects on $\widehat{A}$.  We claim that the Krull dimension of $\widehat{A}$ is at least one, and it is a consequence of the fact that $A$ surjects onto the local ring $\cO_{SK(\mathcal{F}),\alpha}$ at $\pi_\alpha$ of the $1$-dimensional irreducible analytic subspace corresponding to the Saito-Kurokowa family $SK(\mathcal{F})$ specializing to $\pi_\alpha$, where $\mathcal{F}$ is the Hida family specialising to $f_\alpha$ (see \cite[Prop.4.2.5]{urban}).
  
  On the other hand, it follows from Prop.\ref{R=Ttg}  that the relative tangent space of $\varLambda \to \cR^{\ord} $ (i.e $\cD'_{\rho_f}(L[\epsilon])$) is contained in $\rH^1_{f,\unr}(\Q,\ad^0 \rho_f)$, which is trivial under the assumption $({ \bf BK})$. Thus, the tangent space of $\cR^{\ord}$ is at most $1$-dimensional and $\cR^{\ord}$ is unramified  over $\varLambda_1$. However, the existence of the surjection $\cR^{\ord} \twoheadrightarrow
  \widehat{A}$ yields that the tangent space of $\cR^{\ord}$ is necessarily $1$-dimensional and that $\cR^{\ord} \twoheadrightarrow
  \widehat{A}$ is an isomorphism of $1$-dimensional regular rings. Since $\det \rho_A=\epsilon_p^{3-2\kappa_1}$, $\varLambda_1$ injects in $\cR^{\ord} \simeq A$ and $\cR^{\ord}$ is necessarily flat over $\varLambda_1$. Thus, $\cR^{\ord}$ is \'etale  over $\varLambda_1$ because it is unramified and flat.
  
  \end{proof}

\begin{thm} \label{smooththm} Assume $({ \bf BK})$, $({ \bf Reg})$, $({ \bf St})$, $\dim \rH^{1}_{f,\unr}(\Q,\rho_f(k-1))=1$ and when $k=2$ also assume that $L_p(f_\alpha,\omega^{-1}_p,T=p) \ne 0$. Then the local ring $\cT$ is regular of dimension $2$, i.e. $\cE_\Delta$ is smooth at $\pi_\alpha$. Moreover, the reducibility ideal of the pseudo-character $\mathrm{Ps}_\cT$ corresponds to the principal Weil divisor (closed subset of dimension one)  of $\Spec \cT$ corresponding to the Saito-Kurokawa family $SK(\mathcal{F})$ specializing to $\pi_\alpha$, where $\mathcal{F}$ is the Hida family passing through $f_\alpha$.

\end{thm}

\begin{proof}

 We have to show that the tangent space of $\cT$ is of dimension $2$. Since the Krull dimension is always less or equal to the dimension of the tangent space, we have to show that the maximal ideal $\gm$ of $\cT$ has at most two generators. Note that $\mathcal{I}^{\tot}=(g)$ (see Theorem   \ref{princ1} and \S.\ref{Selmergroupvanish}) and $A=\cT/(g)$ is regular of dimension $1$. Hence $\gm$ has at most two generators. Thus $\cT$ is regular. The rest of the assertion follows from the fact that $\mathcal{I}^{\tot}=(g)$ and $\cO_{SK(\mathcal{F}),x}=A=\cT/(g)$ (since $A$ is a discrete valuation ring).

\end{proof}

One has the following general bound of the Zariski tangent space of $\pi_\alpha \in \cE_{\Delta}$.
\begin{cor}

Assume $({ \bf BK})$, $({ \bf Reg})$, $({ \bf St})$ and when $k=2$ also assume that $L_p(f_\alpha,\omega^{-1}_p,T=p) \ne 0$. Then we have

 \[2 \leq \dim \mathfrak{t}_{\pi_\alpha} \leq 1 +(\dim \rH^{1}_{f,\unr}(\Q,\rho_f(k-1)))^2 \] and \[\dim \mathfrak{t}_{\pi_\alpha}^0 \leq (\dim \rH^{1}_{f,\unr}(\Q,\rho_f(k-1)))^2.\]

\end{cor}
\begin{proof}
The assertion follows immediately from Corollary \ref{bound} (i.e $\mathcal{I}^{\tot}$ is generated by at most $s^2$ elements) and from Theorem \ref{smooththm} (i.e $A=\cT/\mathcal{I}^{\tot}$ is \'etale over $\varLambda_1 \simeq \varLambda/(\kappa_1-\kappa_2)$).

\end{proof}

\section{Smoothness failure of $\cE_N$ at $\pi_\alpha$ when $N$ is square free and not prime}\

We prove in this subsection that our main results fail when we change the tame level to $\Gamma(N)$. In  this subsection we can remove the assumption on the global root number $\epsilon_f$ being $-1$ as there exists a Saito-Kurokawa lift of level  $\Gamma(N)$ for either sign (see \cite{Schmidt07}).

Coleman, Gouvea and Jochnowitz proved in \cite{CGJ} that the $p$-adic modular form for $\GL_2(\Q)$ $$G_2(q)=\frac{\zeta(-1)}{2} + \sum_{n=1}^{\infty} \sigma(n) q^n \text{, where }\sigma(n)=\sum_{d \mid n} d$$ is not overconvergent, however the $p$-ordinary $p$-stabilization $E_2^{ord_p}(q)=G_2(q)-p.G_2(q^p)$ of $G_2(q)$ is classical, hence the critical $p$-stabilization $E_2^{\crit}=G_2(q)-G_2(q^p)$ of $G_2(q)$ is not overconvergent. On the other hand, any ordinary $\ell$-stabilization $E_2^{\crit,ord_\ell}$ of $E_2^{\crit}$ is an overconvergent modular form of weight two and level $\Gamma_0(\ell p)$. Note that $a_{\ell'}(E_2^{\crit,ord_\ell})=1 + \ell'$ where $\ell' \nmid \ell.p$, and $a_\ell(E_2^{\crit,ord_\ell})=1$, $a_p(E_2^{\crit,ord_\ell})=p$.

$E_2^{\crit,ord_\ell}$ is a cuspidal overconvergent form of tame level $\Gamma_0(\ell)$ since each constant term of its $q$-expansion is trivial at each cusp of the multiplicative ordinary locus of the the rigid curve attached to the semi-stable modular curve $X_1(\Gamma_1(4\ell) \cap \Gamma_0(p))/\Z_p$ (these cusps are in the $\Gamma_0(p)$-orbit of the standard cusp $\infty$, see \cite[\S.3.1,\S.3.2]{BDPozzi}).

Let $\cC_N$ be the reduced eigencurve of tame level $N$ constructed using the Hecke operators $T_\ell$ for  $\ell \nmid N p$ and $U_p$ (we omit the Hecke operators $U_\ell $ for $\ell \mid N$). Recall that  there exists a flat and locally finite morphism $w: \cC_N \rightarrow \mathcal{V}$, called the weight map, where $\mathcal{V}$ is the weight space ($\mathcal{V}(\C_p)=\Hom_{\mathrm{cont}}(\Z_p^{\times},\C_p^{\times})$).

\begin{prop}\label{localfamilyGL2}\ 
 Let $\mathcal{Y}$ be an irreducible component of the $p$-adic Eigencurve $\cC_{N}$ of tame level $N$ specializing to a point $y$ corresponding to the system of Hecke eigenvalues of $E_2^{\crit}$. Denote by $\rho_{\cU}:G_\Q \to \GL_2(K_{\mathcal{U}})$  the Galois representation attached to $\mathcal{\cU}$, where $K_{\mathcal{\cU}}$ is the field of fractions of some connected affinoid subdomain $\cU$ of $\mathcal{Y}$ containing $y$, then $\rho_{\mathcal{\cU}}$  is Steinberg at least one prime $\ell \mid N$ (hence $N \ne 1$).

\end{prop}

\begin{proof}

Let $A:=\cO_{\mathcal{Y},y}$ be the local ring of $\mathcal{Y} \subset \cC_N$ at point $y$. One has a universal pseudo-character carried by $\cC_N$ \begin{equation}\label{pseudo-charactC_N} G_\Q \to \cO(\cC_N) \end{equation} sending $\Frob_q$ to the Hecke operator $T_q$, where $q \nmid N p$ is a prime number. The localization of the pseudo-character \eqref{pseudo-charactC_N} at $A$ gives rise to a pseudo-character  $$\mathrm{Ps}_A: G_\Q \to A$$ of dimension $2$ specializing to $\epsilon_p^{-1} \oplus \mathbbm{1} $ modulo the maximal ideal of $A$. Moreover, $\mathrm{Ps}_A$ is the trace of a $2$-dimensional irreducible Galois representation $\rho_A:G_\Q \to \GL_2(Q(A))$ (since $\mathcal{Y}$ corresponds to a cuspidal Coleman family). Hence, we obtain from $\rho_A$ a non-trivial cohomology class $c_y$ in $\rH^1(\Q,\epsilon_p)$ (see \cite[\S.1.5]{bb}). The cohomology class $c_y$ corresponds to an extension $V=\Q_p^2$ of $\epsilon_p^{-1}$ by $\mathbbm{1}$ unramified outisde $Np$. It is known that for any classical point $y'$ in $\cC_N$, the semi-simple $p$-adic Galois representation $\rho_{y'}:G_\Q \to \GL(V_{y'})$ of dimension $2$ attached to the modular form corresponding to $y'$ (i.e.$\Tr \rho_{y'}$ is the specialization of \eqref{pseudo-charactC_N} at $y'$) has a crystalline periods equal to $U_p(y')$ (see \cite{kisin}) and it corresponds to its smaller Hodge-Tate weight which is zero (i.e.$\cD_{\cris}(V_{y'})^{\Phi=U_p(y')} \ne 0$), hence by using the analytic continuation of the crystalline periods $U_p$ on the Eigencurve $\cC_N$ (see \cite[Theorem   4.3.6]{bb}), one has $\cD_{\cris}(V)^{\Phi=U_p(y)}=\cD_{\cris}(V)^{\Phi=p} \ne 0$ (note that $U_p(y)=U_p(E_2^{\crit})=p$). Thus, $c_y$ is crystalline extension of of $\epsilon_p^{-1}$ by $\mathbbm{1}$, and it belongs to  $$\rH^1_f(G_\Q^{Np},\epsilon_p)=\ker(\rH^1 (G_\Q^{Np},\epsilon_p) \to \rH^1(\Q_p,\epsilon_p \otimes \mathrm{B}_{\cris})).$$  

Let us proceed now by contradiction. Assume that $\rho_{\cU}$  is not Steinberg at any $\ell \mid N$ (i.e the rank of the monodromy operator of the Weil-Deligne representation attached to $\rho_{\cU}$ by \cite[Lemma 7.8.14]{bb} at any $\ell$  is zero), hence $\rho_{\cU}$ is principal series or supercuspidal, which implies that for any $\ell \mid N$, the image of the inertia group $I_\ell$ by $\rho_{\cU}$ is finite (we also have a natural inclusion $K_\cU \subset Q( \cO_{\mathcal{Y},y})$), and then semi-simple and reducible. Moreover, $\epsilon_p^{-1} \oplus \mathbbm{1}$ is trival on $I_\ell$ when $\ell \nmid p$, hence $\rho_{\cU}$ is unramified outside $p$.

Thus, the extension $c_y$ is not Steinberg at any $\ell \mid N$ (hence unramified outside $p$) and it belongs necessarily to $\rH^1_{f,\unr}(\Q,\epsilon_p)$ which is trivial. Finally, the cohomology class $c_y$ is trivial, contradicting the fact that $\rho_{\mathcal{Y}}$ is absolutely irreducible.

\end{proof}

\begin{rem}\
\begin{enumerate}
\item The Atkin-Lehner eigenvalue of the classical specializations of $\mathcal{Y}_{\ell}$ at $\ell$ is constant and equal to $-1$. 
\item Let $\ell \mid N$ be a prime for which $f$ is special, then the Hida family $\mathcal{F}$ specializing to $f_\alpha$ is {special} at  $\ell$ and the Atkin-Lehner eigenvalue of the classical specializations of $\mathcal{F}$ at  $\ell$ is constant.

\item According to \cite{Maj}, the weight map $w:\cC_\ell \to \mathcal{V}$ is \'etale at $E^{\crit,\ord_\ell}_2$, and since $w$ is locally finite, one can shrink any affinoid neighborhood of $E^{\crit,\ord_\ell}_2$ to ensure that it will be \'etale over $\mathcal{V}$ (see Proposition \ref{geneta}).
\end{enumerate}

\end{rem}

We can use this to construct endoscopic components containing $\pi_{\alpha}$. We first note the following result about classical Yoshida lifts.

\begin{prop}[\cite{A-S} Prop.3.1]\label{liftyoshida}
Let $f_1 \in \mathrm{S}_{k_1}(N_1)$, $f_2 \in \mathrm{S}_{k_2}(N_2)$ be newforms of squarefree level with  even integers $k_1 \geq k_2 \geq 2$ and $M:={\rm gcd}(N_1, N_2)>1$. Assume that the Atkin-Lehner eigenvalues of $f_1$ and $f_2$ for $\ell \mid M$ coincide. Put $N={\rm lcm}(N_1, N_2)$. Then there exists a non-zero holomorphic Yoshida lift of level $\Gamma(N)$ and weight $((k_1+k_2)/2, (k_1-k_2+4)/2)$ with corresponding Galois representation $\rho_{f_1} \oplus \rho_{f_2}(\frac{k_1-k_2}{2})$. For $p \nmid N$ there exists a $p$-stabilisation of this lift (of Iwahori level at $p$) with $U_0$-eigenvalue $\alpha_1$ and $U_1$-eigenvalue\footnote{For the different normalisation \eqref{U1norm} of the $U_1$ operator on the eigenvariety this corresponds to the constant eigenvalue $\alpha_1\alpha_2$.} $\alpha_1\alpha_2 p^{\frac{k_1-k_2-2}{2}}$, where $\alpha_i$ are roots of the Hecke polynomial of $f_i$ at $p$ for $i=1,2$.
\end{prop} 
\begin{proof}
For the existence of the lift of level  $\Gamma(N)$ see \cite{A-S} Prop.3.1. For the $p$-stabilisation of the principal unramified series see \cite{MM} \S7.1.1, but we use the normalization of \cite{urban} \S2.4.16.
\end{proof}

\begin{thm}Let $\ell \mid N$ be a prime number for which $f$ is Steinberg, $\mathcal{U}^{1}$ be an affinoid subdomain of the $p$-adic eigencurve $w_2: \cC_\ell \to \cV $ of tame level $\ell$ containing $E^{\crit,\ord_\ell}_2$, corresponding to a Coleman family $G=\sum_{n=1}^{\infty} a(n,G) q^n$, and such that it is {\it \'etale} over the weight space $\mathcal{V}$. Let $\mathcal{U}^0$ be an affinoid subdomain of the ordinary locus $\cC_N^{\ord}$ of the $p$-adic eigencurve $\cC_N$ of tame level $N$ containing $f_\alpha$ and corresponding to the Hida family $\mathcal{F}=\sum_{n=1}^{\infty} a(n,\mathcal{F}) q^n$, and such that it is {\it \'etale} \footnote{According to Hida's control theorem, the weight map $w_1:\cC_N \to \mathcal{V}$ is \'etale at $f_\alpha$, and since $w$ is locally finite, one can shrink any affinoid neighborhood of $f_\alpha$ to ensure that it will be \'etale over $\mathcal{V}$.} over the weight space $\mathcal{V}$. 

There exists a Zariski closed immersion  $\lambda_{\mathrm{Yo}}: \cU^0 \times_{\Q_p} \cU^1 \hookrightarrow \cE_{N}$ with image denoted by $\mathrm{Yo}(\mathcal{F},G)$ and such that the following diagram commutes $$\xymatrix{\cU^0 \times_{\Q_p} \cU^1 \ar[d]^{w_1 \times w_2} \ar[r]^{\lambda_{\mathrm{Yo}}} & \cE_{N} \ar[d]^\kappa\\
\mathcal{V} \times \cV  \ar[r]^{\lambda_{\kappa}} & \cW}$$
 where $\lambda_{\kappa}(2k_1,2k_2)=(k_1+k_2, k_1-k_2+2)$ and the morphism $\lambda_{\mathrm{Yo}}$ corresponds to the morphism $$\lambda_{\mathrm{Yo}}^*:\cO(\cE_{N}) \to \cO(\mathcal{U}^0) \widehat{\otimes}_{\Q_p} \cO(\mathcal{U}^1)$$ defined by 
 
 { \small 
 \[\lambda_{\mathrm{Yo}}^*(P_\ell(X))=(X^2-a(\ell,\mathcal{F}) X+ \ell^{-3}\kappa_1\kappa_2(\ell))(X^2-\kappa_2(\ell)\ell^{-2}a(\ell,G) X+\ell^{-3}\kappa_2(\ell).\kappa_1(\ell)),\] } for any $ \ell \nmid Np$,  
where $P_\ell(X) \in \cO(\cE_{N})[X]$ is the Hecke-Andrianov polynomial at $\ell \nmid Np$, and $\lambda_{\mathrm{Yo}}^*(U_0)=a(p,\mathcal{F})$, and $\lambda_{\mathrm{Yo}}^*(U_1)=a(p,\mathcal{F}) \times a(p,G)$. \

\end{thm}

\begin{proof}
One can choose the affinoids $\cU^0 \subset \cC_N$ and $\cU^1 \subset \cC_\ell$ \'etale over the weight space and small  enough such that there exist  $\epsilon, v \in \R$ and the Banach sheaf $\omega_\epsilon^{\kappa}$ on $\bar{X}(v) \times W$ of locally analytic $v$-overconvergent $p$-adic familes (see \S.\ref{eigenvariety}), where $W=\Spm R$ is an affinoid of the weight space $\cW$ given by $w_1(\cU^{0}) \times_{\Q_p} w_2(\cU^1)$. Let $\cT_{W,1}$ be the affinoid $\Q_p$-algebra generated over $R$ by the image of the abstract Hecke algebra $\mathcal{H}_N$ in the space of endomorphisms of the sections of $\underset{ v \to 0}{\varinjlim}\rH^0(\bar{X}(v) \times W,\omega_\epsilon^{\kappa}(-D)) $ with slope $\leq 1$. By construction of $\cE_N$ (see \S.\ref{eigenvariety}), $\cE_{N,W}^1=\Spm \cT_{W,1}$ is an affinoid subdomain of $\cE_N$. Let $\theta: \mathcal{H}_N \twoheadrightarrow \cT_{W,1}$ be the natural surjection and $J$ be the kernel of $\theta$ generated by $g_1,...g_n$. 

On the other hand, let $\lambda$ be the morphism $$\lambda:\mathcal{H}_N \to \cO(\cU^{0}) \widehat{\otimes}_{\Q_p} \cO(\cU_1)$$ defined by  { \small 
 \[ \lambda_{\mathrm{Yo}}^*(P_\ell(X))=(X^2-a(\ell,\mathcal{F}) X+ \ell^{-3}\kappa_1\kappa_2(\ell))(X^2-\kappa_2(\ell)\ell^{-2}a(\ell,G) X+ \ell^{-3}\kappa_2(\ell).\kappa_1(\ell)),\] } for $ \ell \nmid Np$,   
where $P_\ell(X) \in \mathcal{H}_N[X]$ is the Hecke-Andrianov polynomial at $\ell \nmid Np$, and $\lambda_{\mathrm{Yo}}^*(U_0)=a(p,\mathcal{F})$, and $\lambda_{\mathrm{Yo}}^*(U^1)=a(p,\mathcal{F}) \times a(p,G)$. 

It is enough to prove that for every $1\leq i \leq n$, $\lambda(g_n)=0$. Note that the classical points old at $p$ of $\cU^0,\cU^1$ form a dense set $\Sigma$ of $\cU^0 \times_{\Q_p} \cU^1$. It follows from Proposition \ref{liftyoshida} that the points $\Sigma$ lift to a set $\tilde{\Sigma}$\footnote{Any point of $\Sigma$ corresponds to a $2$-tuple of old forms $(f_1,g_1)$ at $p$. Hence, $f_1$ (resp. $g_1$) is the $p$-ordinary (resp. $p$-critical) $p$-stabilization of a classical form of level $\Gamma_0(N)$ (resp. $\Gamma_0(\ell)$) denoted by $f^{old}_1$ (resp. $g_1^{old}$). So we can consider the Yoshida lift of $(f^{old}_1$, $ g_1^{old})$ and take its semi-ordinary $p$-stabilization which gives a point of $\tilde{\Sigma} \subset \cE_N^{1}$.} of points of $\cE_{N,W}^1$. Hence, for any $1 \leq i \leq n$, the specialization of $\lambda(g_i)$ at the points of the dense subset $\Sigma$ of $ \cU^0 \times_{\Q_p} \cU^1$ is trivial, yielding that \begin{equation}\label{IsubJ1}
\lambda(g_i)=0 \text{ for any } 1 \leq i \leq n. \end{equation}

Hence, we obtain a surjective homomorphism $$ \cO(\cE_{N,W}^1) \twoheadrightarrow \cO(\cU^{0}) \widehat{\otimes}_{\Q_p} \cO(\cU^1),$$
yielding a morphism $\cU^0 \times_{\Q_p} \cU^1 \to \cE_{N,W}^1$, and its image is an irreducible component of $\cE_{N,W}^1$.

\end{proof}

 \begin{cor}\label{cor12.5}
Assume $N>1$ is squarefree and not prime. Assume that $f$ is Steinberg for at least two primes $\ell_i \mid N, i=1,2$. Then the  Siegel eigenvarieties $\cE_{N}$ of tame level $N$ is singular at $\pi_\alpha$ and there exists at least two $p$-adic families specializing to $\pi_\alpha$.
 \end{cor}
 
\begin{proof}

If $f$ is Steinberg at $\ell_1$ and $\ell_2$, then by the previous theorem we get two irreducible components of $\cE_{N}$ (they are endoscopic) specializing to $\pi_\alpha$ by taking $\cU^1$ arising from $\mathcal{Y}_{\ell_i}$.

 \end{proof} 
A direct consequence of the above corollary is that $\kappa:\cE_N \to \cW $ is ramified at $\pi_\alpha$.
Let $\mathrm{S}^{\dagger}_k(N)^{\mid \mathbb{U} \mid_p=1}  \lsem  \pi_\alpha  \rsem $ be the generalized eigenspace attached to $\pi_\alpha$ inside the $L$-vector space of locally analytic overconvergent Siegel cusp forms $\mathrm{S}^{\dagger}_k(N)^{\mid \mathbb{U} \mid_p=1}$ of tame level $\Gamma(N)$ and slope $1$ for $\mathbb{U}$. 
\begin{cor}\ 
One has $\dim_L \mathrm{S}^{\dagger}_k(N)^{\mid \mathbb{U} \mid_p=1}  \lsem  \pi_\alpha  \rsem  \geq 2$.

\end{cor}

\begin{proof}
Since $\cW$ is smooth at $\kappa(\pi_\alpha)$ and $\cE_{N}$ is singular at $\pi_\alpha$, the local ring $\cT_0=\cO_{\cE_N,\pi_\alpha}/\gm_{\cO_{W,\kappa(\pi_\alpha)}}\cO_{\cE_N,\pi_\alpha}$ of the fiber of $\kappa^{-1}(\kappa(\pi_\alpha))$ at $\pi_\alpha$ is Artinian with a non-trivial tangent space (since $\kappa$ is necessarily ramified at $\pi_\alpha$ in this case). On the other hand, it follows from the construction of eigenvarieties that the local ring $\cT_0$ at $\pi_\alpha$ of the fiber $\kappa^{-1}\kappa(\pi_\alpha)$ acts faithfully on $\mathrm{S}^{\dagger}_k(N)^{\mid \mathbb{U} \mid_p=1}  \lsem  \pi_\alpha  \rsem $. Hence, $\dim \mathrm{S}^{\dagger}_k(N)^{\mid \mathbb{U} \mid_p=1}  \lsem  \pi_\alpha  \rsem  \geq 2$ (since $\dim_{L} \cT_0 \geq 2$).

\end{proof}

\appendix

\section{Some basic facts about rigid analytic geometry}\

We shall recall in this section the notions of ``very Zariski dense'' subset of a rigid analytic space and discuss  accumulation points of a Zariski dense set and irreducible components of rigid analytic spaces. Moreover, we will recall some basic properties of finite and torsion-free morphisms of affinoid spaces.

The following proposition is an analogue to \cite[Prop.2.1.6]{Berkovitch} for $\Q_p$-rigid analytic spaces.

\begin{prop}\label{basen}
Let $g:X \to Y$ be a finite morphism between two $\Q_p$-affinoid spaces, $y \in f(X)\subset Y$, and $g^{-1}(y)=\{x_1,x_2,...,x_n\}$, then there exists a small affinoid neighborhood $\cU_{i_0}$ of $y$ in $Y$ such that $g^{-1}(\cU_{i_0})=\bigcup_{1 \leq k \leq n} V_k^{i_0}$, and $V_k^{i_0} \cap V_j^{i_0}=\emptyset$, when $k \ne j$. Moreover, for any $1\leq k \leq n$, the domains $\{V_k^{i}, i \in I \text{ and } i \leq i_0 \}$ form a basis of neighborhood of $x_k$ when $\cU_i$ varies in a family $\{ \cU_i, i \in I \text{ and } i \leq i_0 \}$ of basis of affinoids containing $y$. 

\end{prop}

\begin{proof}
Let $B$ (resp. $A$) be the affinoid $\Q_p$-algebra corresponding to $X$ (resp. $Y$), and $\varphi: A \to B$ be the finite morphism corresponding to $g$. Let $B_y$ be the finite $\cO_{Y,y}$-algebra $B \otimes_A \cO_{Y,y}$; thanks to \cite[Theorem   2.1.5]{Berkovitch} the local ring $\cO_{Y,y}$ is Henselian, hence $$B_y=\footnote{Since $B_y$ is finite over the Henselian ring $\cO_{Y,y}$, it is necessarily a product of local Henselian rings.} \prod_{x_i \in g^{-1}(y)} \cO_{X,x_i}.$$ 

On the other hand, one has $$B_y=B \otimes_A \text{ }\cO_{Y,x}= B \otimes_A \text{ }\lim_{\stackrel{\longrightarrow}{\cU_i}}\cO_{Y}(\cU_i) =
 \lim_{\stackrel{\longrightarrow}{\cU_i}} \text{ }B \otimes_A \text{ }\cO_{Y}(\cU_i),$$ where $\{\cU_i, i \in I\}$ runs over the affinoid neighorhood of $y$.
 
 Hence, we have {\small \[   \text{ }\lim_{\stackrel{\longrightarrow}{\cU_i, i \in I}} B \otimes_A \text{ }\cO_{Y}(\cU_i)=\footnote{Since $B$ is finite over $A$, $B \widehat{\otimes}_{A} \text{ }\cO_{Y}(\cU_i)=B \otimes_A \text{ }\cO_{Y}(\cU_i)$.}\lim_{\stackrel{\longrightarrow}{\cU_i, i \in I}} B \widehat{\otimes}_A \text{ }\cO_{Y}(\cU_i)=\lim_{\stackrel{\longrightarrow}{\cU_i, i \in I}} \cO_X(g^{-1}(\mathcal{U}_i) )=  \prod_{x_i \in g^{-1}(y)} \cO_{X,x_i}.\]}

Thus, each local component $\cO_{X,x_j}$ of $\prod_{x_i \in g^{-1}(y)} \cO_{X,x_i}$ corresponds to an idempotent $e_j$ of $B_y$. So there exist an $i_0 \in I$ and orthogonal idempotents $\{\tilde{e}_j, 1\leq j \leq n \}$ of $\cO_X(g^{-1}(\mathcal{U}_{i_0}) )$ whose image in $B_y$ is $\{ e_j, 1\leq j \leq n \}$ and corresponding respectively to $\{x_1,...x_n\}$. Thus, $\cO_X(g^{-1}(\mathcal{U}_{i_0}) )= \prod_{\tilde{e}_j, 1\leq j\leq n} \tilde{e}_{j}.\cO_X(g^{-1}(\mathcal{U}_{i_0}) )$, and hence each affinoid subdomain {\small $\Spm \tilde{e}_{k}.\cO_X(g^{-1}(\mathcal{U}_{i_0}))$} of $X$ corresponds to a connected component $V_k^{i_0}$ of $g^{-1}(\cU_{i_0})$ containing $x_k$. Hence, $g^{-1}(\cU_{i_0})=\bigcup_{1 \leq k \leq n} V_i^{i_0}$, and $V_l ^{i_0}\cap V_k^{i_0}=\emptyset$, when $l \ne k$. 

Finally, the rest of the assertion follows from the fact that $$\lim_{\stackrel{\longrightarrow}{i\leq i_0}} \cO_X(V_k^{i}))= \cO_{X,x_k},$$ and the inductive limit is taken on the connected component $V_k^i$ of $g^{-1}(\mathcal{U}_i)$ containing $x_k$, when $\mathcal{U}_i$ varies over the affinoid neighborhoods of $x_k$ inside $\cU_{i_0}$.

\end{proof}

We recall that $F$ is an irreducible component of a $\Q_p$-separated reduced rigid analytic space $X$, if $F$ is the image of a connected component of the normalization $X^{\nor}$ of $X$ via the normalization morphism $X^{\nor} \to X$ (see \cite{conrad}). Moreover, when $X$ is a reduced affinoid $\Spm A$, then the irreducible components of $X$ correspond to $\Spm A/\gp$, where $\gp$ is a minimal prime ideal of $A$.

We recall also that a subset $Z$ of a reduced $\Q_p$-rigid analytic space $X$ is said to be Zariski-dense if the only analytic subset of $X$ containing $Z$ is $X$ itself. 
\begin{rem}
The set $S=\{(1/p^n,1/p^m), \text{ where } n \in \Z, m \in \N\}$ of the rigid affine plane $\mathbb{A}_2^{rig}$ of dimension $2$ is Zariski dense but for any open affinoid subdomain $\cU \subset \mathbb{A}_2^{\rig}$, the set $\cU \cap S$ is not Zariski dense in $\cU$ (it follows from the maximum modulus principle). 
\end{rem}

This example motivates the notion of a very Zariski dense set of a rigid analytic space (see also \cite[Definition II.5.1]{Bbook}):

\begin{defn} \begin{enumerate}
\item Let $X$ be a $\Q_p$-separated reduced rigid analytic space over $\Q_p$, and $\Sigma \subset X$ be a Zariski dense subset. We say that $\Sigma$ is {\it very}  Zariski-dense in $X$ if for every $z \in \Sigma$ there is a basis of open affinoid neighborhoods $\cU$ of $z$ in $X$ such that $\Sigma \cap \cU$ is Zariski-dense in $\cU$.

\item We say that a subset $Z$ of a $\Q_p$-separated rigid analytic space $Y$ accumulates at $y \in Y$ if there is a basis of affinoid neighborhoods $U \subset Y$ of $y$ such that $U\cap Z$ is Zariski-dense in $U$.

\end{enumerate}

\end{defn}

\begin{rem}
Let $X$ be a separated $\Q_p$-rigid space, $\{F_i\}$ be the irreducible components of $X$ and $\cU$ be an admissible open of $X$. Then it follows from \cite[cor.2.2.9]{conrad} that each irreducible component of $\cU$ is contained in a unique $F_i$ and for any $i$, $\cU \cap F_i$ is empty or the union of irreducible components of $\cU$.

\end{rem}

\begin{prop}\label{geneta}\

\begin{enumerate}

\item Let $g:X \to Y$ be a finite flat morphism between two $\Q_p$-affinoid spaces such that $X$ is equidimensional and $Y$ is irreducible. Assume that $g$ is \'etale at a Zariski dense set $\Sigma$ of points of $X$, then after shrinking $X$ to a smaller admissible open $X'$ of $X$, the restriction $g: X' \to g(X')$ is \'etale and $g(X')$ is an admissible open of $Y$.

\item Let $g:X \to Y$ be a finite morphism between rigid analytic spaces, then for any irreducible component $F$ of $X$, $g(X)$ is a closed irreducible component of $Y$.
\end{enumerate}
\end{prop}

\begin{proof}

i) It is known that $g$ is \'etale outside of the support of the relative differential sheaf $\Omega_{X/Y}$. Moreover, since $g$ is \'etale at a Zariski dense set of points of $X$, the support $Z$ of $\Omega_{X/Y}$ is a Zariski closed set of $X$ of dimension $< \dim X$  (since $\Sigma$ is Zariski dense in all irreducible components of $X$ by \cite[Prop.2.2.8]{conrad}). Hence, $g_{\mid {X-Z}}: X-Z \to Y $ is \'etale, and the image of the Zariski {\it open} \footnote{Note that a Zariski open $U$ of rigid analytic space $X$ is not necessarily an affinoid subdomain of $X$. Take $X=\Spm \Q_p<T>$ and $U=D(T)$ the locus where $T$ is invertible; it is clear that $U$ doesn't satisfy the maximal modulus principle for the function $1/T$, and hence $U$ is not an affinoid. However, any Zariski open is an admissible open for the rigid topology.} $X-Z $ under $g$ is a Zariski open of $Y$ (a flat morphism is Zariski open). 

ii) The assertion follows from the fact that a finite morphism is Zariski closed and \cite[Proposition.2.2.3]{conrad}.

\end{proof}

The following proposition was proved by Chenevier in \cite{Chenevier} using base change arguments. We give in the following a more direct proof:
\begin{prop}\label{vzd}
Let $g:X \to Y$ be a finite torsion-free morphism between two reduced $\Q_p$-affinoid spaces and  such that $Y$ is irreducible. Then :
\begin{enumerate}
\item $X$ is equidimensional of dimension equal to $\dim Y$ and the image of each irreducible component of $X$ under $g$ is $Y$.

\item Let $\Sigma$ be a Zariski dense set of $Y$, then $g^{-1}(\Sigma)$ is Zariski dense in $X$.

\end{enumerate}

\end{prop}

\begin{proof}\

i) Let $B$ (resp. $A$) be the affinoid algebra corresponding to $X$ (resp. $Y$) and $g^*:A \to B$ be the finite torsion-free morphism corresponding to $g$. Since $Y$ is irreducible and reduced, $A$ is a domain. Let $\gp$ be a minimal prime ideal of $B$ corresponding to an irreducible component $F$ of $X$, it follows from the fact that $B$ is a torsion-free finite $A$-algebra that the morphism $A \to B/\gp$ is injective (since the zero divisors of a reduced Noetherian ring are the union of its minimal prime ideals). Moreover, the image of the natural composition $F \to X \to Y$ is dense, because $A \to B/\gp$ is injective (so the image of the morphism $\Spec B/\gp \to \Spec A$ is Zariski dense) and $\Spec A$ and $\Spec B$ are Jacobson schemes (so $\Spm A$ is Zariski dense in $\Spec A$, and the same for $B$). 

However, $g$ is also finite, and then Zariski closed. Hence, the irreducible component $F$ of $X$ surjects onto $Y$, and since the morphism $A \to B/\gp$ is injective and finite, then $\dim F=\dim Y$ (it follows from the Going-up theorem), and hence $X$ is equidimensional of dimension equal to $\dim Y$.

ii) A subset $\Sigma ' \subset X$ is a Zariski dense set of a reduced affinoid $X$ if and only if for any irreducible component $F$ of $X$ (see \cite[Prop.2.2.8]{conrad}, $\Sigma ' \cap F$ is a Zariski dense set of $F$. Thus, it is enough to prove the assertion when $X$ is reduced an irreducible. Assume that $X$ is irreducible and that $\Sigma$ is a Zariski dense set of $Y$. Let $\Sigma' \subset X$ denote the subset $g^{-1}(\Sigma)$ of $X$. Since $g$ is finite and torsion-free, then $g$ is closed for the Zariski topology and surjective, and then the Zariski closure of $\Sigma'$ is necessarily an analytic subspace $Z \subset X$ of dimension equal to $\dim Y= \dim X$, because $g(Z)$ is a Zariski closed set of $Y$ containing $\Sigma$ (so $g(Z)$ contains $Y$ the  closure of $\Sigma$). Hence, $Z$ is finite and surjects on $Y$ and it follows that $Z=X$, since they have the same dimension and $X$ is irreducible.

\end{proof}

\begin{lemma}\label{restzarden}
Let $\mathcal{U}=\Spm A$ be an equidimensional affinoid of dimension $2$, $F$ be a Zariski closed subset of $\mathcal{U}$ of dimension $\leq 1$, $\mathcal{U'}$ be the admissible open given by $\mathcal{U}-F$. Let $\Sigma$ be Zariski dense set of $\mathcal{U}$, then $\Sigma'=\Sigma \cap \mathcal{U}' $ is Zariski dense in $\mathcal{U}$ and in $\mathcal{U}'$.

\end{lemma}

\begin{proof}
Note that $\Sigma= \Sigma' \cup (\Sigma \cap F)$. Hence, the Zariski closure $\bar \Sigma$ of $\Sigma$ is equal to the union of the Zariski closure $\bar \Sigma' $  of $\Sigma'$ with the closure $\overline{\Sigma \cap F}$ of $\Sigma \cap F$. On the other hand, $\bar \Sigma=\mathcal{U}$ and it is equidimensional of dimension $2$, and $\overline{\Sigma \cap F} \subset F$ is of dimension at most one. Hence, $\bar{\Sigma'}=\mathcal{U}$, yielding that $ \Sigma'$ is dense in $\mathcal{U}$ and so in $\mathcal{U}'$.

\end{proof}

\section{On the very Zariski density of classical points in the Eigenvariety $\cE_\Delta$} \label{s4}

The goal of this section is to recall quickly the construction of the Siegel eigenvarieties and to prove that classical points which are old at $p$ and of cohomological weights are very Zariski dense in them.

\subsection{The Weight space $\cW$}

Recall that the connected components of $\cW$ are naturally indexed by $\cW^{a,b}$, where  $(a,b) \in (\Z/(p-1)\Z)^2$. The classical weights $(k_1,k_2) \in   \cW^{a,b}$ are congruent to $(a,b) \mod p-1$, in other words, the discrete part of the restriction of any character of $\cW^{a,b}(\C_p)$ to $\Z/p\Z^{\times}$ is $(\omega_p^{a},\omega^b_p)$, where $\omega_p$ is the Teichm{\"{u}}ller character. In addition, the formal scheme $\mathrm{Spf } \text{ } \Z_p   \lsem   T_1,T_1  \rsem $ is a Raynaud's formal model of any connected component\footnote{Note that $\mathrm{Spm} \text{ } \Z_p   \lsem   T_1,T_1  \rsem [1/p]=\cW^{a,b}$, and $\mathrm{Spm} \text{ } \Z_p   \lsem   T_1,T_1  \rsem [1/p]$ is the open disk of dimension $2$ and radius $1$. }  $\cW^{a,b}$of the weight space $\cW$.

\begin{rem}The category of admissible $\Z_p$-formal schemes (which includes locally topologically of finite type formal $\Z_p$-schemes) is localised with respect to blowups in the special fibre and the Raynaud generic fibre functor  defines an equivalence of categories between the localised category of admissible formal $\Z_p$-schemes and the category of rigid analytic spaces over $\Q_p$.
\end{rem}

Now, let $\underline{k}=(k_1,k_2) \in \Z^2$, any morphism $\underline{k}: (\Z_p^\times)^2 \to \Q_p^{\times}$ sending $(z_1,z_2) \to z_1^{k_1}.z_2^{k_2}$ give a point of $\cW(\Q_p)$ and which denote again by $\underline{k}$, and we call it ``an algebraic weight''. More generally, any character of $\cW(\C_p)$ which is a product of a character $\underline{k} \in \Z^2 \subset \cW(\Q_p)$ with a finite character $\chi:(\Z_p^\times)^2 \to \bar{\Q}_p^{\times}$ is called ''an arithmetic character'' and denoted by $(\underline{k},\chi)$. 

\begin{lemma}\label{vzdw}
The classical weights $\Z^2$ of $\cW(\Q_p)$ are very Zariski dense in the weight space $\cW$.

\end{lemma}

\begin{proof}
It follows from the Weierstrass preparation theorem that the set $\Z^2$ of integral weights is Zariski-dense in $\cW$. Moreover, the $p$-adic topology on the union of open discs $\cW$ induces by restriction the topology on $\Z^2$ for which we have a natural basis of neighborhood of $\underline{k} \in \Z^2$ given by the congruence classes modulo $p^n(p-1)$ for all $n$. Hence $\Z^2$ is very Zariski dense.
\end{proof}

\subsection{Geometric Siegel cuspforms}

Let $\mathrm{G}$ denote the algebraic group $\mathrm{GSp}_4$ and $\Gamma(N)$ be the open compact subgroup of $G(\widehat{\Z})$ of level $N$ given by $\{\gamma \in G(\widehat{\Z})\mid \gamma =\mathbbm{1}_4 \mod{N}\}$. 

Assume now that $N\geq 5$, and let $X/\Z_p$\footnote{The generic fiber $X/\Q_p$ is smooth, and the special  fiber $X/\mathbb{F}_p$ is singular, and  it even has vertical components.} be the Siegel scheme of level $\Gamma(N) \cap I_1$, where $I_r$ is the standard Iwahoric at $p$ of $G$ given by $\{\gamma \in \mathrm{GSp}_4(\Z_p) \mid \gamma \mod p^r \in B(\Z/p^r\Z)\}$ and $B$ is the Borel of $\mathrm{GSp}_4$. There exists a universal abelian scheme $A/X$ with identity section $e$ and we let $\omega:=e^{*}(\Omega_{A/X})$ be the conormal sheaf. Note that $\omega$ is a locally free sheaf of rank $2$ over $X$. Let $\bar{X}$ denote a toroidal compactification of $X$ (it is not unique and depends on a combinatorial choice, see \cite{CF}), $\bar{A}$ be the semi-abelian scheme extending $A$ to $\bar{X}$ and $D=\bar{X}/X$ be the normal crossing divisor at infinity. The sheaf $\omega$ extends to a locally free sheaf of rank $2$ over $\bar{X}$, which we again denote by $\omega$.																																																			

The classical cuspidal Siegel forms of level $\Gamma(N) \cap I_1$ and weight $k=(k_1,k_2)$ and coefficients in a $p$-adic field $L$ (we have $k_1 \geq k_2$)  are the elements of $\rH^{0}(\bar{X}_{L}, \omega^{k}_L(-D))$\footnote{It follows from Koecher principle that $\rH^{0}(\bar{X}_{L}, \omega^{k}_L)$ does not depend on the choice of the toroidal compactification $\bar{X}$ of $X$.}, where $\omega^{k}$ is the locally free sheaf $\mathrm{Sym}^{k_1-k_2} \omega \otimes \det \omega^{k_2}$, and  $\omega^{k}_L$ is the base change of $\omega^k$ to $\bar{X}_L$. Let $\bar{X}^{\rig}/\Q_p$ be the rigid analytic space given by taking the {\it generic fiber} of the formal scheme given by the completion of $\bar{X}$ along its special fiber, and writing again $\omega$ for the analytification of $\omega$, and let $\bar{X}^{\ord}$ be the multiplicative ordinary locus of $\bar{X}^{\rig}$ (it is not an affinoid), then the $p$-adic (resp. $v$-overconvergent) cuspidal Siegel modular forms of tame level $\Gamma(N)$  weight $k=(k_1,k_2)$ and coefficients in $L$ are $\rH^{0}(\bar{X}^{\ord}_L, \omega^{k}_L(-D))$ (resp. $\rH^{0}(\bar{X}^{\ord}_L(v), \omega^{k}_L(-D))$), where $\bar{X}(v)$ is the $v$-overconvergent neighborhood of the multiplicative ordinary locus $\bar{X}^{\ord}$.

Andreatta, Iovita and Pilloni constructed for any weight $k \in \cW(\C_p)$ and certain parameters $v,w \in \mathbb{R}^{\times}_{+}$ a Banach sheaf $\omega^k_w$ over $\bar{X}(v)$, and a natural sheaf monomorphism $\omega^k \hookrightarrow \omega_w^k$ when $k=(k_1,k_2)$ is classical (see \cite{AIP}), and they describe precisely the cokernel of that monomorphism. The sheaf $\omega_w^k$ is isomorphic locally for the \'etale topology to the $w$-analytic induction of the Borel $B(\Z_p)$ to the Iwahoric of $\GL_2$ with respect to the character $k$. Note that any character $k \in \cW(\C_p)$ is locally analytic by \cite[\S2.2]{AIP} and hence $\omega^{k}_w$ is a non-zero Banach sheaf (The sections of $\omega^{k}_w$ are congruent to the image of the Hodge-Tate map by \cite[Prop.4.3.1]{AIP}).

The $p$-adic modular forms obtained by this interpolation are locally analytic overconvergent (not necessarily overconvergent), however those satisfying the slope condition of \cite[Theorem    7.1.1]{AIP} are overconvergent (see also \cite[Prop.2.5.1.]{AIP} and \cite[Prop.7.2.1]{AIP}). Note that this construction is independent of the choice of the toroidal compactification of $\bar{X}$ (see \cite[Theorem   1.6.1]{Lu} and \cite[Prop.5.5.2]{AIP}) and we denote the corresponding eigenvariety by $\cE_N$.

\subsection{Local charts of the variety $\cE_N$ and density of classical points of $\cE_N$.}\label{eigenvariety}\

Let $\chi:(\Z/p\Z^\times)^2 \to \Q_p^\times$ be a character, $\cW^\chi$ the connected component of $\cW$ corresponding to $\chi$, and $\cE_N^{\chi}$ the union of connected components of $\cE_N$ given by the restriction of $\cE_N$ to $\cW^{\chi}$.

For $w,v \in \R$ let $W=\Spm R$ be a small enough affinoid subdomain of $\cW^{\chi}$ to ensure the existence of the Banach sheaf $\omega^{\kappa}_w(-D)$ of $\bar{X}(v) \times \Spm R$ interpolating the Banach sheaf $\omega^{k}_w(-D)$ of $w$-analytic $v$-overconvergent Siegel cusp form when $k$ varies in $\Spm R$ ($\kappa$ denotes here the tautological character $(\Z_p^\times)^2 \to R^{\times}$).

On the other hand, let $S^{\dagger}_{\kappa}$ be the Frechet $R$-module of $\epsilon$-overconvergent cuspidal Siegel families over the affinoid $R$  and given by \[\lim_{\stackrel{\longrightarrow}{v \to 0,w \to \infty}}\rH^0(\bar{X}(v)\times \Spm R,\omega^{\kappa}_w(-D)).\] The action of the Hecke operator $\mathbb{U}=U_0.U_1$ is completely continuous on the Frechet $R$-module $S^{\dagger}_{\kappa}$. Let $\cT_{W,r}$ be the image of the Hecke algebra generated over $R$ by the image of $\mathcal{H}_N$ in the space of endomorphisms of $S^{\dagger,v \leq r}_{\kappa}$, where $S^{\dagger, \leq r}_{\kappa}$  is the $R$-finite submodule of $S^{\dagger}_{\kappa}$ of slope at most $r$ for $\mathbb{U}=U_0.U_1$.\footnote{Note that the action of $ \mathbb{U}$ is completely continuous on $S^{\dagger}_{\kappa}$, so we have a slope decomposition.} It follows from the results of \cite[\S.II]{Bbook} that \begin{equation}\label{localcards}\cE_{N,W}^r := \Spm \cT_{W,r}, \end{equation}
is an affinoid subdomain of $\cE_N$ and by construction $\cE_{N,W}^r$ is finite and torsion-free over $W$ and the $\{\cE_{N,W}^r\}$ form an admissble covering of $\cE$.

Since the ordinary locus of any toroidal compactification of the Siegel modular scheme is not an affinoid, we cannot prove that the specialization $$\rH^0(\bar{X}(v) \times_{\bar{\Q}_p}\Spm R,\omega^{\kappa}_w) \to \rH^0(\bar{X}(v),\omega^{k}_w)$$ is surjective and that $\rH^0(\bar{X}(v) \times \Spm R,\omega^{\kappa}_w)$ is a projective $R$-Banach module. However, Andreatta-Iovita-Pilloni proved in \cite[Prop.8.2.3.3]{AIP} a control theorem for cuspidal families and that $\rH^0(\bar{X}(v) \times \Spm R,\omega^{\kappa}_w(-D))$ is a projective $R$-Banach module, by projecting the sheaf $\omega^{\kappa}_w(-D)$ to the minimal compactification of the Siegel modular scheme, and using the fact that small $v$-overconvergent neighborhoods of the multiplicative ordinary locus of the minimal compactification of the Siegel modular scheme are affinoid spaces, and the deep descent result 
\cite[Prop.8.2.2.4]{AIP}.

Skinner-Urban constructed in \cite[\S2]{urban} a semi-ordinary eigenvariety $\cE_N^{\mid U_0 \mid_p=1} \subset \cE_N$ for overconvergent Siegel cusp forms of tame level $\Gamma(N)$ and genus $2$ by interpolating the locally free sheaf $\omega^{k}$ inside a Banach sheaf $\omega^{\kappa}_w$ over the weight space $\cW$ using the Igusa tower.  That construction is a special case of the construction given by Andreatta-Iovita and Pilloni in \cite{AIP} of the eigenvariety $\cE_N$, since the linearization of the Hodge-Tate map $$ \mathrm{HT}_{H_n^D}: H_n^D \to \omega_{H_n}$$ is surjective on the multiplicative ordinary locus ($H_n \subset \bar{A}$ is the level $n$ canonical subgroup and $H_n^D$ is its Cartier dual), and the fact that any semi-ordinary (i.e of slope $0$ for $U_0$) $p$-adic Siegel cuspforms of finite slope for $U_1$ overconverges to a strict neighborhood of the ordinary locus. For the latter note that under the iteration of the Hecke correspondances at $p$, an overconvergent neighborhood of $X^{\ord}$ accumulates around the multiplicative ordinary ordinary locus $X^{\ord}$. The correspondence $U_0$  improves the radius of overconvergence. Hence, the functional equation $U_0.g=U_0(g).g$ allows us to extend $g$ to a bigger neighborhood of the multiplicative ordinary locus when $U_0(f) \ne 0$ (the function degree of \cite[Theorem   3.1.]{Pilloni} increases under the iteration of $U_0$). Meanwhile, one can use a similar functional equation for $U_1$ to get classicality at the level of the sheaves when the slope satisfies the condition of \cite[Prop.7.3.1]{AIP}.

By construction of $\cE_N$ we have an algebra homomorphism $\mathcal{H}_{N} \to \mathcal{O}^{\rig}_{\cE_N}(\cE_N)$, and the image lands in the subring $\mathcal{O}^{\rig}_{\cE_N}(\cE_N)^{+}$ given by the global section bounded by $1$ on $\cE_N$. Therefore, the canonical application ``system of eigenvalues'' induces a correspondence between the systems of eigenvalues for Hecke operators occuring in $\mathcal{H}_N$ of locally analytic overconvergent cuspidal Siegel eigenforms of tame level $\Gamma(N)$ and weight $k \in \mathcal{W}(\C_{p})$ having non-zero $\mathbb{U}$-eigenvalue, and the set of $\mathbbm{C}_{p}$-valued points of weight $k=(k_1,k_2)$ on the Siegel eigenvariety $\cE_N$. Note that for any overconvergent form $g $ corresponding to a point of $\cE_N$ of weights $(l_1,l_2)$, \begin{equation} \label{U1norm} g\mid U_1=p^{l_2-3}U_1(g).g; \end{equation} we renormalize $U_1$ in the aim to have a good $p$-adic interpolation (see for example \cite[Theorem   2.4.14]{urban}).

One has the following Lemmas proving the very Zariski density of the classical points having a crystalline representation at $p$ in $\cE_N$, which is important for applying further the results of \cite[\S4]{bb} (see the hypothesis (HT) of \cite[\S.3.3.2]{bb}).

\begin{lemma}\label{zdop} Let $z \in \cE_N$ be a classical point old at $p$, then there exists an affinoid neighborhood $\Omega$ of $z$ in $\cE_N$ of constant slopes for $U_0,U_1$ and such that the old at $p$ classical points of regular weights of $\Omega$ are very Zariski dense in it, $\kappa(\Omega)$ is an open affinoid subdomain of $\cW$, and each irreducible component of $\Omega$ surjects to $\kappa(\Omega)$.

\end{lemma}

\begin{proof} Note that $\cE_N$ is admissibly covered by $\{\cE_{N,W}^r \}$. Hence, there exists an affinoid subdomain $\cE_{N,W}^r$ of $\cE_N$ containing $z$ and surjecting on the affinoid subdomain $W \subset \cW$. By construction of $\cE_N$, the slopes of $U_0,U_1$ are locally constant. Then Prop.\ref{basen} and Prop.\ref{vzd} yields that we can shrink $\cE_{N,W}^r$ to a smaller open affinoid subdomain $\Omega$ of $\cE_N$ containing $z$ and with constant slope $\mathrm{S}_1$ (resp. $\mathrm{S}_2$) for the Hecke operator $U_0$ (resp. $U_1$) and such that $\kappa(\Omega)$ is an open affinoid subdomain of $\cW$, and $\kappa:\Omega \to \kappa(\Omega)$ is finite and torsion-free (so the restriction of $\kappa$ to any irreducible component of $\Omega$ is surjective by Prop.\ref{vzd}). 

Since $\Omega$ contains the classical point $z$, then the points of $\Omega$ with weights satisfying the small slope conditions of \cite[Theorem   7.1.1]{AIP} form a  Zariski dense set in $\Omega$, because the {\it algebraic} points $(l_1,l_2)$ of $\kappa(\Omega)$ satisfying the inequality of the small slope conditions of \cite[Theorem   7.1.1]{AIP} form a Zariski dense set of $\kappa(\Omega)$ (so their preimage is dense in $\Omega$ by Prop.\ref{vzd}). Moreover, it follows from the criterion of classicality of overconvergent forms that the points  satisfying the small slope conditions of \cite[Theorem   7.1.1]{AIP} are necessarily classical.  Actually, Prop.\ref{basen}, Prop.\ref{vzd} and Lemma \ref{vzdw} show that classical points of $\Omega$ are very Zariski-dense in it. Finally, the assertion follows from the fact that the classical points of $\Omega$ with sufficiently {\it regular} weights satisfy the assumptions of \cite[Theorem   2.4.17]{urban}, and hence they are old at $p$. 

\end{proof}

\begin{cor}\label{veryzariskidense}
Let $\cE_N^{\ord,1}$ be the admissible open of $\cE_N$ defined by \[\cE_N^{\ord,1}:=\{z \in \cE_N, \mid U_0(z) \mid_p=1, \mid U_1(z) \mid_p=p^{-1}\},\] $C \in \N_{>1},$ and $\Sigma_C$ be the set of points of $\cE_N^{\ord,1}$ of  ``algebraic weights'' $(k_1,k_2)$ satisfying $k_1 > k_2+C \geq \mathrm{Max}(9,C)$. Then:

\begin{enumerate}

\item The overconvergent cuspforms of $\Sigma_C$ are classical and old at $p$.

\item The set $\Sigma_C$ is very Zariski dense in $\cE_N^{\ord,1}$. 

\item The point $\pi_\alpha$ of $\cE_N^{\ord,1}$ corresponding to $\pi_\alpha$ is an accumulation point of $\Sigma_C$.

\end{enumerate}
\end{cor}

\begin{proof}
The points of $\Sigma_C$ have slope equal to $1$, Iwahoric level at $p$ and satisfy the slope condition $1 < k_1 -k_2 +1, k_2 >>0 $ of the classicality criterion for overconvergent Siegel cuspforms. Hence they are necessarily classical. A direct computation shows that the points of $\Sigma_C$ satisfy the assumptions of \cite[Theorem   2.4.17]{urban}, and hence they are necessarily old at $p$.

Since the algebraic weights $(k_1,k_2)$ with $k_1 > k_2+C \geq \mathrm{Max}(9,C)$ are very Zariski dense in $\cW$ (see Lemma \ref{vzdw}), the assertion of (ii) and (iii) follows directly from the argument already used to proof Lemma \ref{zdop}.

\end{proof}

\subsection{Siegel eigenvariety of paramodular level $N$}\label{eigenvariety2}

Let $\cE_\Delta$ be the Siegel eigenvariety of tame level the paramodular group $\Delta$. Since the classical Siegel cuspforms of level $\Delta \cap I_1$ are necessarily of level $\Gamma(N) \cap I_1$, the results of \cite[II.5.]{Bbook} yields that there exists a natural closed immersion $\iota:\cE_{\Delta} \hookrightarrow \cE_N$ compatible with the system of Hecke eigenvalues and the weights: $\xymatrix{
  & \mathcal{E}_{\Delta} \ar[dl]_{\iota} \ar[dr]^{\kappa}
  \ar@{}[d]|-{\circlearrowleft} \\
  \cE_{N} \ar[rr]_{\kappa}
  && \cW
}$

Since the restricted Hecke algebra $\mathcal{H}_{Np}$ generated over $\Z$ by the Hecke operators $T_{\ell,1},T_{\ell,2}, \mathrm{S}_\ell$ for $\ell \nmid Np$ acts semi-simply on classical cuspidal Siegel paramodular eigenforms of cohomological weights, \cite[Lemma.I.9.1]{Bbook} implies that $\cE_\Delta$ is reduced. Note also that $\cE_{\Delta}$  is equidimensional of dimension $2$.

\begin{cor}\label{veryzariskidense1}
Let $\cE_\Delta^{\ord,1}$ be the admissible open of $\cE_\Delta$ defined by \[\cE_\Delta^{\ord,1}:=\{z \in \cE_\Delta, \mid U_0(z) \mid_p=1, \mid U_1(z) \mid_p=p^{-1}\},\] $C \in \N_{>1},$ and $\Sigma_C$ be the set of points of $\cE_\Delta^{\ord,1}$ of  ``algebraic weights'' $(k_1,k_2)$ satisfying $k_1 > k_2+C \geq \mathrm{Max}(9,C)$. Then:

\begin{enumerate}

\item The overconvergent cuspforms of $\Sigma_C$ are classical and old at $p$.

\item The set $\Sigma_C$ is very Zariski dense in $\cE_\Delta^{\ord,1}$. 

\item The point $\pi_\alpha$ of $\cE_\Delta^{\ord,1}$ is an accumulation point of $\Sigma_C$.

\end{enumerate}
\end{cor}

\begin{proof}
It follows immediately from Corollary \ref{veryzariskidense} and the fact that a subset of an affinoid space is a Zariski dense if and only if its intersection with any irreducible component is Zariski dense in that irreducible component (see \cite[Prop.2.2.8]{conrad}).

\end{proof}

\section{Some examples where $\dim \rH^1_{f,\unr}(\Q,\rho_f(k-1)=1$}\label{example}

Using Nekovar's result \cite[Prop.4.2.3]{Nek} about $\mathbb{I}$-adic Selmer groups mentioned before Corollary \ref{nonItor} we can exhibit infinitely many examples of modular forms $f$ of weight $k \geq 3$ such that they satisfy the condition $\dim \rH^{1}_{f,\unr}(\Q,\rho_f(k-1))=1$ in Theorem \ref{princ1}. This requires finding suitable elliptic curves with ordinary reduction at $p$ and considering their corresponding Hida family $\mathcal{F}$. One such example is discussed in section 9.1 of \cite{bk}, where for $p=5$ and $N=731$ the residual Selmer group {\small \[ \rH^1_{f,\unr}(\Q, \overline{\rho}_{E,p}(1)) =\rH^1_{f,\unr}(\Q, {\rho}_{E,p}(1) \otimes \Q_p/\Z_p)[p]={\rm Sel}_p(E)[p] \]} of the rank 1 elliptic curve $E$ (Cremona label 731a1) is calculated to have order $5$ (since the order of vanishing of $L(f,s)$ at $s=1$ is one we know that the BSD conjecture holds). This elliptic curve has non-split reduction at both primes dividing $N$ and good ordinary reduction at $5$, with $a_5(E)=-1$ and therefore $\alpha \neq 1$. 
{In addition this example satisfied the condition  $L_p(f_\alpha,\omega^{-1}_p,T=p) \ne 0$.}

In the following assume that $f$ is the $p$-ordinary stabilization of the weight two cuspform attached to a rank 1 elliptic curve $E/\Q$. Recall that $\mathbb{I}$ is the finite flat extension of $\Z_p\lsem T \rsem$ generated by the Fourier coefficients of the Hida family $\mathcal{F}$ specializing to $f$ ($\mathbb{I}$ is an integral domain).

Note that the cohomology groups $\rH^i_{f,\unr}(G_\Q^{Np},\rho_{\mathcal{F}} \otimes \chi_{\univ}^{-1/2})$ of the Selmer complex are of finite type over $\mathbb{I}$ when $i \in \{1,2\}$ (see \cite[Prop.4.2.3]{Nek}).

Let $\gp_f \subset \mathbb{I}$ be the height one prime ideal corresponding to the system of Hecke eigenvalues of $f$. We have the following control theorem proved by Nekovar \cite[(0.15.1.1)]{Nek}
{ \small
\begin{equation}\label{control} 0 \to \rH^1_{f,\unr}(\Q,\rho_{\mathcal{F}}(\chi_{\univ}^{-1/2})) \otimes\text{ } \mathbb{I}_{\gp_f}/\gp_f \to \rH^1_{f,\unr}(\Q, \rho_f(1)) \to \rH^2_{f,\unr}(\Q,\rho_{\mathcal{F}}(\chi_{\univ}^{-1/2}) \otimes \mathbb{I}_{\gp_f})[\gp_f]. 
\end{equation} }
where $\rH^2_{f,\unr}(\Q,\rho_{\mathcal{F}}(\chi_{\univ}^{-1/2}) \otimes_{\mathbb{I}} \mathbb{I}_{\gp_f})[\gp_f]$ means the submodule annihilated by the prime ideal $\gp_f$.

Since $\dim \rH^1_{f,unr}(\Q, \rho_f(1))=1$ Nakayama's lemma applied to \eqref{control}  yields that the $\mathbb{I}_{\gp_f}$-module $\rH^1_{f,\unr}(\Q,\rho_{\mathcal{F}}(\chi_{\univ}^{-1/2}) \otimes \mathbb{I}_{\gp_f})$ is a monogenic. Moreover, it follows from Corollary \ref{nonItor} that  $\rH^1_{f,\unr}(\Q,\rho_{\mathcal{F}}(\chi_{\univ}^{-1/2}) \otimes_{\mathbb{I}} \mathbb{I}_{\gp_f})$ is a torsion-free $\mathbb{I}_{\gp_f}$-module, so {\small \[ \rH^1_{f,\unr}(\Q,\rho_{\mathcal{F}}(\chi_{\univ}^{-1/2}) \otimes_{\mathbb{I}} \mathbb{I}_{\gp_f})=\rH^1_{f,\unr}(\Q,\rho_{\mathcal{F}}(\chi_{\univ}^{-1/2})) \otimes_{\mathbb{I}} \mathbb{I}_{\gp_f}\]} is a free rank one $\mathbb{I}_{\gp_f}$-module. Thus there exists a principal  Zariski open $D(s)$ of $\Spec \mathbb{I}$ (where $s \in \mathbb{I}$) such that the localization of $\rH^1_{f,\unr}(\Q,\rho_{\mathcal{F}}(\chi_{\univ}^{-1/2}))$ at the non-vanishing locus $D(s)$ is a free rank one $\mathbb{I}[1/s]$-module. On the other hand, let $\mathcal{U} \subset D(s)$ be the Zariski open defined as the complementary of the support of the $\mathbb{I}$-torsion part of $\rH^2_{f,\unr}(\Q,\rho_{\mathcal{F}}(\chi_{\univ}^{-1/2}))$. Note that the classical points of $\mathcal{U}$ are Zariski dense, hence \eqref{control} yields that all the classical specialization $\mathcal{F}_z$ of the Hida family $\mathcal{F}$ at a point $z \in \mathcal{U}$ of weight $k_z$ satisfy
  $$\dim \rH^{1}_{f,\unr}(\Q,\rho_{\mathcal{F}_z}(k-1))=1.$$


\bibliographystyle{amsalpha}
\bibliography{xampl}

\providecommand{\bysame}{\leavevmode\hbox to3em{\hrulefill}\thinspace}
\providecommand{\MR}{\relax\ifhmode\unskip\space\fi MR }
\providecommand{\MRhref}[2]{%
  \href{http://www.ams.org/mathscinet-getitem?mr=#1}{#2}
}
\providecommand{\href}[2]{#2}
\begin{thebibliography}{BDP18}

\bibitem[AIP15]{AIP}
Fabrizio Andreatta, Adrian Iovita, and Vincent Pilloni, \emph{{$p$}-adic
  families of {S}iegel modular cuspforms}, Ann. of Math. (2) \textbf{181}
  (2015), no.~2, 623--697. \MR{3275848}

\bibitem[Art04]{A}
James Arthur, \emph{Automorphic representations of {${\rm GSp(4)}$}},
  Contributions to automorphic forms, geometry, and number theory, Johns
  Hopkins Univ. Press, Baltimore, MD, 2004, pp.~65--81. \MR{2058604}

\bibitem[BC06]{B-C}
J.~Bella\"{\i}che and G.~Chenevier, \emph{Lissit\'{e} de la courbe de {H}ecke
  de {$\rm GL_2$} aux points {E}isenstein critiques}, J. Inst. Math. Jussieu
  \textbf{5} (2006), no.~2, 333--349. \MR{2225045}

\bibitem[BC09]{bb}
Jo\"{e}l Bella\"{\i}che and Ga\"{e}tan Chenevier, \emph{Families of {G}alois
  representations and {S}elmer groups}, Ast\'{e}risque (2009), no.~324,
  xii+314. \MR{2656025}

\bibitem[BD16]{D-B}
Jo\"{e}l Bella\"{\i}che and Mladen Dimitrov, \emph{On the eigencurve at
  classical weight 1 points}, Duke Math. J. \textbf{165} (2016), no.~2,
  245--266. \MR{3457673}

\bibitem[BD19]{BD19}
Adel Betina and Mladen Dimitrov, \emph{Geometry of the eigencurve at cm points
  and trivial zeros of {K}atz $p$-adic {$L$}-functions}, preprint (2019),
  \url{https://arxiv.org/abs/1907.09422}.

\bibitem[BDP18]{BDPozzi}
Adel Betina, Mladen Dimitrov, and Alice Pozzi, \emph{On the failure of
  {G}orensteinness at weight 1 {E}isenstein points of the eigencurve}, preprint
  (2018), \url{https://arxiv.org/pdf/1804.00648.pdf}.

\bibitem[Bel08]{Bellaiche08}
Jo\"{e}l Bella\"{\i}che, \emph{Nonsmooth classical points on eigenvarieties},
  Duke Math. J. \textbf{145} (2008), no.~1, 71--90. \MR{2451290}

\bibitem[Bel10]{Bbook}
Jo\"{e}l Bella{\"{\i}}che, \emph{Eigenvarieties and $p$-adic {$L$}-functions},
  Book in preparation (2010),
  \url{http://people.brandeis.edu/~jbellaic/preprint/coursebook.pdf}.

\bibitem[Bel12]{Be}
Jo\"{e}l Bella\"{\i}che, \emph{Critical {$p$}-adic {$L$}-functions}, Invent.
  Math. \textbf{189} (2012), no.~1, 1--60. \MR{2929082}

\bibitem[Ber93]{Berkovitch}
Vladimir~G. Berkovich, \emph{\'{E}tale cohomology for non-{A}rchimedean
  analytic spaces}, Inst. Hautes \'{E}tudes Sci. Publ. Math. (1993), no.~78,
  5--161 (1994). \MR{1259429}

\bibitem[BK]{bk}
Tobias Berger and Kris Klosin, \emph{Deformations of {S}aito-{K}urokawa type
  and the paramodular conjecture}, to appear in Amer. J. Math.,
  \url{https://arxiv.org/pdf/1710.10228.pdf}.

\bibitem[CGJ95]{CGJ}
Robert~F. Coleman, Fernando~Q. Gouv\^{e}a, and Naomi Jochnowitz, \emph{{$E_2$},
  {$\Theta$}, and overconvergence}, Internat. Math. Res. Notices (1995), no.~1,
  23--41. \MR{1317641}

\bibitem[Che04]{Chenevier}
Ga\"{e}tan Chenevier, \emph{Familles {$p$}-adiques de formes automorphes pour
  {${\rm GL}_n$}}, J. Reine Angew. Math. \textbf{570} (2004), 143--217.
  \MR{2075765}

\bibitem[CM98]{coleman-mazur}
R.~Coleman and B.~Mazur, \emph{The eigencurve}, Galois representations in
  arithmetic algebraic geometry ({D}urham, 1996), London Math. Soc. Lecture
  Note Ser., vol. 254, Cambridge Univ. Press, Cambridge, 1998, pp.~1--113.
  \MR{1696469}

\bibitem[Con99]{conrad}
Brian Conrad, \emph{Irreducible components of rigid spaces}, Ann. Inst. Fourier
  (Grenoble) \textbf{49} (1999), no.~2, 473--541. \MR{1697371}

\bibitem[FC90]{CF}
Gerd Faltings and Ching-Li Chai, \emph{Degeneration of abelian varieties},
  Ergebnisse der Mathematik und ihrer Grenzgebiete (3) [Results in Mathematics
  and Related Areas (3)], vol.~22, Springer-Verlag, Berlin, 1990, With an
  appendix by David Mumford. \MR{1083353}

\bibitem[GT19]{G-T}
Toby Gee and Olivier Ta\"{\i}bi, \emph{Arthur's multiplicity formula for {${\bf
  GSp}_4$} and restriction to {${\bf Sp}_4$}}, J. \'{E}c. polytech. Math.
  \textbf{6} (2019), 469--535.

\bibitem[Her19]{He}
Valentin Hernandez, \emph{Families of {P}icard modular forms and an application
  to the {B}loch--{K}ato conjecture}, Compos. Math. \textbf{155} (2019), no.~7,
  1327--1401. \MR{3975499}

\bibitem[Hid86]{hida85}
Haruzo Hida, \emph{Galois representations into {${\rm GL}_2({\bf Z}_p[[X]])$}
  attached to ordinary cusp forms}, Invent. Math. \textbf{85} (1986), no.~3,
  545--613. \MR{848685}

\bibitem[Hid02]{Hida1}
\bysame, \emph{Control theorems of coherent sheaves on {S}himura varieties of
  {PEL} type}, J. Inst. Math. Jussieu \textbf{1} (2002), no.~1, 1--76.
  \MR{1954939}

\bibitem[Kat04]{Kato}
Kazuya Kato, \emph{{$p$}-adic {H}odge theory and values of zeta functions of
  modular forms}, Ast\'{e}risque (2004), no.~295, ix, 117--290, Cohomologies
  $p$-adiques et applications arithm\'{e}tiques. III. \MR{2104361}

\bibitem[Kis03]{kisin}
Mark Kisin, \emph{Overconvergent modular forms and the {F}ontaine-{M}azur
  conjecture}, Invent. Math. \textbf{153} (2003), no.~2, 373--454. \MR{1992017}

\bibitem[Kis04]{Kisin04}
\bysame, \emph{Geometric deformations of modular {G}alois representations},
  Invent. Math. \textbf{157} (2004), no.~2, 275--328.

\bibitem[L\"74]{Lu}
Werner L\"{u}tkebohmert, \emph{Der {S}atz von {R}emmert-{S}tein in der
  nichtarchimedischen {F}unktionentheorie}, Math. Z. \textbf{139} (1974),
  69--84. \MR{352527}

\bibitem[Lau05]{Laumon}
G\'{e}rard Laumon, \emph{Fonctions z\^{e}tas des vari\'{e}t\'{e}s de {S}iegel
  de dimension trois}, Ast\'{e}risque (2005), no.~302, 1--66, Formes
  automorphes. II. Le cas du groupe ${\rm{G}}Sp(4)$. \MR{2234859}

\bibitem[Maj15]{Maj}
Dipramit Majumdar, \emph{Geometry of the eigencurve at critical {E}isenstein
  series of weight 2}, J. Th\'{e}or. Nombres Bordeaux \textbf{27} (2015),
  no.~1, 183--197. \MR{3346969}

\bibitem[Maz97]{Ma95}
Barry Mazur, \emph{An introduction to the deformation theory of {G}alois
  representations}, Modular forms and {F}ermat's last theorem ({B}oston, {MA},
  1995), Springer, New York, 1997, pp.~243--311. \MR{1638481}

\bibitem[Miy89]{Miyake}
Toshitsune Miyake, \emph{Modular forms}, Springer-Verlag, Berlin, 1989,
  Translated from the Japanese by Yoshitaka Maeda. \MR{1021004}

\bibitem[Mok14]{Mok}
Chung~Pang Mok, \emph{Galois representations attached to automorphic forms on
  {${\rm GL}_2$} over {CM} fields}, Compos. Math. \textbf{150} (2014), no.~4,
  523--567. \MR{3200667}

\bibitem[MY14]{MM}
Michitaka Miyauchi and Takuya Yamauchi, \emph{An explicit computation of
  {$p$}-stabilized vectors}, J. Th\'{e}or. Nombres Bordeaux \textbf{26} (2014),
  no.~2, 531--558. \MR{3320491}

\bibitem[Nek93]{Nek93}
Jan Nekov\'{a}\v{r}, \emph{On {$p$}-adic height pairings}, S\'{e}minaire de
  {T}h\'{e}orie des {N}ombres, {P}aris, 1990--91, Progr. Math., vol. 108,
  Birkh\"{a}user Boston, Boston, MA, 1993, pp.~127--202. \MR{1263527}

\bibitem[Nek06]{Nek}
\bysame, \emph{Selmer complexes}, Ast\'{e}risque (2006), no.~310, viii+559.
  \MR{2333680}

\bibitem[Och01]{Ochiai}
Tadashi Ochiai, \emph{Control theorem for {G}reenberg's {S}elmer groups of
  {G}alois deformations}, J. Number Theory \textbf{88} (2001), no.~1, 59--85.
  \MR{1825991}

\bibitem[Pil11]{Pilloni}
Vincent Pilloni, \emph{Prolongement analytique sur les vari\'{e}t\'{e}s de
  {S}iegel}, Duke Math. J. \textbf{157} (2011), no.~1, 167--222. \MR{2783930}

\bibitem[PP94]{PP}
\emph{P\'{e}riodes {$p$}-adiques}, Soci\'{e}t\'{e} Math\'{e}matique de France,
  Paris, 1994, Papers from the seminar held in Bures-sur-Yvette, 1988,
  Ast\'{e}risque No. 223 (1994) (1994), pp.~1--397. \MR{1293969}

\bibitem[Sch07]{Schmidt07}
Ralf Schmidt, \emph{On classical {S}aito-{K}urokawa liftings}, J. Reine Angew.
  Math. \textbf{604} (2007), 211--236. \MR{2320318}

\bibitem[Sch18a]{Schmidt}
\bysame, \emph{Packet structure and paramodular forms}, Trans. Amer. Math. Soc.
  \textbf{370} (2018), no.~5, 3085--3112. \MR{3766842}

\bibitem[Sch18b]{Schmidt18}
\bysame, \emph{Paramodular forms in {CAP} representations of {${\mathrm{
  GSp}}(4)$}}, preprint (2018), \url{http://www.math.unt.edu/~schmidt/}.

\bibitem[Sor10]{Sch}
Claus~M. Sorensen, \emph{Galois representations attached to {H}ilbert-{S}iegel
  modular forms}, Doc. Math. \textbf{15} (2010), 623--670. \MR{2735984}

\bibitem[SS13]{A-S}
Abhishek Saha and Ralf Schmidt, \emph{Yoshida lifts and simultaneous
  non-vanishing of dihedral twists of modular {$L$}-functions}, J. Lond. Math.
  Soc. (2) \textbf{88} (2013), no.~1, 251--270. \MR{3092267}

\bibitem[SU06]{urban}
Christopher Skinner and Eric Urban, \emph{Sur les d\'{e}formations
  {$p$}-adiques de certaines repr\'{e}sentations automorphes}, J. Inst. Math.
  Jussieu \textbf{5} (2006), no.~4, 629--698. \MR{2261226}

\bibitem[SU14]{SkinnerUrban14}
\bysame, \emph{The {I}wasawa main conjectures for {$\rm GL_2$}}, Invent. Math.
  \textbf{195} (2014), no.~1, 1--277. \MR{3148103}

\bibitem[Tay93]{Taylor2}
Richard Taylor, \emph{On the {$l$}-adic cohomology of {S}iegel threefolds},
  Invent. Math. \textbf{114} (1993), no.~2, 289--310. \MR{1240640}

\bibitem[TU99]{U-T}
Jacques Tilouine and Eric Urban, \emph{Several-variable {$p$}-adic families of
  {S}iegel-{H}ilbert cusp eigensystems and their {G}alois representations},
  Ann. Sci. \'{E}cole Norm. Sup. (4) \textbf{32} (1999), no.~4, 499--574.
  \MR{1693583}

\bibitem[Urb01]{urban1}
Eric Urban, \emph{Selmer groups and the {E}isenstein-{K}lingen ideal}, Duke
  Math. J. \textbf{106} (2001), no.~3, 485--525. \MR{1813234}

\bibitem[Urb05]{U1}
\bysame, \emph{Sur les repr\'{e}sentations {$p$}-adiques associ\'{e}es aux
  repr\'{e}sentations cuspidales de {${\rm GSp}_{4/{\Bbb Q}}$}}, Ast\'{e}risque
  (2005), no.~302, 151--176, Formes automorphes. II. Le cas du groupe
  ${\rm{G}}Sp(4)$. \MR{2234861}

\bibitem[Wei05]{W}
Rainer Weissauer, \emph{Four dimensional {G}alois representations},
  Ast\'{e}risque (2005), no.~302, 67--150, Formes automorphes. II. Le cas du
  groupe ${\rm{G}}Sp(4)$. \MR{2234860}

\bibitem[Wes04]{Weston}
Tom Weston, \emph{Geometric {E}uler systems for locally isotropic motives},
  Compos. Math. \textbf{140} (2004), no.~2, 317--332.

\end{thebibliography}
\end{document}